\documentclass[a4paper]{amsart} % for other books use amsbook!
\usepackage{fullpage}   % to save paper, but it cuts off the todonotes
\usepackage{amsmath, amsfonts, amsthm}  % in case a non-AMS document class is used
\usepackage{amssymb, latexsym}  % need amssymb for \restriction and \smallfrown
\usepackage{cmap}
\usepackage{refcount}
\usepackage{hyperref}

\newcommand{\pred}[1]{#1_\downarrow}
\newcommand{\cone}[1]{{#1}^\uparrow}

\newcommand{\s}{\subseteq}
\newcommand \sq{\sqsubseteq}

\newcommand{\sqleft}[1]{\mathrel{_{#1}{\sqsubseteq}}}
\newcommand{\sqx}{\sqleft{\chi}}
\newcommand{\sql}{\sqleft{\lambda}}
\newcommand{\sqn}{\sqleft{\nu}}

\newcommand\defaultaction{\textsf{extend}}
\newcommand\sealantichain{\textsf{anti}}
\newcommand\free{\textsf{free}}

\renewcommand{\restriction}{\mathbin\upharpoonright}    % by default in amssymb it's mathrel

\newcommand\sd{\framebox[3mm][l]{$\diamondsuit$}\hspace{0.5mm}{}}
\DeclareMathOperator{\acc}{acc}
\DeclareMathOperator{\cf}{cf}
\DeclareMathOperator{\dom}{dom}
\DeclareMathOperator{\cof}{cof}
\newcommand\bd{\textup{bd}}
\newcommand\fin{\textup{fin}}
\DeclareMathOperator{\height}{ht}
\DeclareMathOperator{\nacc}{nacc}
\DeclareMathOperator{\otp}{otp}
\DeclareMathOperator{\range}{Im}
\DeclareMathOperator{\suc}{succ}
\newcommand\ch{\textup{CH}}
\DeclareMathOperator{\p}{P}

\newcommand\axiomfont[1]{\textsf{\textup{#1}}}
\newcommand\zfc{\axiomfont{ZFC}}
\newcommand\gch{\axiomfont{GCH}}
\newcommand\mm{\axiomfont{MM}}
\newcommand\pfa{\axiomfont{PFA}}

\newtheorem{theorem}{Theorem}[section]
\newtheorem{lemma}[theorem]{Lemma}
\newtheorem{corollary}[theorem]{Corollary}
\newtheorem{fact}[theorem]{Fact}

\newtheorem*{claim-template}{Coherence Claim Template}
\newtheorem*{extension-lemma}{Extension Lemma}

\newtheorem{claim}{Claim}[theorem]
\newtheorem{subclaim}{Subclaim}[claim]

\theoremstyle{definition}
\newtheorem{definition}[theorem]{Definition}

\theoremstyle{remark}

\newtheorem*{Q}{Question}

\title{Reduced powers of Souslin trees}
\author{Ari Meir Brodsky}
\author{Assaf Rinot}
\address{Department of Mathematics\\
    Bar-Ilan University\\
    Ramat-Gan 5290002\\
    Israel}
\urladdr{http://www.assafrinot.com}



\begin{document}
\begin{abstract}
We study the relationship between a $\kappa$-Souslin tree $T$ and its reduced powers $T^\theta/\mathcal U$.

Previous works addressed this problem from the viewpoint of a single power $\theta$, whereas here, tools are developed for controlling different powers simultaneously.
As a sample corollary, we obtain the consistency of an $\aleph_6$-Souslin tree $T$ and a sequence of uniform ultrafilters $\langle \mathcal U_n\mid n<6\rangle$ such that $ T^{\aleph_n}/\mathcal U_n$ is $\aleph_6$-Aronszajn iff $n<6$ is not a prime number.
\end{abstract}
\maketitle

\section{Introduction}

A \emph{tree} is a partially ordered set $(T,<_T)$ with the property that for every $x\in T$,
the downward cone $x_\downarrow=\{ y\in T\mid y<_T x\}$ is well-ordered by $<_T$. The \emph{height} of $x\in T$, denoted $\height(x)$,
is the order-type of $(x_\downarrow,<_T)$. Then, the $\alpha^{\text{th}}$ level of $(T,<_T)$ is the set $T_\alpha=\{x\in T\mid \height(x)=\alpha\}$.
We also write $T \restriction X = \{t \in T \mid \height(t) \in X \}$.
A tree  $(T,<_T)$ is said to be \emph{$\chi$-complete} if any $<_T$-increasing sequence of elements from $T$,
and of length ${<\chi}$, has an upper bound in $T$. On the other extreme,
the tree  $(T,<_T)$ is said to be \emph{slim} if
$\left|T_\alpha\right| \leq \max \{ \left|\alpha\right|, \aleph_0 \}$ for every ordinal $\alpha$.
Throughout, let $\kappa$ denote a regular uncountable cardinal.
A tree $(T,<_T)$ is a \emph{$\kappa$-tree} whenever $\{ \alpha\mid T_\alpha\neq\emptyset\}=\kappa$, and $|T_\alpha|<\kappa$ for all $\alpha<\kappa$.

A subset $B\s T$ is a \emph{cofinal branch} if $(B,<_T)$ is linearly ordered and  $\{ \height(t)\mid t\in B\}=\{\height(t)\mid t\in T\}$.
A \emph{$\kappa$-Aronszajn tree} is a $\kappa$-tree with no cofinal branches.
On the other extreme, there is the concept of a \emph{$\kappa$-Kurepa tree},
which is a $\kappa$-tree with at least $\kappa^+$ many cofinal branches.

A $\kappa$-Aronszajn tree $(T,<_T)$ is a \emph{$\kappa$-Souslin tree} if it has no antichains of size $\kappa$. Equivalently,
if for every antichain $A\s T$, the set $\{ \height(x)\mid x\in A\}$ has size $<\kappa$.
A $\lambda^+$-tree $(T,<_T)$ is said to be \emph{almost Souslin}
if for every antichain $A\s T$, the set $\{ \height(x)\mid x\in A\}\cap E^{\lambda^+}_{\cf(\lambda)}$ is nonstationary.\footnote{Denote $E^\kappa_{\chi} = \{\alpha <\kappa \mid \cf(\alpha) = \chi \}$,
and similarly define $E^\kappa_{<\chi}$, $E^\kappa_{\geq\chi}$, and $E^\kappa_{>\chi}$.}

Given a $\kappa$-Souslin tree $(T,<_T)$ and a set $I$, let $T^I=\{ f:I\rightarrow T\mid \height\circ f\text{ is constant}\}$
denote the collection of all \emph{level sequences} indexed by $I$.
For any uniform ultrafilter $\mathcal U$ over $I$, we then consider the \emph{reduced $I$-power  tree} $\check T=T^I/\mathcal U$, as follows.\footnote{Recall that $\mathcal U$ is said to be \emph{uniform}
if $|X|=|Y|$ for all $X,Y\in\mathcal U$.}
The elements of $\check T$ are equivalence classes $[f]_\mathcal U$, where $f=_{\mathcal U}g$ iff $\{ i\in I\mid f(i)=g(i)\}\in\mathcal U$.
The ordering $<_{\check T}$ of $\check T$ is defined by letting $[f]_{\mathcal U}<_{\check T}[g]_{\mathcal U}$ iff $\{ i\in I\mid f(i)<_T g(i)\}\in\mathcal U$.

Suppose now that $I=\theta$ is an infinite cardinal such that $\lambda^{\theta}<\kappa$ for all $\lambda<\kappa$.
Then, $T^\theta/\mathcal U$  is again a   $\kappa$-tree.
How do the original $\kappa$-Souslin tree and its reduced $\theta$-power compare?

By an argument essentially due to Kurepa \cite{MR0049973}, the reduced $\theta$-power of a $\kappa$-Souslin tree is never $\kappa$-Souslin. Can it at least remain almost Souslin? Aronszajn?

Devlin gave a consistent example of an $\aleph_2$-Souslin tree whose reduced $\omega$-power is $\aleph_2$-Aronszajn \cite{MR732661},
and another consistent example of an $\aleph_2$-Souslin tree whose reduced $\omega$-power is $\aleph_2$-Kurepa \cite{MR645905}.
In this paper, we give an example of an $\aleph_2$-Souslin tree whose reduced $\omega$-power is almost Souslin,
and another example whose reduced $\omega$-power is not almost Souslin.
In fact, more is true,  and is better seen at the level of $\aleph_3$:
\begin{theorem}\label{thm11} Assume $V=L$.

Then there exist trees $T_0,T_1,T_2,T_3$, and uniform ultrafilters $\mathcal U_0$ over $\omega$, $\mathcal U_1$ over $\omega_1$,
such that:
$$\begin{array}{c|c|c|c}
&T& T^\omega/\mathcal U_0&T^{\omega_1}/\mathcal U_1  \\ \hline
T_0&\aleph_3\text{-Souslin}&\aleph_3\text{-Aronszajn + almost Souslin}&\aleph_3\text{-Aronszajn + almost Souslin}\\
T_1&\aleph_3\text{-Souslin}&\aleph_3\text{-Kurepa + $\neg$ almost Souslin}&\aleph_3\text{-Kurepa + $\neg$ almost Souslin}\\
T_2&\aleph_3\text{-Souslin}&\aleph_3\text{-Aronszajn + almost Souslin}&\aleph_3\text{-Kurepa + $\neg$ almost Souslin}\\
T_3&\aleph_3\text{-Souslin}&\neg\aleph_3\text{-Aronszajn}&\aleph_3\text{-Aronszajn + almost Souslin}
\end{array}$$
\end{theorem}

Let us introduce the concepts and tools that will be used in proving the theorems of this paper.

\begin{definition} Suppose that $X\s{}^{<\kappa}\kappa$ is a downward-closed family such that $(X,\subset)$ is a $\kappa$-tree,
and  $\mathcal F$ is a collection of sets.

An \emph{$(\mathcal F,X)$-ascent path} through a $\kappa$-tree $(T,<_T)$ is
a sequence $\vec f =\langle f_x\mid x\in X\rangle$
such that:
\begin{enumerate}
\item $f_x : \bigcup \mathcal F\to T_{\dom(x)}$ is a function for each $x \in X$;
\item $\{ i\in\bigcup \mathcal F \mid f_x (i) <_T f_y (i) \}\in \mathcal F$ for all $x\subset y$ from $X$;
\item $\{ i\in\bigcup \mathcal F \mid f_x (i) \neq f_y (i) \}\in \mathcal F$
for any two distinct elements $x,y$ of $X_\alpha$ for some $\alpha<\kappa$.
\end{enumerate}

If $(X,\subset)$ is isomorphic to $(\kappa,\in)$ (e.g., $X=\bigcup_{\alpha<\kappa}{}^\alpha1$), then $\vec f\ $ is simply said to be an \emph{$\mathcal F$-ascent path}.
\end{definition}

Denote
$\mathcal F^{\bd}_\theta=\{ Z\s\theta\mid \sup(\theta\setminus Z)<\theta\}$,
and $\mathcal F_\theta=\mathcal P(\theta)\setminus\{\emptyset\}$.
It is easy to see that if  $(T,<_T)$ admits an $(\mathcal F^{\bd}_\theta,X)$-ascent path,
then the reduced $\theta$-power tree (by any uniform ultrafilter  over $\theta$)
contains a copy of the tree $(X,\subset)$.
Likewise, if $(T,<_T)$ admits no $\mathcal F_\theta$-ascent path, then
the reduced $\theta$-power (using any uniform ultrafilter) is Aronszajn.

We remark that the concept of an ascent path makes sense also in the absence of $\gch$, that is, regardless of the associated power trees.
It was discovered by Laver while working on the problem of obtaining a model in which all $\aleph_2$-Aronszajn trees are special \cite{MR603771}.\footnote{Recall that a $\lambda^+$-tree is said to be \emph{special} if it may be covered by $\lambda$-many antichains.
Note that a special $\lambda^+$-tree is never almost Souslin.}
By \cite{MR964870}, if $(T,<_T)$ is a special $\lambda^+$-tree
that admits an $\mathcal F^{\bd}_\theta$-ascent path, then $\cf(\lambda)=\cf(\theta)$.\footnote{The statement of Lemma 3 from \cite{MR964870} is slightly weaker,
but its proof establishes the statement given here. For $\lambda$ regular, an even stronger statement holds true, and we refer the reader to Proposition 2.3 of \cite{MR2965421}.}
This provides an approach to constructions of $\lambda^+$-trees that are impossible to specialize without changing cofinalities.
In this paper, among other things, we shall construct $\kappa$-Souslin trees with $\mathcal F$-ascent path,
where $\mathcal F$ is $\mathcal F^{\fin}_\theta=\{ Z\s\theta\mid |\theta\setminus Z|<\aleph_0\}$.
Such trees are even harder to specialize, since $\mathcal F^{\fin}_\theta$ projects to $\mathcal F^{\bd}_\mu$ for all infinite cardinals $\mu\le\theta$,
and so a model in which such a  $\lambda^+$-tree becomes special would have to satisfy $\cf(\lambda)=\cf(\mu)$ for all infinite cardinals $\mu\le\theta$.
On a dual front, L\"ucke \cite{lucke} proved that assuming $\lambda^{<\lambda}=\lambda$, any $\lambda^+$-tree with the property that for all infinite $\theta<\lambda$, the tree admits no $\mathcal F_\theta$-ascent path,
can be made special via a cofinality-preserving notion of forcing. That is, the tree is \emph{specializable}. A systematic study of specializable, nearly special, and such, Souslin trees is carried in a forthcoming paper.

\medskip

Thus, we have an approach  for introducing objects into the reduced $\theta$-power tree. Recalling Theorem \ref{thm11},
we shall also need an approach for preventing objects from appearing there.
\begin{definition}
A $\kappa$-Souslin tree $(T,<_T)$ is said to be \emph{$\chi$-free} if for every nonzero $\tau<\chi$,
any $\beta < \kappa$, and any sequence of distinct nodes
$\langle w_i \mid i < \tau \rangle \in {}^\tau T_\beta$, the derived tree
$\bigotimes_{i < \tau}  \cone{w_i}$
is again a  $\kappa$-Souslin tree.
\end{definition}

Here, the derived tree $\bigotimes_{i < \tau}  \cone{w_i}$ stands for the tree $(\hat T,<_{\hat T})$, as follows:
\begin{itemize}
\item $\hat T=\{ f\in T^\tau \mid \forall i<\tau(f(i)\text{ is $<_T$-compatible with } w_i)\}$;
\item $f<_{\hat T}\ g$ iff $f(i)<_{T} g(i)$ for all $i<\tau$.
\end{itemize}

\medskip

An $\aleph_0$-free Souslin tree is simply called \emph{free}.
By Lemma \ref{lemma28} below, if $(T,<_T)$ is a $\theta^+$-free Souslin tree, then for an appropriate uniform ultrafilter $\mathcal U$ over $\theta$,
$T^\theta/\mathcal U$ is Aronszajn and almost Souslin. This gives a promising approach for pulling the other side of the blanket with respect to the reduced $\theta$-power,
but raises the very problem of constructing free Souslin trees.

Jensen constructed a free $\aleph_1$-Souslin tree from $\diamondsuit(\omega_1)$ (see \cite[Theorem V.1]{MR384542}). The argument generalizes to show that whenever $\lambda^{<\lambda}=\lambda$,
$\diamondsuit(E^{\lambda^+}_\lambda)$ entails a $\lambda$-complete $\lambda$-free $\lambda^+$-Souslin tree.
The $\lambda$-completeness of this tree is not a bonus but a necessity,\footnote{The level $\alpha\in E^{\lambda^+}_\lambda$ of Jensen's tree is derived from a generic over some poset that lives in a model of size $\lambda$.
The freeness comes from genericity. The very existence of a generic comes from the $\lambda$-completeness of the tree.}
and indeed the consistency of existence of a free $\lambda^+$-Souslin tree for $\lambda$ singular was unknown.

Motivated by the above, in this paper, an alternative construction of a free $\kappa$-Souslin tree is given,
covering the case that $\kappa$ is a successor of a singular cardinal
(as well as inaccessible).  To exemplify:
\begin{theorem} Suppose that $\lambda$ is a singular cardinal. Let $\chi$ denote the least cardinal to satisfy $\lambda^\chi>\lambda$.

If $\square_\lambda+\ch_\lambda$ holds, then there exists a (slim!) $\chi$-free $\lambda^+$-Souslin tree.\footnote{Here, $\ch_\lambda$ stands for the assertion that $2^\lambda=\lambda^+$. The choice of $\chi$ is sharp,
as by Lemma \ref{lemma85} below, the existence of a $\chi$-free $\lambda^+$-Souslin tree entails that $\lambda^{<\chi}=\lambda$.}
\end{theorem}

A reader who is familiar with previous constructions of $\kappa$-Souslin trees with an ascent path (due to Baumgartner, Cummings, Devlin, and Laver)
is probably wondering how is it possible to construct the $\kappa$-tree without taking into consideration whether $\kappa$
is inaccessible, or successor of regular, or a successor of singular of countable cofinality, or of uncountable cofinality.
The answer is that all constructions in this paper will go through the parameterized proxy principle $\p(\kappa, \mu, \mathcal R, \theta, \mathcal S,  \nu,\sigma,\mathcal E)$ from \cite{axioms},\cite{rinot23}.
This allows a uniform construction that is indifferent of the identity of $\kappa$,
and was one of the motivations for the introduction of this principle.
For the purpose of this paper, we shall only be concerned with the special case $\mu=\nu=2$, $\sigma=\omega$, and $\mathcal E=\mathcal P(\kappa)^2$,
and hence we choose to define only this simpler instance, which we denote by $\p_{14}(\kappa,\mathcal R,\theta,\mathcal S)$.\footnote{%Why 14?
Since $\p_{14}(\kappa,\cdot,\cdot,\cdot)\equiv \p(\kappa,2,\cdot,\cdot,\cdot,2,\omega,(\mathcal P(\kappa))^2)$,
and $14$ is the decimal interpretation of the flip of the binary string $0111000$.}
For a complete account, the reader is referred to \cite{axioms} and \cite{rinot23}.

Before giving the definition of the proxy principle, let us agree to denote for a set of ordinals $D$,
$\acc(D) = \{\alpha \in D \mid \sup (D\cap\alpha) = \alpha>0 \}$,
$\nacc(D)=D\setminus\acc(D)$,
and  $\suc_\omega(D)=\{\delta\in D\mid 0<\otp(D\cap\delta)<\omega\}$.

\begin{definition}[Proxy principle]
Suppose that:
\begin{itemize}
\item
$\kappa$ is a regular uncountable cardinal;
\item
$\mathcal R$ is a binary relation over $[\kappa]^{<\kappa}$;
\item
$\theta$ is a cardinal such that $1 \leq \theta \leq \kappa$; and
\item
$\mathcal S$ is a nonempty collection of stationary subsets of $\kappa$.
\end{itemize}
The principle
$\p_{14}^-(\kappa, \mathcal R, \theta, \mathcal S)$
asserts the existence of a sequence
$\langle C_\alpha \mid \alpha < \kappa \rangle$
such that:
\begin{itemize}
\item for every limit ordinal $\alpha < \kappa$, $C_\alpha$ is a
club subset of $\alpha$;
\item for every ordinal $\alpha < \kappa$, if $\bar\alpha\in\acc(C_\alpha)$, then $C_{\bar\alpha} \mathrel{\mathcal R} C_\alpha$;
\item for every sequence $\langle A_i \mid i < \theta \rangle$ of cofinal subsets of $\kappa$,
and every $S \in \mathcal S$,
there exist stationarily many $\alpha \in S$ such that for every  $i < \min \{\alpha, \theta\}$, we have
\[
\sup \{\beta \in C_\alpha \mid \suc_\omega (C_\alpha \setminus \beta) \subseteq A_i \} = \alpha.
\]
\end{itemize}
\end{definition}

As for the relation $\mathcal R$, in this paper, we shall only be concerned with the relations $\sqsubseteq$, $\sqn$, $\sqsubseteq_\nu$, where:
\begin{itemize}
\item
$D \sqsubseteq C$ iff
there exists some ordinal $\beta$ such that $D = C \cap \beta$, that is, $C$ \emph{end-extends} $D$;
\item $D \sqn C$ iff (($D \sqsubseteq C$) or ($\cf(\sup(D))<\nu$));
\item $D \sqsubseteq_\nu C$ iff (($D \sqsubseteq C$) or ($\otp(C)<\nu$ and $\nacc(C)$ consists only of successor ordinals)).
\end{itemize}

It is easy to see that  $\p_{14}^-(\kappa,{\sq},\theta,\mathcal S)\Rightarrow \p_{14}^-(\kappa,{\sqn},\theta,\mathcal S)\Rightarrow \p_{14}^-(\kappa,{\sql},\theta,\mathcal S)$ for all $\nu<\lambda<\kappa$.
Likewise, $\p_{14}^-(\kappa,{\sq},\theta,\mathcal S)\Rightarrow \p_{14}^-(\kappa,{\sq_\nu},\theta,\mathcal S)\Rightarrow \p_{14}^-(\kappa,{\sq_\lambda},\theta,\mathcal S)$ for all $\nu<\lambda<\kappa$.

\begin{definition}\label{def16} $\p_{14}(\kappa, \mathcal R, \theta, \mathcal S)$
asserts that both $\p_{14}^-(\kappa,  \mathcal R, \theta, \mathcal S)$ and $\diamondsuit(\kappa)$ hold.
\end{definition}

The consistency of the preceding principle is extensively studied in \cite{axioms},\cite{rinot23}. We mention two extremes from \cite{axioms}.
If $V=L$, then $\p_{14}(\kappa,{\sq},\kappa,\{E^\kappa_{\ge\chi}\mid \chi<\kappa\ \&\ \forall\lambda<\kappa(\lambda^{<\chi}<\kappa)\})$ holds for every regular uncountable cardinal $\kappa$ that is not weakly compact.
Assuming the existence of a supercompact cardinal, it is consistent that for some infinite cardinals $\nu<\lambda$,
$\p_{14}(\lambda^+,{\sq_\nu},\lambda^+,\{\lambda^+\})$  holds, while $\square^*_\lambda$ fails,
and the same is true replacing $\sq_\nu$ with $\sql$.

\medskip

So far, we have described a strategy for constructing  $\kappa$-Souslin trees whose $\theta_0$-power contains a prescribed tree,
and another strategy for constructing $\kappa$-Souslin trees whose $\theta_1$-power omits prescribed objects.
Could these strategies live side by side? The answer is clearly negative if $\theta_0=\theta_1$.
But even if $\theta_0\neq\theta_1$, there are further obstructions. These obstructions lead us to introducing the following concept.

\begin{definition}\label{respecting}
We say that a $\kappa$-tree $X\s{}^{<\kappa}\kappa$
is \emph{$\p_{14}^-(\kappa,  \mathcal R, \theta, \mathcal S)$-respecting}
if there exists a subset $\S\s\kappa$ and a sequence of mappings
 $\langle \mathbf b^\alpha:(X\restriction C_\alpha)\rightarrow {}^\alpha\kappa\cup\{\emptyset\}\mid\alpha<\kappa\rangle$ such that:
\begin{enumerate}
\item\label{respectingonto} $X_\alpha\s \range({\mathbf b}^\alpha)$ for every $\alpha\in\S$;
\item $\langle C_\alpha\mid\alpha<\kappa\rangle$ witnesses $\p_{14}^-(\kappa, \mathcal R, \theta, \{S\cap\S\mid S\in \mathcal S\})$;
\item \label{respectcohere} if $x\in X\restriction C_{\bar\alpha}$ and $C_\alpha\cap[\height(x),{\bar\alpha})=C_{\bar\alpha}\cap[\height(x),{\bar\alpha})$, then ${\mathbf b}^{\bar\alpha}(x)={\mathbf b}^\alpha(x)\restriction{\bar\alpha}$.
\end{enumerate}
\end{definition}

It is not hard to show that $\p_{14}(\kappa,{\sq},1,\{\kappa\})$ entails a $\kappa$-Souslin tree that is  $\p^-_{14}(\kappa,{\sq},1,\{\kappa\})$-respecting,
and with witnessing mappings ${\mathbf b}^\alpha$ having the property that ${\mathbf b}^\alpha(x)$ is always compatible with $x$.
Arguably, the proof was already given in \cite[Theorem IV.2.4]{MR750828}.

Clearly, if $(X,<_X)$ is isomorphic to $(\kappa,\in)$, then it is $\p^-_{14}(\kappa,\dots)$-respecting (provided that $\p^-_{14}(\kappa,\dots)$ holds).
What is unclear is whether trees that are not built in a bottom-up fashion can be $\p^-_{14}$-respecting.
In \cite{rinotschindler}, Rinot and Schindler gave consistent examples of $\p^-_{14}$-respecting trees whose natural description is indeed  top-down.
This includes Kurepa trees, and the special tree $\mathcal T(\rho_0)$ from \cite{MR908147}, which fully encodes the process of \emph{walking} down from one ordinal to another.

Therefore, we feel that one of the most interesting theorems of this paper is the following.
\begin{theorem} Suppose that $\theta<\kappa$ are regular infinite cardinals, and
\begin{itemize}
\item $X\s{}^{<\kappa}\kappa$ is a downward-closed $\kappa$-tree that is $\p_{14}^-(\kappa,  {\sqsubseteq}_\theta, \kappa, \{E^\kappa_{\ge\eta}\})$-respecting;
\item  $\eta$ is an infinite cardinal satisfying $\lambda^{<\eta} < \kappa$ for all $\lambda < \kappa$;
\item $\chi=\min\{\eta,\theta\}$;
\item $\diamondsuit(\kappa)$ holds.
\end{itemize}

Then there exists a $\chi$-free $\eta$-complete $\kappa$-Souslin tree that admits an $(\mathcal F^{\bd}_\theta,X)$-ascent path.
\end{theorem}

Using the above theorem, we for instance infer that assuming $V=L$, there exists an $\aleph_0$-free $\aleph_1$-complete  $\aleph_2$-Souslin tree
whose reduced $\omega$-power tree (by any uniform ultrafilter) is $\aleph_2$-Kurepa.\footnote{This is sharp, as the results of Section \ref{section2} entails that such a tree cannot be $\aleph_1$-free.}

Previously, Cummings \cite{MR1376756} gave a consistent construction of an $\aleph_1$-complete $\aleph_2$-Souslin tree with an $\mathcal F^{\bd}_{\aleph_0}$-ascent path.
Roughly speaking, the idea was to construct the levels of the tree $T_\alpha$ together with the portion of the ascent path $f_\alpha:\omega\rightarrow T_\alpha$
by recursion over $\alpha<\kappa$,\footnote{It is customary to identify an $\mathcal F$-ascent path with a sequence of the form $\langle f_\alpha\mid\alpha<\kappa\rangle$.}
making sure that any $t\in T_\alpha$ for $\alpha\in E^{\aleph_2}_{\aleph_1}$
extends some node from a guessed antichain in $\bigcup_{\beta<\alpha}T_\beta$. Now, as here we want the tree to be moreover $\aleph_0$-free,
we must construct the nodes of $T_\alpha$ in such a way that an analogous statement holds for \emph{a sequence of nodes} from $T_\alpha$, rather than just a single node.
This requires the construction of branches through $\bigcup_{\beta<\alpha}T_\beta$ to be aware of all other branches that are expected to be constructed and put inside $T_\alpha$,
including those that are there to insure the extensibility of the ascent path.

And there is one more obstruction. To make the corresponding reduced $\omega$-power an $\aleph_2$-Kurepa tree, instead of constructing a single ascent path,
we shall need to construct $\aleph_3$ many (sincerely distinct) $\mathcal F^{\bd}_{\aleph_0}$-ascent paths.
The latter makes the previous task (of anticipating future branches) even more challenging, and is resolved by assuming that the tree to-be-embedded is  $\p^-_{14}$-respecting.
The existence of a $\p^-_{14}(\aleph_2,{\sq},\aleph_2,\{E^{\aleph_2}_{\aleph_1}\})$-respecting $\aleph_2$-Kurepa tree was shown to follow from $V=L$
by  Rinot and Schindler in \cite{rinotschindler}.

\subsection{Organization of this paper} In Section~\ref{section2}, we provide the necessary notions and preliminaries concerning trees that are needed to understand the results of this paper.
In Section~\ref{section3}, we discuss the so-called \emph{microscopic approach to Souslin-tree constructions} that was developed in \cite{axioms},\cite{rinot23}
and serves as the framework for the tree constructions in this paper.
The reader is not expected to be familiar with~\cite{axioms},\cite{rinot23};
the relevant results from those papers will be stated where needed.

Sections~\ref{sectioncountablewidth} through~\ref{section:free} contain the heart of the paper:
the theorems involving the construction of Souslin trees with various properties.
This material is organized as a sequence of theorems of increasing complexity.
We start with Theorem \ref{ctbl-ascent-thm}, which is merely a rendition of a well-known construction from \cite{MR732661}.
Nevertheless, the proof of Theorem \ref{ctbl-ascent-thm}
will be given in great detail, as we shall return to components of this construction repeatedly throughout the whole paper.
Moreover, the breakdown is such that each new theorem builds on ideas already established in the previous theorems in this sequence,
and adds one or more new ideas in order to obtain a stronger result.
Therefore, when proceeding through any proof in these sections, the reader is expected to accept
the techniques used in the preceding proofs.

The next table exemplifies various types of $\kappa$-Souslin trees with an $(\mathcal F,X)$-ascent path constructed in this paper.
As one can see, the third parameter used in $\p_{14}(\kappa,\dots)$
increases in value from $1$ in Section~\ref{sectioncountablewidth}
(where we construct trees with countable ascent paths),
to an infinite cardinal $\theta < \kappa$ in Section~\ref{sectionwiderwidth}
(where we construct trees with ascent paths of width $\theta$),
to $\kappa$ in Section~\ref{section:free} (where we construct free trees).\footnote{The two-cardinal version of freeness is defined on page \pageref{deftwofree}.}

$$\begin{array}{c|l|c|c|c|c|c|c}
\text{Theorem} &\text{2nd}&\text{3rd}&\text{4th}&\text{Growth}&\mathcal F&(X,\subset)\text{ is}&\text{Freeness degree}\\ \hline
\ref{tree-ctbl-ascent}&\sqsubseteq&1&\{\kappa\}&\text{slim}&\mathcal F^{\fin}_{\aleph_0}&\text{slim}&\text{none}\\
\ref{tree-ctbl-ascent-complete-thm}&\sqsubseteq&1&\{E^{\kappa}_{\ge\chi}\}&\chi\text{-complete}&\mathcal F^{\fin}_{\aleph_0}&\text{arbitrary}&\text{none}\\
\ref{NEW-theta-ascent-thm}&\sqsubseteq&\theta&\{\kappa\}&\text{slim}&\mathcal F^{\fin}_{\theta}&\text{slim}&\text{none}\\
\ref{tree-theta-ascent-complete-thm}&\sqsubseteq_{\cf(\theta)}&\theta&\{E^\kappa_{\ge\chi}\}&\chi\text{-complete}&\mathcal F^{\bd}_{\theta}&\text{arbitrary}&\text{none}\\
\ref{theta-ascent-free-thm1}&\sqsubseteq&\kappa&\{E^{\kappa}_{\ge\chi}\}&\text{slim}&\mathcal F^{\fin}_{\theta}&\cong(\kappa,\in)&(\chi,\theta^+)\\
\ref{theta-ascent-free-thm2}&\sqsubseteq&\kappa&\{E^{\kappa}_{\ge\chi}\}&\text{slim}&\text{---}&\text{---}&\chi\\
\ref{free-with-u-ascent}&\sqsubseteq_{\cf(\theta)}&\kappa&\{E^{\kappa}_{\ge\chi}\}&\chi\text{-complete}&\mathcal F^{\bd}_\theta&\text{respecting}&\chi\\
\ref{thm65}&\sqsubseteq_{\cf(\theta)}&\kappa&\{E^{\kappa}_{\ge\chi}\}&\chi\text{-complete}&\mathcal F^{\bd}_{\theta}&\cong(\kappa,\in)&\cf(\theta)\text{ and }(\chi,\theta^+)
\end{array}$$
\smallskip

The paper is concluded with an Appendix section, where we inspect a natural process that produces,
for any given $\kappa$-tree,
a corresponding downward-closed subtree of ${}^{<\kappa} 2$
sharing many properties of the original one.
This allows us to focus on binary Hausdorff trees when proving various theorems,
without losing any generality.

\subsection{Sample corollaries} To give an idea of the flavor of consequences the results of this paper entail, we state here a few sample corollaries.
While the constructions of $\kappa$-Souslin in this paper apply to arbitrary regular uncountable cardinals $\kappa$,
in the following list, we shall concentrate only on $\lambda^+$-Souslin trees, mostly because we are unaware of a reasonable definition of an almost Souslin $\kappa$-tree for inaccessible $\kappa$ (let alone Mahlo).

All undefined terms may be found in Sections~\ref{section2} and~\ref{section3} below.

\begin{corollary} Suppose  $\square_\lambda+\ch_\lambda+\lambda^{<\lambda}=\lambda$ holds for a given regular uncountable cardinal $\lambda$.

Then after forcing to add a single Cohen subset to $\lambda$,
there exists a $\lambda$-complete $\lambda^+$-Souslin tree with an $\mathcal F^{\fin}_{\lambda}$-ascent path.
\end{corollary}
\begin{proof} By \cite{axioms},
after forcing to add a single Cohen subset  to an uncountable cardinal $\lambda=\lambda^{<\lambda}$ over a model of $\square_\lambda+\ch_\lambda$,
$\p_{14}(\lambda^+,{\sq},\lambda^+,\{E^{\lambda^+}_{\lambda}\})$ holds.
The conclusion now follows from Corollary \ref{cor52} below.
\end{proof}
By \cite{MR0314621}, every $\aleph_1$-Aronszajn tree is specializable.
The next example is of a $\lambda^+$-Souslin tree that cannot be specialized without reducing it to the scenario of \cite{MR0314621}.

\begin{corollary} Suppose $\square_\lambda+\ch_\lambda$ holds for a given singular cardinal $\lambda$ of countable cofinality.

Then there exists a $\lambda^+$-Souslin tree $(T,<_T)$ satisfying the following.
If $W$ is a ZFC extension of the universe  in which $(T,<_T)$ is a special $|\lambda|^+$-tree, then $W\models|\lambda|=\aleph_0$.
\end{corollary}
\begin{proof} By \cite{axioms}, for every singular cardinal $\lambda$, $\square_\lambda+\ch_\lambda$ entails $\p_{14}(\lambda^+,{\sq},\lambda,\{\lambda^+\})$.
Then, by Corollary \ref{cor52} below, there exists a $\lambda^+$-Souslin tree $(T,<_T)$ with an $\mathcal F^{\fin}_\lambda$-ascent path.

Towards a contradiction, suppose that $W$ is a ZFC extension of our universe $V$, satisfying:
\begin{itemize}
\item[(i)]  $W\models (T,<_T)\text{ is a special }|\lambda|^+\text{-tree}$;
\item[(ii)] $W\models|\lambda|>\aleph_0$.
\end{itemize}
By (i), $\lambda^+$ was not collapsed.
By (ii), pick a cardinal $\mu\le\lambda$ in $V$ such that $W\models \mu=\aleph_1$.
Since $\mathcal F^{\fin}_{|\lambda|}$ projects to $\mathcal F^{\bd}_\mu$, we know that
\begin{itemize}
\item[(iii)] $W\models (T,<_T)\text{ admits an }\mathcal F^{\bd}_\mu\text{-ascent path}$.
\end{itemize}
By (i),(iii) and \cite{MR964870}, then, we must have $W\models \cf(\mu)=\cf(|\lambda|)$.
As $\lambda^+$ was not collapsed, and $V\models\square_\lambda$, we get from \cite[Page 440]{MR675955} that $W\models\cf(|\lambda|)=\cf(\lambda)$. But $V\models\cf(\lambda)=\aleph_0$, and so
$$W\models\aleph_1=\cf(\mu)=\cf(|\lambda|)=\cf(\lambda)=\aleph_0.$$ This is a contradiction.
\end{proof}

\begin{corollary}\label{cor62} Suppose  $\diamondsuit(E^{\lambda^+}_{\lambda})+\gch$ holds for a given regular uncountable cardinal $\lambda$.

Then there exists a $\lambda$-complete $\lambda^+$-Souslin tree $(T,<_T)$,
satisfying the following. For every infinite $\mu<\lambda$, there exists a uniform ultrafilter $\mathcal U$ over $\mu$ such that  $T^\mu/\mathcal U$ is $\lambda^+$-Aronszajn and almost Souslin.
\end{corollary}
\begin{proof} By Theorem \ref{theta-ascent-free-thm2} below, $\p_{14}(\lambda^+,{\sq},\lambda^+,\{E^{\lambda^+}_{\lambda}\})+\lambda^{<\lambda}=\lambda$ entails
a \emph{slim}, Hausdorff,  $\lambda$-free $\lambda^+$-Souslin tree.
By the explanation opening subsection \ref{subsectioncomplete}, the same proof shows that $\p_{14}(\lambda^+,{\sql},\lambda^+,\{E^{\lambda^+}_{\lambda}\})+\lambda^{<\lambda}=\lambda$
entails a \emph{$\lambda$-complete},  Hausdorff,   $\lambda$-free, $\lambda^+$-Souslin tree.
By \cite{axioms}, for $\lambda$ uncountable, $\diamondsuit(E^{\lambda^+}_\lambda)$ entails $\p_{14}(\lambda^+,{\sql},\lambda^+,\{E^{\lambda^+}_{\lambda}\})$.
Thus, let $(T,<_T)$ be a Hausdorff  $\lambda$-complete  $\lambda$-free  $\lambda^+$-Souslin tree.
Now, given an infinite $\mu<\lambda$,
we utilize GCH to pick a selective ultrafilter $\mathcal U$ over $\mu$.
Then, by Lemmas \ref{lemma25} and \ref{lemma28} below,  $T^\mu/\mathcal U$ is $\lambda^+$-Aronszajn and almost Souslin.
\end{proof}

We remark that assuming $\sd_{\lambda}$ (a strong combination of the principles $\square_\lambda$ and $\diamondsuit(\lambda^+)$), one can construct a \emph{slim} $\lambda^+$-Souslin tree
whose reduced $\mu$-power by \emph{any} uniform ultrafilter over any cardinal $\mu$ such that $\lambda^\mu=\lambda$ is $\lambda^+$-Aronszajn and almost Souslin.
The construction requires additional ideas, which we feel are out of the scope of this already-lengthy paper.

\begin{corollary}
Suppose $\square_\lambda + \diamondsuit^*(\lambda^+) + \lambda^\theta=\lambda$ holds for given infinite cardinals $\cf(\theta)=\theta<\lambda$.

Then there exists a $\lambda^+$-Souslin tree whose reduced $\mu$-power (by any uniform ultrafilter) is $\lambda^+$-Kurepa, for all infinite $\mu\le\theta$.
\end{corollary}
\begin{proof} By \cite{MR2781096}, $\diamondsuit^*(\lambda^+)$ together with $\lambda^{\aleph_0}=\lambda$ entails $\diamondsuit^+(\lambda^+)$,
which in turn entails a $\lambda^+$-Kurepa tree, $(U,\subset)$, where $U\s{}^{<\lambda^+}2$ is downward-closed.
By \cite{axioms}, $\square_\lambda+\diamondsuit(\lambda^+)$ entails $\p_{14}(\lambda^+,{\sq},\theta,\{E^{\lambda^+}_{\ge\theta}\})$.
Now, appealing to Theorem~\ref{tree-theta-ascent-complete-thm} below, with $(\nu,\theta,\chi,\kappa)=(\aleph_0,\theta,\theta,\lambda^+)$,
we obtain a $\theta$-complete  $\lambda^+$-Souslin tree $(T,<_T)$
with an injective $(\mathcal F^{\fin}_{\theta},U)$-ascent path. In particular, for every infinite cardinal $\mu\le\theta$,
since $\mathcal F^{\fin}_{\theta}$ projects to $\mathcal F^{\bd}_{\mu}$,
the reduced $\mu$-power of $T$ (by any uniform ultrafilter) would contain a copy of $(U,\subset)$,
hence, is  $\lambda^+$-Kurepa.
\end{proof}

\begin{corollary} If $\sd_{\lambda}^++\gch$ holds, then for every regular cardinal $\theta<\cf(\lambda)$,
there exists a $\lambda^+$-Souslin tree $(T,<_T)$, satisfying the following.
\begin{itemize}
\item If $\aleph_0\le\mu<\theta$, then there exists a uniform ultrafilter $\mathcal U$ over $\mu$ for which $T^\mu/\mathcal U$ is $\lambda^+$-Aronszajn and almost Souslin;
\item $T^{\theta}/\mathcal U$ is $\lambda^+$-Kurepa and not almost Souslin  for every uniform ultrafilter $\mathcal U$ over $\theta$.
\end{itemize}
\end{corollary}
\begin{proof} By \cite{rinotschindler}, $\sd_\lambda^+$ entails a $\p^-_{14}(\lambda^+,{\sq},\lambda^+,\{E^{\lambda^+}_{\cf(\lambda)}\})$-respecting downward-closed tree $X\s{}^{<\lambda^+}\lambda$
that is the disjoint union of a $\lambda^+$-Kurepa tree and a special $\lambda^+$-tree. In particular, $(X,\subset)$ is $\lambda^+$-Kurepa and not almost Souslin.
By Theorem \ref{free-with-u-ascent} below, then, there exists a Hausdorff  $\theta$-free,  $\cf(\lambda)$-complete,  $\lambda^+$-Souslin tree, $(T,<_T)$
that admits an injective $(\mathcal F^{\bd}_{\theta},X)$-ascent path.
In particular, for any uniform ultrafilter $\mathcal U$ over $\theta$,
$T^{\theta}/\mathcal U$ contains a copy of $(X,\subset)$ and hence is $\lambda^+$-Kurepa and not almost Souslin.

Finally, given an infinite $\mu<\theta$, we utilize GCH to pick a selective ultrafilter $\mathcal U$ over $\mu$.
Then, by Lemmas \ref{lemma25} and \ref{lemma28} below,  $T^\mu/\mathcal U$ is $\lambda^+$-Aronszajn and almost Souslin.
\end{proof}

Here, the principle  $\sd_\lambda^+$ stands for a certain strong combination of the principles $\square_\lambda$ and $\diamondsuit^+(\lambda^+)$.
It was introduced in \cite{rinotschindler}, where it was proven to hold in $L$ for every infinite cardinal $\lambda$.

\begin{corollary} If $\sd_\lambda+\gch$ holds,
then for every infinite cardinal $\theta<\cf(\lambda)$,
there exists a  $\lambda^+$-Souslin tree $(T,<_T)$,
satisfying the following.
\begin{itemize}
\item If $\aleph_0\le\mu\le\theta$, then $T^\mu/\mathcal U$ is not $\lambda^+$-Aronszajn for every  uniform ultrafilter $\mathcal U$ over $\mu$;
\item If $\theta<\mu<\cf(\lambda)$, then there exist a uniform ultrafilter $\mathcal U$ over $\mu$ such that $T^\mu/\mathcal U$ is $\lambda^+$-Aronszajn and almost Souslin.
\end{itemize}
\end{corollary}
\begin{proof}
 By \cite{axioms}, $\sd_\lambda$, entails $\p_{14}(\lambda^+,{\sq},\lambda^+,\{E^{\lambda^+}_{\cf(\lambda)}\})$.

$\blacktriangleright$ If $\theta^+=\cf(\lambda)$, then by Corollary \ref{cor52} below,
there exists a $\lambda^+$-Souslin tree with an $\mathcal F^{\fin}_\lambda$-ascent path.
As $\mathcal F^{\fin}_\lambda$ projects to $\mathcal F^{\bd}_\mu$ for every infinite $\mu\le\lambda$, we have established the first bullet,
and the second bullet is vacuous.

$\blacktriangleright$  If $\theta^+<\cf(\lambda)$,
then by Theorem \ref{theta-ascent-free-thm1} below, let us pick a Hausdorff $(\cf(\lambda),\theta^+)$-free $\lambda^+$-Souslin tree $(T,<_T)$ with an $\mathcal F^{\fin}_\theta$-ascent path.
As $\mathcal F^{\fin}_\theta$ projects to $\mathcal F^{\bd}_\mu$ for every infinite $\mu\le\theta$, we have established the first bullet.

Next, suppose that $\theta<\mu<\cf(\lambda)$. By $\gch$, let $\mathcal U$ be a selective ultrafilter over $\mu$.
Then, by Lemmas \ref{lemma25} and \ref{lemma28} below,  $T^\mu/\mathcal U$ is $\lambda^+$-Aronszajn and almost Souslin.
\end{proof}

We now give an even more informative corollary than the one stated in the abstract.

\begin{corollary}\label{cor115} If $\sd_{\aleph_6}+\gch$ holds,
then there exists an  $\aleph_7$-Souslin tree $(T,<_T)$,
and a sequence of uniform ultrafilters $\langle \mathcal U_n\mid n<7\rangle$ such that:
\begin{itemize}
\item If $n\in\{0,1,4,5\}$, then $T^{\aleph_n}/\mathcal U_n$ is $\aleph_7$-Aronszajn and almost Souslin;
\item If $n\in\{2,3,6\}$, then $T^{\aleph_n}/\mathcal U_n$ is not an $\aleph_7$-Aronszajn tree.
\end{itemize}
\end{corollary}
\begin{proof}  By \cite{axioms}, $\sd_{\aleph_6}$, entails $\p_{14}(\aleph_7,{\sq},\aleph_7,\{E^{\aleph_7}_{\aleph_6}\})$.
Thus, appealing to Theorem \ref{thm65} below with $(\nu,\theta,\chi,\kappa)=(\aleph_2,\aleph_3,\aleph_6,\aleph_7)$,
we then obtain a prolific Hausdorff $\aleph_2$-free $(\aleph_6,\aleph_4)$-free $\aleph_7$-Souslin tree $(T,<_T)$ with an $\mathcal F_{\aleph_3}^{\aleph_2}$-ascent path
(the filter $\mathcal F_\theta^\mu$ is defined on page \pageref{defnfnutheta}).

$\blacktriangleright$ As $(T,<_T)$ is $\aleph_2$-free, for any selective ultrafilters $\mathcal U_{0},\mathcal U_{1}$ over $\aleph_0, \aleph_1$, respectively,
we get from Lemmas \ref{lemma25} and \ref{lemma28} below,  that $T^{\aleph_0}/\mathcal U_{0}$ and $T^{\aleph_1}/\mathcal U_{1}$ are $\aleph_7$-Aronszajn and almost Souslin.

$\blacktriangleright$ As $\mathcal F_{\aleph_3}^{\aleph_2}$ projects to $\mathcal F_{\aleph_2}^{\bd}$, the reduced power $T^{\aleph_2}/\mathcal U_2$ by any uniform ultrafilter $\mathcal U_2$ over $\aleph_2$ contains a cofinal branch.

$\blacktriangleright$ As $\mathcal F_{\aleph_3}^{\aleph_2}\s \mathcal F_{\aleph_3}^{\bd}$,
the reduced power $T^{\aleph_3}/\mathcal U_3$ by any uniform ultrafilter $\mathcal U_3$ over $\aleph_3$ contains a cofinal branch.

$\blacktriangleright$ As $(T,<_T)$ is $(\aleph_6,\aleph_4)$-free, for any selective ultrafilters $\mathcal U_{4},\mathcal U_{5}$ over $\aleph_4, \aleph_5$, respectively,
we get from Lemmas \ref{lemma25} and \ref{lemma28} below,  that $T^{\aleph_4}/\mathcal U_{\aleph_4}$ and $T^{5}/\mathcal U_{5}$ are $\aleph_7$-Aronszajn and almost Souslin.

$\blacktriangleright$ By \cite[Proposition 4.3.5]{MR1059055}, let $\mathcal U_{6}$ be an $\aleph_6$-regular ultrafilter over $\aleph_6$.
As $(T,<_T)$ is prolific, we get that the $\aleph_6^{\text{th}}$ level of $T$ has size $\aleph_6$,
and then by  \cite[Proposition 4.3.7]{MR1059055}, the  $\aleph_6^{\text{th}}$ level of $T^{\aleph_6}/\mathcal U_{6}$ has size $\aleph_7$. In particular, $T^{\aleph_6}/\mathcal U_{6}$  it is not an $\aleph_7$-Aronszajn tree.
\end{proof}

We conclude with a non-trivial improvement of Theorem 2 from \cite{MR836425}.
In particular, demonstrating the consistency of:  \emph{all $\lambda^+$-Aronszajn trees are nonspecial
for every singular cardinal $\lambda$ of countable cofinality}.

\begin{corollary}\label{cor116} If $\zfc+\exists$ supercompact cardinal is consistent, then so is $\zfc + \text{Martin's Maximum}+\text{ the following}$:
\begin{enumerate}
\item There exists a $\cf(\lambda)$-free $\cf(\lambda)$-complete $\lambda^+$-Souslin tree for every cardinal $\lambda\ge\aleph_2$;
\item There exists no special $\lambda^+$-Aronszajn trees, for every singular cardinal $\lambda$ of countable cofinality.
\end{enumerate}
\end{corollary}
\begin{proof} By \cite{axioms}, it is consistent, relative to ZFC + $\exists$ supercompact cardinal,
that all of the following hold together:
\begin{itemize}
\item ($\zfc$ and) Martin's Maximum ($\mm$);
\item $\p_{14}(\lambda^+,{\sq_{\aleph_2}},\lambda^+,\{E^{\lambda^+}_{\cf(\lambda)}\})$ for every singular cardinal $\lambda$;
\item $\p_{14}(\lambda^+,{\sql},\lambda^+,\{E^{\lambda^+}_\lambda\})$ for every regular uncountable cardinal $\lambda$.
\end{itemize}
Work in this model. By the second and the third bullets, $\ch_\lambda$ holds for every uncountable cardinal $\lambda$.

(1) Let $\lambda$ denote a regular cardinal $\ge\aleph_2$. Then $\lambda^{<\lambda}=\lambda$.
As explained in the proof of Corollary~\ref{cor62},
$\p_{14}(\lambda^+,{\sql},\lambda^+,\{E^{\lambda^+}_{\lambda}\})+\lambda^{<\lambda}=\lambda$
entails a $\lambda$-free $\lambda$-complete $\lambda^+$-Souslin tree.

Let $\lambda$ denote a singular cardinal. Write $\mu=\max\{\cf(\lambda),\aleph_2\}$. By $\p_{14}(\lambda^+,{\sq_{\mu}},\lambda^+,\{E^{\lambda^+}_{\cf(\lambda)}\})$ and Theorem \ref{free-with-u-ascent} below,
taking $(\chi,\eta,\nu,\theta,\kappa)=(\cf(\lambda),\cf(\lambda),\mu,\mu,\lambda^+)$
and $U=\bigcup_{\alpha<\lambda^+}{}^\alpha1$,
we infer the existence of a $\cf(\lambda)$-free $\cf(\lambda)$-complete $\lambda^+$-Souslin tree.

(2) By \cite{MR2811288}, $\mm$ refutes $\square_{\lambda}^*$ for every singular cardinal $\lambda$ of countable cofinality. As $\square_{\lambda}^*$ is equivalent to the existence of a special $\lambda^+$-Aronszajn tree, we are done.
\end{proof}

Looking at the versatile list of hypotheses of the above corollaries (Cohen forcing, $\square_\lambda+\ch_\lambda,\diamondsuit(E^{\lambda^+}_\lambda)+\gch$, a model of $\mm$)
demonstrates well the utility of the proxy principle $\p(\kappa,\dots)$
as a device that provides a disconnection between the tree constructions and the study of the combinatorial hypotheses.

\section{Some theory of trees}\label{section2}

\begin{definition} A tree $(T,<_T)$ is said to be \emph{Hausdorff} if for all $x,y\in T$,
$x_\downarrow=y_\downarrow$ entails that $x=y$.
\end{definition}

\begin{definition}
A tree $(T,<_T)$ is said to be \emph{normal} if for every $\alpha < \beta$
and $x \in T_\alpha$, if $T_\beta\neq\emptyset$,
then there exists some $y \in T_\beta$ such that $x <_T y$.
\end{definition}

\begin{definition}
A tree $(T,<_T)$ is said to be \emph{splitting} if any node in $T$ admits at least two immediate successors.
\end{definition}

The following is a basic, yet very useful, fact.

\begin{lemma}\label{normalandsplitting} Suppose that $(T,<_T)$ is a $\kappa$-Souslin tree for some regular uncountable cardinal $\kappa$.

Then there exists a club $E$ in $\kappa$ such that $(T\restriction E,<_T)$ is normal and splitting.
\end{lemma}
\begin{proof}
Let $H=\{ x\in T\mid \exists \beta<\kappa\forall y\in T_\beta(y\text{ is incompatible with }x)\}$.
Since $(T,<_T)$ has no antichains of size $\kappa$,  we may pick a large enough $\alpha<\kappa$ such that $H\s T\restriction\alpha$.
Let $G=\{ x\in T\mid \left(x^\uparrow,<_T\right)\text{ is linearly ordered}\}$. Since $(T,<_T)$ has no cofinal branches, we know that $G\s H$.
Consequently,
$$E=\{ \beta<\kappa\mid (\forall x\in T\restriction[\alpha,\beta))(\exists y_0,y_1\in T\restriction \beta)[y_0,y_1\text{ are incompatible extensions of $x$}]\}\setminus\alpha$$
is a club in $\kappa$. Evidently, $(T\restriction E,<_T)$ is normal and splitting.
\end{proof}

\begin{definition}
A subtree $T$ of ${}^{<\kappa}\kappa$ is said to be \emph{prolific} if for every $\alpha < \kappa$ and every $x\in T\cap{}^\alpha\kappa$,
we have $\{ x^\smallfrown \langle i\rangle\mid  i<\max \{\omega, \alpha\} \}\s T$.
\end{definition}

Notice that a prolific tree is always splitting. On the opposite extreme from prolific, we have the following.

\begin{definition}
A $\kappa$-tree is said to be \emph{binary} if it is a downward-closed
subtree of the complete binary tree ${}^{<\kappa}2$.
\end{definition}

As mentioned earlier, in the Appendix section below, we analyze a natural process that produces,
for any given $\kappa$-tree, a corresponding binary $\kappa$-tree sharing many properties of the original one.

\begin{definition} A filter $\mathcal F$ over a cardinal $\theta$ is said to be \emph{selective} if it is uniform,
and for every function $f$ with $\dom(f)\in\mathcal F$, one of the following holds:
\begin{itemize}
\item there exists some $C\in\mathcal F^+$ such that $f\restriction C$ is constant, or
\item there exists some $I\in\mathcal F^+$ such that $f\restriction I$ is injective.
\end{itemize}
\end{definition}

Note that for every infinite cardinal $\theta$, $\mathcal F^{\bd}_\theta$ is a selective filter.
As for ultrafilters, note that
assuming $\gch$, for any infinite regular cardinal $\theta$, one can enumerate all functions from $\theta$ to $\theta$ in order-type $\theta^+$,
and then use this enumeration to recursively construct a tower that generates a selective ultrafilter over $\theta$.

\begin{definition}\label{tree-indexed-ascent-def}
An $(\mathcal F,X)$-ascent path $\vec f =\langle f_x\mid x\in X\rangle$ through a $\kappa$-tree $(T,<_T)$ is said to be \emph{injective}
if for every $b : \kappa \to \kappa$ such that $\{ b \restriction \alpha \mid \alpha < \kappa \} \subseteq X$,
there exist some $\alpha < \kappa$ and some $I \in \mathcal F$ such that
$f_{b \restriction \alpha} \restriction I$ is injective.
\end{definition}

Recall that if the downward-closed tree $(X,\subset)$ is isomorphic to $(\kappa,\in)$, then an $(\mathcal F, X)$-ascent path $\vec f$ is said to be an $\mathcal F$-ascent path.
In this special case, it is customary to identify $\vec f$ with a sequence $\langle f_\alpha\mid\alpha<\kappa\rangle$.
So, $\vec f$ is injective iff there exist $\alpha < \kappa$ and $I\in \mathcal F$ such that $f_\alpha \restriction I$ is injective.
Note that if $\mathcal F$ is a filter, then $\vec f$ is injective iff for co-boundedly many
$\alpha < \kappa$, there exist  $I_\alpha \in \mathcal F$ such that $f_\alpha \restriction I_\alpha$ is injective.

\begin{lemma}\label{lemma26}
Suppose that $\kappa$ is a regular uncountable cardinal, and
$(T,<_T)$ is a $\theta^+$-free $\kappa$-Souslin tree.
Then none of the following can occur:
\begin{enumerate}
\item $(T,<_T)$ admits an $\mathcal F$-ascent path for some selective filter $\mathcal F$ over $\theta$;
\item $(T,<_T)$ admits an injective $\mathcal F$-ascent path for some (proper) filter $\mathcal F$ over $\theta$.
\end{enumerate}
\end{lemma}

\begin{proof} (1) Suppose not.
Then by the results of the Appendix section below, we may assume that $\left(T,<_T\right)=\left(T,\subset\right)$ is a binary  $\theta^+$-free $\kappa$-Souslin tree
that admits an $\mathcal F$-ascent path, for some selective filter $\mathcal F$ over $\theta$.
Let  $\vec f=\langle f_\alpha:\theta\rightarrow T_\alpha\mid \alpha<\kappa\rangle$
denote an  $\mathcal F$-ascent path through $(T,{\subset})$.

By Lemma \ref{lemma85} below, $2^\theta<\kappa$.
As $\mathcal F$ is a selective filter over $\theta$, and $2^\theta<\kappa$, one of the following must hold:
\begin{enumerate}
\item[(a)] There exists a stationary $S_0\s\kappa$ and $C\in\mathcal F^+$ such that $f_\alpha\restriction C$ is constant for all $\alpha\in S_0$;
\item[(b)] There exists a stationary $S_1\s\kappa$ and $I\in\mathcal F^+$ such that $f_\alpha\restriction I$ is injective for all $\alpha\in S_1$.
\end{enumerate}

In case (a), we then get that $\{ f_\alpha(\min(C))\mid \alpha\in S_0\}$ generates a cofinal branch through the $\kappa$-Souslin tree $(T,{\subset})$.
This is a contradiction.

In case (b), we do the following. Pick $\varepsilon<\kappa$ such that $f_\varepsilon\restriction I$ is injective.
By Lemma \ref{normalandsplitting}, let $E\s\kappa$ be a club such that $(T\restriction E,\subset)$ is normal and splitting.
By discarding an initial segment of $E$, we may assume that $\min(E)>\varepsilon$.

For every $\alpha\in E$, the set $Z_\alpha=\{ i<\theta\mid f_{\varepsilon}(i) \subseteq f_\alpha(i)\}$ is in $\mathcal F$.
By $2^\theta<\kappa$, we can then pick $Z\in\mathcal F$ and some stationary subset $S\s E$ such that $Z_\alpha=Z$ for all $\alpha\in S$.
As $I\in\mathcal F^+$ and $Z\in\mathcal F$, the set $I'=I\cap Z$ is nonempty,
hence we consider the derived tree $\hat T=\bigotimes_{i\in I'}  \cone{f_{\varepsilon}(i)}$.

Let $\alpha\in S$ be arbitrary. Denote $\alpha^+=\min(S\setminus(\alpha+1))$.
Since $\alpha<\alpha^+$ are elements of $E$, we may find $g_{\alpha^+}\in (T_{\alpha^+})^{I'}$ such that for all $i\in I'$:
$$g_{\alpha^+}(i)\restriction\alpha=f_\alpha(i)\text{ and }g_{\alpha^+}(i)\neq f_{\alpha^+}(i).$$
In particular, $f_\varepsilon(i)\s g_{\alpha^+}(i)$ for all $i\in I'\s Z_{\alpha}$, and hence $g_{\alpha^+}\in\hat T$.

Since $|I'|<\theta^+$ and $(T,{\subset})$ is $\theta^+$-free,
 $(\hat T,<_{\hat T})$ is $\kappa$-Souslin,
thus, let us pick $\alpha<\beta$ in $S$ such that $g_{\alpha^+}(i)\s g_{\beta^+}(i)$ for all $i\in I'$.
In particular, $\alpha^+\le\beta$.
Since $\vec f$ is an $\mathcal F$-ascent path, the following set
$$A_{\alpha^+,\beta}=\{ i<\theta\mid f_{\alpha^+}(i)\s f_\beta(i)\}$$
is in $\mathcal F$. In particular, $I'\cap A_{\alpha^+,\beta}$ is nonempty. Pick $i\in I'\cap A_{\alpha^+,\beta}$.
Then  $f_{\alpha^+}(i)=f_\beta(i)\restriction\alpha^+=(g_{\beta^+}(i)\restriction\beta)\restriction\alpha^+$.
Recalling that $g_{\alpha^+}(i)\neq f_{\alpha^+}(i)$, we conclude that $g_{\alpha^+}(i)\not\s g_{\beta^+}(i)$, contradicting the choice of $\alpha<\beta$.

(2) Compared to Clause (1), instead of assuming that $\mathcal F$ is selective, the hypothesis readily asserts the existence of $\varepsilon<\kappa$ and $I\in\mathcal F$ such that $f_\varepsilon\restriction I$ is injective.
In particular, $I\in\mathcal F^+$, and the rest of the proof is identical.
\end{proof}

\begin{corollary} If $\theta<\kappa=\cf(\kappa)$ are infinite cardinals,
then no $\theta^+$-free $\kappa$-Souslin tree admits an  $\mathcal F_\theta^{\bd}$-ascent path.
\end{corollary}

We now introduce a two-cardinal version of freeness. As Corollary \ref{cor115} demonstrates, this is a fruitful concept.

\begin{definition}\label{deftwofree}
A $\kappa$-tree $(T,<_T)$ is said to be \emph{$(\chi,\eta)$-free} if for every nonzero $\tau<\chi$,
any $\beta < \kappa$, and any sequence of distinct nodes
$\langle w_i \mid i < \tau \rangle \in {}^\tau T_\beta$, for all $A\s \bigotimes_{i<\tau}  \cone{w_i}$ of size $\kappa$,
there exist  $\vec x$ and $\vec y$ in $A$ such that $|\{ i<\tau\mid \neg( \vec x(i)<_T\vec  y(i))\}|<\eta$.
\end{definition}

\begin{lemma}\label{lemma25} \begin{enumerate}
\item A $\kappa$-Souslin tree is $\chi$-free iff it is  $(\chi,1)$-free;
\item If $\chi_0\ge\chi_1$ and $\eta_0\le\eta_1$, then $(\chi_0,\eta_0)$-free implies $(\chi_1,\eta_1)$-free.
\end{enumerate}
\end{lemma}
\begin{proof} Obvious.\end{proof}

\begin{lemma}\label{lemma28} Suppose that $(X,<_X)$ is a Hausdorff $(\theta^+,\theta)$-free $\kappa$-Souslin tree,
$\mathcal U$ is a selective ultrafilter over  $\theta$, and $\lambda^{<\theta}<\kappa$ for all $\lambda<\kappa$. Then:
\begin{enumerate}
\item If $A\s X^\theta/\mathcal U$ is an antichain, then $\{\height(x)\mid x\in A\}\cap E^{\kappa}_{>\theta}$ is nonstationary;
\item $X^\theta/\mathcal U$ is $\kappa$-Aronszajn.
\end{enumerate}
\end{lemma}
\begin{proof}
By Lemma \ref{lemma85} below, we infer that $\lambda^\theta<\kappa$ for all $\lambda<\kappa$.

For every $x\in X^\theta$, let $\height(x)$ denote the height of $x(0)$ in $(X,<_X)$. Note that $\height(x)$ coincides with the height of $[x]_{\mathcal U}$ in $X^\theta/\mathcal U$.

(1) Suppose we are given $A\s X^\theta$, for which $S=\{\height(x)\mid x\in A\}\cap E^{\kappa}_{>\theta}$ is stationary. We shall prove that $\{ [x]_{\mathcal U}\mid x\in A\}$ is not an antichain.

For every $\alpha\in S$, pick $x_\alpha\in A$ with $\height(x_\alpha)=\alpha$.
As $\mathcal U$ is a selective ultrafilter over $\theta$, and $2^\theta<\kappa$, one of the following must hold:
\begin{enumerate}
\item[(a)] There exists a stationary $S_0\s S$ and $C\in\mathcal U$ such that $x_\alpha\restriction C$ is constant for all $\alpha\in S_0$;
\item[(b)] There exists a stationary $S_1\s S$ and $I\in\mathcal U$ such that $x_\alpha\restriction I$ is injective for all $\alpha\in S_1$.
\end{enumerate}

In case (a), we use the fact $(X,<_X)$ is $\kappa$-Souslin, to find $\alpha<\beta$ in $S_0$ such that $x_\alpha(\min(C))<_X x_\beta(\min(C))$.
Consequently, $[x_\alpha]_{\mathcal U}$ and $[x_\beta]_{\mathcal U}$ are two  compatible elements. Thereby, $\{ [x]_{\mathcal U}\mid x\in A\}$ is not an antichain.

In case (b), since $|I|=\theta<\cf(\alpha)$ and $(X,<_X)$ is Hausdorff,
we may find some large enough ordinal $\beta_\alpha<\alpha$ such that $x_\alpha(i)_\downarrow\cap (X\restriction\beta_\alpha)\neq x_\alpha(j)_\downarrow\cap (X\restriction\beta_\alpha)$ for all two distinct $i,j\in I$.
For  $\alpha\in S_1$, let $y_\alpha:I\rightarrow X_{\beta_\alpha}$ be such that
$y_\alpha(i)$ is the unique element of $X_{\beta_\alpha}$ that is $<_X$ below $x_\alpha(i)$, for all $i\in I$.
Of course, for all $\alpha\in S_1$, $y_\alpha$ is an injection.

So, fix a stationary $S_2\s S_1$ such that $\{y_\alpha\mid \alpha\in S_2\}$ is a singleton, say $\{y\}$.
Then $\{ x_\alpha\mid \alpha\in S_2\}$ is a $\kappa$-sized set in the derived tree $\hat T=\bigotimes_{i\in I}  {y(i)}^\uparrow$.
Since $(X,<_X)$ is $(\theta^+,\theta)$-free, there exist $\alpha<\beta$ in $S_2$ such that $|\{ i\in I\mid \neg(x_\alpha(i)<_X x_\beta(i))\}|<\theta$.
Since $I\in \mathcal U$ and latter is uniform, we altogether get that $\{ i<\theta\mid x_\alpha(i)<_X x_\beta(i))\}\in\mathcal U$.
Consequently, $[x_\alpha]_{\mathcal U}$ and $[x_\beta]_{\mathcal U}$ are two  compatible elements. Thereby, $\{[x]_{\mathcal U}\mid x\in A\}$ is not an antichain.

\smallskip

(2) Since $\lambda^\theta<\kappa$ for all $\lambda<\kappa$,  $X^\theta/\mathcal U$ is a $\kappa$-tree.
Towards a contradiction, suppose that we may find $x_\alpha\in  X^\theta$ for each $\alpha<\kappa$, in such a way that $\{ [x_\alpha]_\mathcal U\mid \alpha<\kappa\}$ is a cofinal branch through $X^\theta/\mathcal U$.
As $\mathcal U$ is a selective ultrafilter, and $(X,<_X)$ is $\kappa$-Aronszajn,
an analysis similar to the above entails a stationary $S_1\s\kappa$ and  $I\in\mathcal U$ such that $x_\alpha\restriction I$ is injective for all $\alpha\in S_1$.
Pick an arbitrary $\varepsilon\in S_1$.

By Lemma \ref{normalandsplitting}, let $E\s\kappa$ be a club such that $(X\restriction E,<_X)$ is normal and splitting.
Without loss of generality, $\min(E)>\varepsilon$.
Find a stationary subset $S\s E$, and some  $Z\in\mathcal U$ such that
$Z=\{ i<\theta\mid x_{\varepsilon}(i) <_X x_\alpha(i)\}$ for all $\alpha\in S$.
Put $I'=Z\cap I$, and consider the derived tree  $\hat T=\bigotimes_{i\in I'}  x_{\varepsilon}(i)^\uparrow$.

As in the proof of Lemma \ref{lemma26},
for all $\alpha\in S$, we let $\alpha^+=\min(S\setminus(\alpha+1))$,
and pick $g_{\alpha^+}\in\hat T\cap(T_{\alpha^+})^{I'}$
such that for all $i\in I'$:
\begin{itemize}
\item $g_{\alpha^+}(i)_\downarrow\cap(T\restriction\alpha)=x_\alpha(i)_\downarrow$;
\item $g_{\alpha^+}(i)\neq x_{\alpha^+}(i)$.
\end{itemize}

Since $|I'|<\theta^+$ and $(T,{\subset})$ is $(\theta^+,\theta)$-free,
we may pick $\alpha<\beta$ in $S$ such that
$$B=\{i\in I'\mid \neg(g_{\alpha^+}(i)<_X g_{\beta^+}(i))\}$$ is of cardinality $<\theta$.
Since  $I\in \mathcal U$, and the latter is a uniform ultrafilter over $\theta$, $I'\setminus B\in\mathcal U$.
By the choice of $x_\alpha$ and $x_\beta$, also
$$A_{\alpha^+,\beta}=\{ i<\theta\mid x_{\alpha^+}(i)<_X x_\beta(i)\}$$ is in $\mathcal U$.
Then  $(I'\setminus B)\cap A_{\alpha^+,\beta}$ is nonempty, contradicting the choice of $\alpha<\beta$.
\end{proof}

\section{The microscopic approach to Souslin-tree constructions}\label{section3}

All of the trees that are constructed in this paper will be normal prolific subtrees of $^{<\kappa}\kappa$
for some regular uncountable cardinal $\kappa$.
Each node of such a tree $T$ is a function $t : \alpha \to \kappa$ for some ordinal $\alpha < \kappa$;
the tree order $<_T$ is simply extension of functions $\subset$;
and we require that if $t: \alpha \to \kappa$ is in $T$, then $t \restriction \beta \in T$ for every $\beta < \alpha$.
For any node $t \in T$, the height of $t$ in $T$ is just its domain, that is, $\height(t) = \dom(t)$,
and the set of its predecessors is $t_\downarrow = \{ t \restriction \beta \mid \beta < \dom(t) \}$.
For any $\alpha < \kappa$, the level $T_\alpha$ of the tree $T$ will be the set of all elements of $T$
that have domain $\alpha$, that is, $T_\alpha = T \cap {}^\alpha \kappa$.
Any function $f : \kappa \to \kappa$ determines a cofinal branch through $^{<\kappa} \kappa$, namely
$\{ f \restriction \alpha \mid \alpha < \kappa \}$, which may or may not be a subset of a given tree $T$.

The main advantage of this notational approach is the ease of completing a branch at a limit level.
Suppose that, during the process of constructing a tree $T$, we have already inserted into $T$
a $\subseteq$-increasing sequence of nodes $\langle t_\alpha \mid \alpha < \beta\rangle$ for some $\beta<\kappa$.
The (unique) limit of this sequence, which may or may not be a member of $T$, is simply $\bigcup_{\alpha < \beta} t_\alpha.$

It is clear that any downward-closed tree $T \subseteq {}^{<\kappa}\kappa$ is Hausdorff,
and conversely that any Hausdorff tree is isomorphic to a downward-closed subtree of $^{<\kappa}\kappa$
for some cardinal $\kappa$.

\medskip

While classical constructions of $\kappa$-Souslin trees typically involve
a recursive process of determining a partial order $<_T$ over $\kappa$
by advising with a $\diamondsuit(\kappa)$-sequence,
here, the order is already known (being $\subset$),
and the recursive process involves the determination of a subset of ${}^{<\kappa}\kappa$.
For this reason, it is more convenient to work with the following variation of $\diamondsuit(\kappa)$:

\begin{definition}[\cite{axioms}]\label{defDiamondHkappa}
$\diamondsuit(H_\kappa)$ asserts the existence of a partition $\langle R_i \mid  i < \kappa \rangle$ of $\kappa$
and a sequence $\langle S_\beta \mid \beta < \kappa \rangle$ of subsets of $H_\kappa$
such that for every $p\in H_{\kappa^{+}}$,  $i<\kappa$, and $\Omega \subseteq H_\kappa$,
there exists an elementary submodel $\mathcal M\prec H_{\kappa^{+}}$ such that:
\begin{itemize}
\item $p\in\mathcal M$;
\item $\mathcal M\cap\kappa\in  R_i$;
\item $\mathcal M\cap \Omega = S_{\mathcal M\cap\kappa}$.
\end{itemize}
\end{definition}

Notice that if we let $Z_\beta=S_\beta$ whenever $S_\beta\s\beta$, and $Z_\beta=\emptyset$ otherwise,
then $\langle Z_\beta\mid \beta<\kappa\rangle$ forms a $\diamondsuit(\kappa)$-sequence. A converse is also available:

\begin{fact}[\cite{axioms}]\label{fact32} $\diamondsuit(\kappa)$ is equivalent to $\diamondsuit(H_\kappa)$ for any regular uncountable cardinal $\kappa$.
\end{fact}

\label{section:basicsetup}

Let us uncover the setup that $\diamondsuit(H_\kappa)$ entails.
The functions, sequences, and notation defined here do not rely on anything but diamond, and
will be used in all of the tree constructions in this paper.

Fix a partition $\langle R_i \mid  i < \kappa \rangle$ of $\kappa$,
and a sequence $\langle S_\beta \mid \beta < \kappa \rangle$ of subsets of $H_\kappa$
as in Definition  $\ref{defDiamondHkappa}$.
Fix a bijection $\phi:\kappa\leftrightarrow H_\kappa$.
Let $<_\phi$ denote the well-ordering that $\phi$ induces on  $H_\kappa$.
That is,  $x<_\phi y$ iff $\phi^{-1}(x) < \phi^{-1}(y)$.
Define $\psi:\kappa\rightarrow H_\kappa$ by letting $\psi(\beta)=\phi(i)$ for the unique $i<\kappa$ such that $\beta\in R_i$.

\smallskip

For every $T\in H_\kappa$, denote $\beta(T)=0$ unless there exists $\beta<\kappa$
such that $T\s{}^{\le\beta}\kappa$ and $T \nsubseteq {}^{<\beta}\kappa$.
Of course, in the latter case, $\beta$ is uniquely determined, so we denote $\beta(T)=\beta$ for this unique $\beta$.

We now define three functions:
\begin{enumerate}
\item The default extension function,
 $\defaultaction:H_\kappa\times H_\kappa\rightarrow H_\kappa$, is defined as follows.

Let $\defaultaction(x,T)=x$, unless $\bar Q=\{ z\in T\cap{}^{\beta(T)}\kappa\mid x\s z\}$ is non-empty,
in which case we let $\defaultaction(x,T)=\min(\bar Q,<_\phi)$.

\item The function for sealing antichains, $\sealantichain:H_\kappa\times H_\kappa\rightarrow H_\kappa$,
is defined as follows.

Let $\sealantichain(x,T)=\defaultaction(x,T)$,
unless $Q=\{ z\in T\cap{}^{\beta(T)}\kappa\mid \exists y\in S_{\beta(T)}( x\cup y\s z)\}$ is non-empty,
in which case we let $\sealantichain(x,T)=\min(Q,<_\phi)$.

\item
The function for sealing antichains in the product tree,
$\free:H_\kappa\times H_\kappa\times H_\kappa\rightarrow H_\kappa$,
is defined as follows.

Given $(x,T,\vec{b})$: if there exists $\tau<\kappa$ such that $\vec b\in{}^\tau T$
and
\[
Q=\left\{ \vec z\in {}^\tau(T\cap{}^{\beta(T)}\kappa) \mid
    \exists \vec y\in S_{\beta(T)} \cap {}^\tau T
        \left( \forall \xi<\tau\  \vec b(\xi)\cup \vec y(\xi)\s \vec z(\xi)\right) \right\}
\]
 is non-empty,
let $\vec z=\min(Q,<_\phi)$. If there exists a unique $\xi<\tau$ such that $x = \vec b(\xi)$,
let $\free (x,T,\vec b)=\vec z(\xi)$. Otherwise, Let $\free(x,T,\vec b)=\defaultaction(x,T)$.%
\footnote{Notice that $\free(x,T,\langle x\rangle)=\sealantichain(x,T)$ whenever $x \in T$.}

\end{enumerate}
The following is obvious.

\begin{extension-lemma}\label{extendfact}
If $x\in T\in H_\kappa$ and $T$ is a normal subtree of ${}^{\le\beta(T)}\kappa$,
then $\defaultaction(x,T)$, $\sealantichain(x,T)$, and $\free(x,T,\vec b)$ are
elements of $T\cap{}^{\beta(T)}\kappa$ extending $x$.
\end{extension-lemma}

A core component of the uniform construction of $\kappa$-Souslin trees in this paper
is the following. Given $\langle C_\alpha\mid\alpha<\kappa\rangle$, we first derive some stationary subset $\Gamma$ of $\kappa$.
Then, for every $\alpha\in \Gamma$ for which the initial tree $T\restriction\alpha$ has already been defined, and every $x \in T \restriction C_\alpha$,
we identify a branch $b^\alpha_x$ through $T \restriction \alpha$, containing $x$.
The recursive process of identification of such a branch consists of ``microscopic'' steps, where each step involves invoking one of the three functions mentioned in the Extension Lemma,
with parameters that are suggested by the oracle $\psi$.
As usual, to be able to keep climbing up, it is necessary that at a limit step $\bar\alpha$ of the recursion,
the limit of the portion of the branch identified so far, $b^\alpha_x \restriction \bar\alpha$,
is an element of $T_{\bar\alpha}$.
This will be accomplished by ensuring that such a portion was in fact already identified as $b^{\bar\alpha}_x$
at an earlier stage of the recursive construction of the tree, when constructing $T_{\bar\alpha}$.
That is:
\begin{claim-template}\label{coherence-template}
Suppose ${\bar\alpha} < \alpha$ are elements of $\Gamma$,
and $T \restriction \alpha$ has been constructed to satisfy the relevant properties.
Suppose also that $b^\alpha_y \restriction {\bar\alpha}$ has already been constructed for all $y\in T\restriction C_{\bar\alpha}$.

If  $C_{\bar\alpha} = C_\alpha \cap {\bar\alpha}$, then for all $x\in T\restriction C_{\bar\alpha}$:
\[
b^{\bar\alpha}_x = b^\alpha_x \restriction{\bar\alpha}.
\]
\end{claim-template}

As a by-product, some of the $\kappa$-Souslin trees $(T,{\subset})$ constructed in this paper will happen to be $\p^-_{14}(\kappa,\mathcal R,1,\{\Gamma\})$-respecting,
by simply letting for all $\alpha<\kappa$ and $x\in T\restriction C_\alpha$: $${\mathbf b^\alpha}(x)=\begin{cases}\bigcup\{b^\alpha_x(\beta)\mid \beta\in\dom(b^\alpha_x)\},&\alpha\in\Gamma\\
\emptyset,&\text{otherwise}\end{cases}.$$
Recalling Clause~(\ref{respectcohere}) of Definition \ref{respecting}, the reader can now probably guess the definition of the set  $\Gamma$.

\section{Souslin Tree with a Countable Ascent Path}\label{sectioncountablewidth}

In this section, we shall present constructions of $\kappa$-Souslin trees from $\p_{14}(\kappa,\mathcal R,\theta,\mathcal S)$,
where $\mathcal R= {\sq}$ and $\theta=1$. We remark that by \cite{axioms}, in $L$, this principle holds for every regular uncountable cardinal $\kappa$ that is not weakly compact.
As for $\kappa=\lambda^+$, where $\lambda$ is an uncountable cardinal, we have that $\square_\lambda+\ch_\lambda$ entails $\p_{14}(\lambda^+,{\sq},\theta,\{ E^{\lambda^+}_{\chi}\mid \aleph_0\le\cf(\chi)=\chi<\lambda\})$
for any ordinal $\theta<\lambda$.

\subsection{Slim Trees}

\begin{theorem}\label{ctbl-ascent-thm}
Suppose that $\kappa$ is a regular uncountable cardinal, and $\p_{14}(\kappa,  {\sqsubseteq}, 1, \{\kappa\})$ holds.

Then there exists a prolific slim $\kappa$-Souslin tree with an injective $\mathcal F^{\fin}_{\aleph_0}$-ascent path.
\end{theorem}
\begin{proof} Recalling Definition \ref{def16} and Fact \ref{fact32},
we know that $\p_{14}^-(\kappa,  {\sqsubseteq}, 1, \{\kappa\})$ and $\diamondsuit(H_\kappa)$ hold.

Let  $\langle C_\alpha \mid \alpha < \kappa \rangle$ be a witness to $\p_{14}^-(\kappa,  {\sqsubseteq}, 1, \{\kappa\})$.
Next, by $\diamondsuit(H_\kappa)$, we fix
the function $\phi:\kappa\leftrightarrow H_\kappa$,
sequences $\langle S_\beta \mid \beta<\kappa\rangle$, $\langle R_i  \mid i < \kappa \rangle$,
well-ordering $<_\phi$ of $H_\kappa$,
notation $\beta(T)$,
and the function $\sealantichain:H_\kappa\times H_\kappa\rightarrow H_\kappa$
as described in Section~\ref{section:basicsetup}.

The constructed tree $T$ will be a downward-closed subset of $^{<\kappa} \kappa$,
so that each level $T_\alpha$ will be a subset of $^\alpha \kappa$,
and the tree relation $\leq_T$ will simply be extension of sequences.
We will construct,
simultaneously by recursion over $\alpha < \kappa$,
the levels $\langle T_\alpha \mid \alpha < \kappa \rangle$ of the tree $T$,
as well as the functions $\langle f_{\alpha} \mid \alpha<\kappa \rangle$
and the nodes
$\langle\langle \mathbf b^\alpha_x \mid x \in T \restriction C_\alpha \rangle \mid \alpha \in \acc(\kappa)\rangle$,
so that after each stage $\alpha$ of the construction the following properties are satisfied:

\begin{enumerate}
\item \label{right-level}
$T_\alpha \subseteq {}^\alpha \kappa$;

\item \label{downward-closed}
The tree constructed so far is a \emph{downward-closed} subset of $^{\leq\alpha}\kappa$,
that is, for each $t \in T_\alpha$
we have $\{ t \restriction \beta \mid \beta < \alpha \} \subseteq T \restriction \alpha$.

\item \label{normal}
The tree is \emph{normal}, that is, for each $s \in T \restriction \alpha$,
there is $t \in T_\alpha$ with $s <_T t$;

\item \label{prolific}
The tree is \emph{prolific}, that is, for each $s \in T \restriction \alpha$,
we have
\[
\{ s^\smallfrown \langle i\rangle \mid i<\max \{\omega, \height(s) \} \} \subseteq T \restriction (\alpha+1).
\]

\item \label{slim}
The tree is \emph{slim}, that is, $\left| T_\alpha \right| \leq \max \{ \left|\alpha\right|, \aleph_0 \}$;

\item \label{b-simple}
If $\alpha$ is a limit ordinal, then
for every $x \in T \restriction C_\alpha$,
$\mathbf b^\alpha_x \in T_\alpha$ is the limit of the
increasing, continuous, cofinal sequence $b^\alpha_x$
in $( T \restriction \alpha, \subseteq )$,
satisfying the following properties:

\begin{enumerate}
\item
$\dom (b^\alpha_x) = C_\alpha \setminus \height(x)$;
\item
$b^\alpha_x (\height(x)) =x$;
\item
For all $\beta \in \dom(b^\alpha_x)$, $b^\alpha_x (\beta) \in T_\beta$;
\item
If $\beta_0 < \beta_1$ are two consecutive points in $\dom(b^\alpha_x)$,
then
\[
b^\alpha_x (\beta_1) =     \sealantichain ( b_x^\alpha(\beta_0), T\restriction (\beta_1+1) );
\]
\end{enumerate}

\item \label{ascent-is-function}\label{ascent-injection}
$f_\alpha : \omega \to T_\alpha$ is a function;
moreover, if $\alpha > 0$, then $f_\alpha$ is injective;

\item \label{cofinite-ascent}
For every $\beta < \alpha$,
\[
\{n < \omega \mid f_{\beta}(n) <_T f_{\alpha}(n) \} \in \mathcal F^{\fin}_{\aleph_0};
\]

\item \label{ascent-on-acc}
If $\alpha$ is a limit ordinal and $\beta \in \acc (C_\alpha)$ then
\[
\{ n < \omega \mid f_{\beta}(n) <_T f_{\alpha}(n) \} = \omega;
\]

\item \label{order2-nacc}
If $\alpha$ is  a limit ordinal such that $\sup (\acc(C_\alpha)) < \alpha$, then
for every $n < \omega$ there is some $x \in T \restriction C_\alpha$ such that
\[
f_\alpha (n) = \mathbf b^\alpha_x;
\]

\item \label{limit-ctbl-ascent}
If $\alpha$ is a limit ordinal, then
\[
T_\alpha = \{ \mathbf b^\alpha_x \mid x \in T \restriction C_\alpha \} \cup \{ f_\alpha (n) \mid n < \omega \}.
\]

\end{enumerate}

The following instance of the Coherence Claim Template from page \pageref{coherence-template}
gives a hint as to how we will ensure that the sequences described in property~(\ref{b-simple})
can always be constructed:

\begin{claim}\label{coherence-simple}
Fix limit ordinals ${\bar\alpha} < \alpha < \kappa$, and suppose $T \restriction \alpha$ has been constructed
to satisfy the above properties.
If $C_{\bar\alpha} = C_\alpha \cap {\bar\alpha}$,  $x \in T \restriction C_{\bar\alpha}$,
and $b^\alpha_x \restriction (C_{\bar\alpha} \setminus \height(x))$ has already been constructed,
then
\[
b^{\bar\alpha}_x = b^\alpha_x \restriction \left( C_{\bar\alpha} \setminus \height(x) \right).
\]
\end{claim}

\begin{proof}
Property~(\ref{b-simple})(d) says that,
for two consecutive points $\beta_0 < \beta_1$ in $\dom(b^\alpha_x)$,
the value of $b^\alpha_x({\bar\alpha}_1)$ depends only on:
\begin{itemize}
\item the function $\sealantichain$;
\item the value $b^\alpha_x (\beta_0)$; and
\item the tree $T \restriction (\beta_1+1)$.

\end{itemize}
In particular, there is no further dependency on the initial point $x$ or on $\alpha = \sup (C_\alpha)$.
This is the idea being captured by this claim.
Formally, we prove the claim by induction over the common domain:

On each side of the equation, the sequence has domain $C_{\bar\alpha} \setminus \height(x)$.
We prove that the two sequences $b^{\bar\alpha}_x$ and $b^\alpha_x$
have equal values on their common domain,
by induction over $\beta \in C_{\bar\alpha} \setminus \height(x)$:
\begin{description}
\item [$\blacktriangleright \beta = \height(x)$]
Using property~(\ref{b-simple})(b), we have
\[
b^{\bar\alpha}_x (\height(x)) = x = b^\alpha_x (\height(x)).
\]
\item[$\blacktriangleright \beta > \height(x), \beta \in \nacc({C_{\bar\alpha}})$]
Since $C_{\bar\alpha}$ is a club, $\beta$ must have an immediate predecessor in $C_{\bar\alpha} \setminus \height(x)$,
which we call $\beta^-$.
Then using property~(\ref{b-simple})(d) and the induction hypothesis, we have
\[
b^{\bar\alpha}_x (\beta) = \sealantichain ( b^{\bar\alpha}_x (\beta^-), T \restriction (\beta+1))
                = \sealantichain    ( b^\alpha_x (\beta^-), T \restriction (\beta+1) )
                = b^\alpha_x (\beta).
\]
\item[$\blacktriangleright \beta > \height(x), \beta \in \acc({C_{\bar\alpha}})$]
Then also $\beta \in \acc (C_\alpha)$,
and by continuity of each sequence and application of the induction hypothesis,
we have
\[
b^{\bar\alpha}_x(\beta)
    = \bigcup \{ b^{\bar\alpha}_x (\delta)  \mid \delta \in C_{\bar\alpha} \cap \beta \setminus \height(x) \}
    = \bigcup \{ b^\alpha_x (\delta) \mid \delta \in C_{\bar\alpha} \cap \beta \setminus \height(x) \}
    = b^\alpha_x (\beta).
\qedhere
\]
\end{description}
\end{proof}

The recursive construction proceeds as follows:

\begin{description}
\item[Base case, $\alpha = 0$]
Let
$T_0 = \{ \emptyset \}$,
and define $f_0 : \omega \to \{\emptyset\}$ by setting $f_0(n) = \emptyset$ for all $n < \omega$.
The required properties are automatically satisfied as there is nothing to check.

\item[Successor ordinal, $\alpha = \beta+1$]
Define
\[
T_\alpha = \{ s^\smallfrown \langle i \rangle\mid s \in T_\beta, i<\max \{\omega, \alpha\} \}.
\]
In addition, define $f_\alpha : \omega \to T_\alpha$ by setting, for $n < \omega$,
$f_\alpha(n) = f_\beta(n)^\smallfrown\langle n\rangle$.
The required properties are easy to verify.

\item[Limit level, $\alpha = \sup \alpha > 0$]
We begin by constructing $\mathbf b^\alpha_x \in {}^\alpha \kappa$ for each $x \in T \restriction C_\alpha$.

Recall that $C_\alpha$ is a club subset of $\alpha$.
For each $x \in T \restriction C_\alpha$,
we will use $C_\alpha$ to determine a cofinal branch through $(T \restriction \alpha,\s)$, containing $x$,
by defining an increasing, continuous sequence $b^\alpha_x$
of nodes.
The domain of the sequence $b^\alpha_x$ will be
${C_\alpha} \setminus \height(x)$.
Notice that $\dom(b^\alpha_x)$ is a club subset of $\alpha$, since ${C_\alpha}$ is club.
Also, we have
\begin{gather*}
\acc \left(\dom(b^\alpha_x)\right) = \acc ( {C_\alpha}) \setminus (\height(x)+1) ;  \\
\nacc \left(\dom(b^\alpha_x)\right) \setminus \{\height(x)\} = \nacc ( {C_\alpha} ) \setminus (\height(x)+1) .
\end{gather*}

We define the values $b^\alpha_x (\beta)$ of the sequence
by recursion over $\beta \in \dom(b^\alpha_x)$,
where for every
$\beta$ we will have $b^\alpha_x(\beta) \in T_\beta$:

\begin{description}
\item[$\blacktriangleright \beta  = \height(x)$]
This is where the sequence begins.
Let $b^\alpha_x(\height(x)) = x$.

\item[$\blacktriangleright \beta > \height(x),    \beta \in \nacc({C_\alpha})$]
In this case, we denote the predecessor of $\beta$ in ${C_\alpha}$ by $\beta^-$.
That is, we define $\beta^- = \max ( {C_\alpha} \cap \beta)$.
This maximum necessarily exists, and it is in $\dom(b^\alpha_x)$,
because ${C_\alpha}$ is club and $\beta \in \nacc({C_\alpha}) \setminus  (\height(x)+1)$.
Let
\[
b_x^\alpha(\beta)=\sealantichain ( b_x^\alpha(\beta^-), T\restriction (\beta+1) ).
\]

Since $b_x^\alpha(\beta^-)$ belongs to the normal tree $T\restriction (\beta+1)$,
we get from the Extension Lemma (page \pageref{extendfact}) that
$b_x^\alpha(\beta)$ is an element of $T_\beta$ extending $b_x^\alpha(\beta^-)$.

\item[$\blacktriangleright \beta > \height(x), \beta \in \acc({C_\alpha})$]
In this case
we define
\[
b^\alpha_x(\beta) = \bigcup
\left\{ b^\alpha_x(\gamma) \mid \gamma \in \dom(b^\alpha_x) \cap \beta \right\}.
\]

It is clear that $b^\alpha_x(\beta) \in {}^\beta \kappa$,
but we need to show
that in fact we have $b^\alpha_x(\beta) \in T_\beta$ as well.

From our choice of the sequence $\langle C_\alpha \mid \alpha < \kappa \rangle$ satisfying the proxy principle,
since $\alpha$ is a limit ordinal
and $\beta \in \acc({C_\alpha})$,
we must have
$C_\beta \sqsubseteq C_\alpha$, so that $C_\beta = C_\alpha \cap \beta$.
Then, applying Claim~\ref{coherence-simple},
we have
\[
b^\alpha_x(\beta)
    = \bigcup \{ b^\alpha_x(\gamma) \mid \gamma \in \dom(b^\alpha_x) \cap \beta \}
    = \bigcup \{ b^\beta_x(\gamma) \mid \gamma \in \dom(b^\beta_x) \}
    = \mathbf b^\beta_x.
\]
Since $\beta < \alpha$, by induction hypothesis the level $T_\beta$ has already been constructed, and
the construction guarantees that we have included
the limit $\mathbf b^\beta_{x}$ of the sequence $b^\beta_{x}$ into $T_\beta$.
But we have just shown that this is exactly $b^\alpha_x (\beta)$,
so that $b^\alpha_x(\beta) \in T_\beta$, as required.

\end{description}

Having defined $b^\alpha_x(\beta)$ for all $\beta$ in the required domain,
it is clear that the sequence $b^\alpha_x$ defines a cofinal branch through $(T \restriction \alpha,\s)$,
since ${C_\alpha}$ is club in $\alpha$, and that this branch contains $x$.
We now let
\[
\mathbf b^\alpha_x = \bigcup \left\{ b^\alpha_x (\beta) \mid \beta \in \dom(b^\alpha_x) \right\},
\]
and it is clear that $\mathbf b^\alpha_x \in {}^{\alpha}\kappa$.

Next, we fix $n < \omega$, and we must prescribe a function value $f_\alpha (n) \in {}^\alpha \kappa$.

\begin{claim} \label{full-domination-acc}
The sequence $\langle f_\beta(n) \mid \beta \in \acc(C_\alpha) \rangle$ is increasing
in $\left( T \restriction \alpha, \subseteq\right)$.
\end{claim}
\begin{proof}
Consider any $\beta_1, \beta_2 \in \acc(C_\alpha)$ with $\beta_1 < \beta_2$.
Since $\beta_2 \in \acc(C_\alpha)$, our proxy principle gives $C_{\beta_2} = C_\alpha \cap \beta_2$
(since we have used $\sqsubseteq$ as the second parameter).
Since $\beta_1 < \beta_2$ and $\beta_1 \in \acc(C_\alpha)$, we must then have $\beta_1 \in \acc (C_{\beta_2})$.
By property~(\ref{ascent-on-acc}) of the induction hypothesis applied to $\beta_2$,
it follows that $f_{\beta_1}(n) <_T f_{\beta_2}(n)$, as required.
\end{proof}

Now let
\[
\alpha_0 = \sup \left( \acc(C_\alpha) \cup \{ \min \left( C_\alpha \setminus \{ 0 \} \right) \} \right).
\]
It is clear from the definition that $0 < \alpha_0 \leq \alpha$,
and that $\alpha_0 = \sup(\acc(C_\alpha))$ iff $\acc(C_\alpha) \neq\emptyset$.
The definition of $f_\alpha(n)$ splits into two possibilities:

\begin{description}
\item [$\blacktriangleright \alpha_0 = \alpha$]
In particular, $\sup(\acc(C_\alpha)) = \alpha$.
By Claim~\ref{full-domination-acc},
the sequence $\langle f_\beta(n) \mid \beta \in \acc(C_\alpha) \rangle$ is
increasing, and in this case it is cofinal in $(T \restriction \alpha, \subseteq)$.
Let
\[
f_\alpha (n) = \bigcup\{ f_\beta (n)\mid \beta \in \acc (C_\alpha)\}.
\]
It is clear that $f_\alpha (n) \in {}^\alpha \kappa$.

\item [$\blacktriangleright 0 < \alpha_0 < \alpha$]
From the definition of $\alpha_0$,
any point in $C_\alpha \setminus (\alpha_0+1)$ must be in $\nacc (C_\alpha)$.
Since $C_\alpha$ is club in $\alpha > \alpha_0$,
it follows that
$C_\alpha \setminus \alpha_0$ is an $\omega$-type cofinal subset of $\alpha$.
Enumerate $C_\alpha  \setminus \alpha_0$ as an increasing sequence
$\langle \alpha_m \mid m < \omega \rangle$ cofinal in $\alpha$,
and let
\[
m_n = \max \left\{ m \leq n  \mid \langle f_{\alpha_i} (n) \mid i \leq m \rangle \text{ is $<_T$-increasing} \right\}.
\]
The maximum necessarily exists, because the set is nonempty
 (as $m=0$ satisfies the defining condition vacuously) and finite.
Then
let
\[
f_\alpha(n) = \mathbf b^\alpha_{f_{\alpha_{m_n}} (n)}.
\]

By $\alpha_{m_n} \in C_\alpha$,
and property~(\ref{ascent-is-function}) of the induction hypothesis (applied to the level $\alpha_{m_n}$),
we have
\[
f_{\alpha_{m_n}} (n) \in T_{\alpha_{m_n}} \subseteq T \restriction C_\alpha,
\]
and hence
$f_\alpha(n) \in {}^\alpha \kappa$ is well-defined.

\end{description}

Having constructed $f_\alpha(n)$, we now claim:

\begin{claim}\label{a0-goes-right}
If $\alpha_0 < \alpha$ then
$f_{\alpha_0} (n) <_T f_\alpha(n)$.
\end{claim}

\begin{proof}
Referring back to the construction of $f_\alpha(n)$,
we have
\begin{align*}
f_{\alpha_0} (n)
        &\leq_T f_{\alpha_{m_n}}(n)
            &&\text{by choice of $m_n$}             \\
        &<_T \mathbf b^\alpha_{f_{\alpha_{m_n}} (n)}  = f_\alpha (n)
            &&\text{by construction},
\end{align*}
as required.
\end{proof}

Finally, as promised, we set
\[
T_\alpha = \left\{ \mathbf b^\alpha_x \mid  x \in T \restriction C_\alpha \right\} \cup \left\{f_\alpha (n) \mid n < \omega \right\}.
\]

To verify some of the required properties:

\begin{itemize}
\item [(\ref{right-level})]
Each sequence $b^\alpha_x$ defined a cofinal branch through $(T \restriction \alpha,\s)$, so that
its limit $\mathbf b^\alpha_x \in {}^\alpha \kappa$.

In case $\alpha_0 < \alpha$, we have (for each $n<\omega$)
$f_\alpha(n) = \mathbf b^\alpha_{f_{\alpha_{m_n}} (n)} \in {}^\alpha \kappa$.

If $\alpha_0 = \alpha$, then $f_\alpha(n)$ is also the limit of a cofinal branch in $(T \restriction \alpha,\s)$,
so it is in ${}^\alpha \kappa$.

\item [(\ref{normal})]
For every $s \in T \restriction \alpha$, since $C_\alpha$ is club in $\alpha$, we can find some $\delta \in C_\alpha$
such that $\height(s) < \delta$.
Then, applying the induction hypothesis at the level $\delta$, there is some $x \in T_\delta$ such that $s <_T x$.
Since $x \in T \restriction C_\alpha$,
we have constructed a branch $b^\alpha_x$ through $x$ and
placed its limit $\mathbf b^\alpha_x$ into $T_\alpha$.
We then have $s <_T x <_T \mathbf b^\alpha_x$,
so that this property is satisfied.

\item [(\ref{slim})]
Since $\alpha$ is a limit ordinal, we have $\left| \alpha \right| \geq \aleph_0$.
Applying the induction hypothesis, for each $\beta < \alpha$ we have
$\left|T_{\beta}\right| \leq \max \{ \left|\beta\right|, \aleph_0 \} \leq \left|\alpha\right|$.
Thus
\[
\left|T \restriction C_\alpha \right|
    \leq \left|T \restriction \alpha \right|
    = \sum_{\beta < \alpha} \left|T_\beta\right|
    \leq \left|\alpha\right| \cdot \left|\alpha\right|
    = \left|\alpha\right|.
\]
Since every node of the form $\mathbf b^\alpha_x$
is produced from some node $x \in T \restriction C_\alpha$,
and every node of the form $f_\alpha (n)$ comes from some $n < \omega$,
it follows that
\[
\left|T_\alpha \right| \leq
    \left|T \restriction C_\alpha \right| + \aleph_0
    \leq \left|\alpha\right| + \aleph_0
    = \left|\alpha\right|,
\]
as required.

\item [(\ref{b-simple})]
This condition essentially summarizes how our construction of $\mathbf b^\alpha_x$ was carried out.

\item [(\ref{ascent-on-acc})]
Fix $\beta \in \acc (C_\alpha)$ and $n<\omega$.
We must show that $f_\beta (n) <_T f_\alpha (n)$.
Referring back to the construction of $f_\alpha(n)$, there are two cases to check:

\begin{description}
\item[$\blacktriangleright \alpha_0 = \alpha$]
In this case, $f_\alpha (n)$ was constructed to be above $f_\beta(n)$.

\item[$\blacktriangleright \alpha_0 < \alpha$]
Since $\beta \in \acc(C_\alpha)$, in particular $\acc(C_\alpha) \neq \emptyset$, so that
$\alpha_0 = \sup(\acc(C_\alpha))$,
and it follows that $\beta \leq \alpha_0$ and (since $C_\alpha$ is club in $\alpha > \alpha_0$) $\alpha_0 \in \acc (C_\alpha)$.
We then have
\begin{align*}
f_\beta(n)  &\leq_T f_{\alpha_0} (n)
            &&\text{from Claim~\ref{full-domination-acc}}   \\
    &<_T f_\alpha (n)
            &&\text{from Claim~\ref{a0-goes-right}},
\end{align*}
as required.

\end{description}

\item [(\ref{ascent-injection})]
Fix $n_1 < n_2 < \omega$, and we must show that $f_\alpha(n_1) \neq f_\alpha (n_2)$.
Again,
there are two cases to check:

\begin{description}
\item[$\blacktriangleright \alpha_0 = \alpha$]
Since $\sup(\acc(C_\alpha)) = \alpha$,
we find $\beta \in \acc(C_\alpha)$ such that $0 < \beta < \alpha$.
Applying the induction hypothesis to $\beta$,
we get $f_\beta(n_1) \neq f_\beta (n_2)$.
Since $\beta \in \acc(C_\alpha)$,
property~(\ref{ascent-on-acc}) gives us $f_{\beta} (n_1) <_T f_\alpha (n_1)$
and $f_{\beta} (n_2) <_T f_\alpha (n_2)$,
and the result follows.

\item[$\blacktriangleright \alpha_0 < \alpha$]
Applying the induction hypothesis to $\alpha_0$,
we get $f_{\alpha_0}(n_1) \neq f_{\alpha_0} (n_2)$.
Then, Claim~\ref{a0-goes-right}
gives us $f_{\alpha_0} (n_1) <_T f_\alpha (n_1)$
and $f_{\alpha_0} (n_2) <_T f_\alpha (n_2)$,
and the result follows.

\end{description}

\item [(\ref{cofinite-ascent})]
Fix $\beta < \alpha$.
Referring back to the construction of $f_\alpha(n)$, there are two cases to check:

\begin{description}
\item[$\blacktriangleright \alpha_0 = \alpha$]
Since $\sup(\acc(C_\alpha)) = \alpha$,
we find $\beta' \in \acc(C_\alpha)$ such that $\beta < \beta' < \alpha$.
Let
\[
F = \{ n < \omega \mid f_\beta (n) <_T f_{\beta'} (n) \},
\]
so that applying the induction hypothesis to $\beta'$ gives us $F \in \mathcal F^{\fin}_{\aleph_0}$.
Since $\beta' \in \acc(C_\alpha)$,
property~(\ref{ascent-on-acc}) gives us $f_{\beta'} (n) <_T f_\alpha (n)$ for all $n< \omega$.
In particular, for $n \in F$, we have
\[
f_\beta (n) <_T f_{\beta'} (n) <_T f_\alpha (n),
\]
as required.

\item[$\blacktriangleright \alpha_0 < \alpha$]
In this case we have identified a sequence $\langle \alpha_m \mid m < \omega \rangle$ cofinal in $\alpha$,
so we fix some $m < \omega$ such that $\beta < \alpha_m$.
For each natural number $i < m$, let
\[
F_i = \{ n < \omega \mid f_{\alpha_i}(n) <_T f_{\alpha_m}(n) \}.
\]
Also let
\[
G = \{ n < \omega \mid f_{\beta}(n) <_T f_{\alpha_m}(n) \} \text{ and } H = \{ n < \omega \mid n > m \}.
\]
Applying the induction hypothesis to $\alpha_m$, we have $F_i \in \mathcal F^{\fin}_{\aleph_0}$ for all $i < m$,
as well as $G, H \in \mathcal F^{\fin}_{\aleph_0}$.
Define
\[
F = G \cap H \cap \bigcap_{i < m} F_i .
\]
Clearly, $F \in \mathcal F^{\fin}_{\aleph_0}$ as it is the intersection of finitely many sets from that filter.
Now, fix any $n \in F$, and we will show that $f_\beta (n) <_T f_\alpha (n)$:
For every $i<m$ we have $n \in F \subseteq F_i$, so that
$f_{\alpha_i} (n) <_T f_{\alpha_m} (n)$.
As $\pred{f_{\alpha_m}(n)}$ is  a linear order and $f_{\alpha_i}(n) \in T_{\alpha_i}$ for each $i$,
it follows that $\langle f_{\alpha_i}(n) \mid  i \leq m\rangle$ is $<_T$-increasing.
Furthermore, since $n \in H$,
we have $n > m$.
Thus $m$ satisfies the defining properties for membership in the set whose maximum is $m_n$, so
it follows that $m \leq m_n$.
In particular,
we then have
$f_{\alpha_m} (n) \leq_T f_{\alpha_{m_n}} (n)$.
We also have $n \in G$, so that $f_\beta (n) <_T f_{\alpha_m} (n)$.
Putting everything together, we have
\[
f_\beta (n) <_T f_{\alpha_m} (n) \leq_T f_{\alpha_{m_n}} (n)
<_T \mathbf b^\alpha_{f_{\alpha_{m_n}} (n)} = f_\alpha (n),
\]
as required.

\end{description}

\item [(\ref{order2-nacc})]
In this case,
$\alpha_0 < \alpha$,
so that every $f_\alpha (n)$ was defined to be equal to some $\mathbf b^\alpha_{f_{\alpha_{m_n}} (n)}$.

\end{itemize}

\end{description}

Now we let
\[
T = \bigcup_{\alpha < \kappa} T_\alpha.
\]
Having built the tree,
we now claim:

\begin{claim}\label{Souslin-claim-ctbl-ascent}
The tree $(T,{\subset})$ is a $\kappa$-Souslin tree.
\end{claim}

\begin{proof} It is clear that $(T,{\subset})$ is a $\kappa$-tree.
Let $A \subseteq T$ be a maximal antichain.
We will show that $\left|A\right| < \kappa$,
by showing that $A \subseteq T \restriction \alpha$ for some $\alpha < \kappa$.

\begin{subclaim}\label{sc621}
For every $i<\kappa$, the following set is stationary:
\[
A_i = \{\beta\in R_i\mid  A\cap(T\restriction\beta)=S_\beta\text{ is a maximal antichain in }T\restriction\beta \}.
\]
\end{subclaim}

\begin{proof}
Let $i<\kappa$ be an arbitrary ordinal, and $D\s\kappa$ be an arbitrary club.
We must show that $D \cap A_i \neq \emptyset$.
Put $p=\{A,T,D\}$.
Using the fact that the sequences $\langle S_\beta \mid \beta<\kappa\rangle$ and $\langle R_i  \mid i < \kappa \rangle$
satisfy $\diamondsuit(H_\kappa)$,
pick $\mathcal M\prec H_{\kappa^{+}}$
with $p\in\mathcal M$ such that $\beta=\mathcal M\cap\kappa$ is in $R_i$, and $S_\beta=A\cap\mathcal M$.
Since $D\in\mathcal M$ and $D$ is club in $\kappa$, we have $\beta\in D$. We claim that $\beta\in A_i$.

For all $\alpha<\beta$, by $\alpha,T\in \mathcal M$, we have $T_\alpha\in \mathcal M$,
and by $\mathcal M\models |T_\alpha|<\kappa$,
we have $T_\alpha\s \mathcal M$. So $T\restriction\beta\s \mathcal M$.
As $\dom(z)\in \mathcal M$ for all $z \in T\cap \mathcal M$,
we conclude that $T\cap \mathcal M=T\restriction\beta$.
So, $S_\beta=A\cap(T\restriction\beta)$.
As $H_{\kappa^{+}}\models A\text{ is a maximal antichain in }T$ and $T\cap \mathcal M=T\restriction \beta$,
we get that $A\cap(T\restriction\beta)$ is maximal in $T\restriction\beta$.
\end{proof}

In particular, the set $A_0$ is a cofinal subset of $\kappa$,
so we apply the last part of the proxy principle to obtain a limit ordinal $\alpha < \kappa$ such that
\[
\sup \{ \beta \in C_\alpha  \mid \suc_\omega (C_\alpha \setminus \beta) \subseteq A_0 \} = \alpha.
\]

To see that $A \subseteq T \restriction \alpha$,
consider any $v \in T \restriction (\kappa \setminus \alpha)$, and we will show that $v \notin A$.
We have $\height(v) \geq \alpha$.
Thus we can let
$v' = v \restriction \alpha \in T_\alpha$,
so that $v' \leq_T v$.

\begin{subclaim}\label{get-b-ctbl}
There are some ${\bar\alpha} \in \acc(C_\alpha) \cup \{ \alpha \}$ and ${x} \in T \restriction C_{{\bar\alpha}}$ such that
\[
\mathbf b^{{{\bar\alpha}}}_{x} \leq_T v' \text{ and }
\sup \left(\nacc (C_{{\bar\alpha}}) \cap A_0 \right) = {\bar\alpha}.
\]
\end{subclaim}

\begin{proof}
Recall that $v' \in T_\alpha$.
Since $\alpha$ is a limit ordinal,
by property~(\ref{limit-ctbl-ascent}) there are now two possibilities to consider:
\begin{description}
\item[$\blacktriangleright v' = \mathbf b^\alpha_{x}$ for some $x \in T \restriction C_\alpha$]
In this case, fix such an $x$, and let ${\bar\alpha} = \alpha$,
and the subclaim is satisfied.

\item[$\blacktriangleright v' = f_\alpha (n)$ for some $n < \omega$]
Fix such an $n$.
By our choice of $\alpha$, we can choose $\epsilon \in C_\alpha$ such that
$\suc_\omega (C_\alpha \setminus \epsilon) \subseteq A_0$.
Define
\[
{\bar\alpha} = \sup \left( \suc_\omega (C_\alpha \setminus \epsilon) \right).
\]
It is clear that ${\bar\alpha}$ is a limit ordinal, $\epsilon < {\bar\alpha} \leq \alpha$,
and ${\bar\alpha} \in \acc (C_\alpha) \cup \{ \alpha \}$.
Thus by property~(\ref{ascent-on-acc}) we have $f_{{\bar\alpha}}(n) \leq_T f_\alpha (n)$.
Since ${\bar\alpha} \in \acc (C_\alpha) \cup \{ \alpha \}$,
the proxy principle gives us $C_{{\bar\alpha}} = C_\alpha \cap {\bar\alpha}$,
so that
\[
\sup(\acc(C_{{\bar\alpha}})) = \sup(\acc(C_\alpha \cap {\bar\alpha})) \leq \epsilon < {\bar\alpha}.
\]
Thus by applying property~(\ref{order2-nacc}) to ${\bar\alpha}$,
we must have $f_{{\bar\alpha}}(n) = \mathbf b^{{{\bar\alpha}}}_{x}$ for some $x \in T \restriction C_{{\bar\alpha}}$.
Fix such an $x$.
It follows that
\[
\mathbf b^{{{\bar\alpha}}}_{x} = f_{{\bar\alpha}}(n) \leq_T f_\alpha (n) = v'.
\]
Notice that
\[
{\bar\alpha} = \sup \left( \suc_\omega (C_\alpha \setminus \epsilon) \cap A_0 \right)
    = \sup \left( \nacc (C_{{\bar\alpha}}) \cap A_0 \right),
\]
giving the required conclusion.
\qedhere
\end{description}
\end{proof}

We now fix ${\bar\alpha}$ and $x$ as in Subclaim~\ref{get-b-ctbl}.

\begin{subclaim} \label{point-in-antichain}
There is some $y \in A$ such that $y <_T \mathbf b^{{{\bar\alpha}}}_{x}$.
\end{subclaim}

\begin{proof}
Fix $\beta \in \nacc(C_{{\bar\alpha}}) \cap A_0$ with $\height(x) < \beta < {\bar\alpha}$.

Of course, $\beta \in \dom (b^{{{\bar\alpha}}}_{x})$, and by construction of $\mathbf b^{{{\bar\alpha}}}_{x}$,
we know that $b^{{{\bar\alpha}}}_{x}(\beta) <_T \mathbf b^{{{\bar\alpha}}}_{x}$.

Since $\beta \in A_0$,
we know that $S_\beta = A \cap (T \restriction \beta)$ is a maximal antichain in $T \restriction \beta$.
Since $\beta \in \nacc( {C_{{\bar\alpha}}}) \setminus (\height(x) +1)$,
we refer back to the construction of $b^{{{\bar\alpha}}}_{x}(\beta)$.
We have
\[
b_{x}^{{{\bar\alpha}}}(\beta)=\sealantichain ( b_{x}^{{{\bar\alpha}}}(\beta^-), T\restriction (\beta+1) ).
\]

It is clear that $\beta(T\restriction (\beta+1))=\beta$.
Since $T\restriction (\beta+1)$ is normal, and since $S_\beta$ is a maximal antichain in $T \restriction \beta$,
the set $Q=\{ z\in T\cap{}^{\beta}\kappa\mid \exists y\in S_{\beta}( b_{x}^{{{\bar\alpha}}}(\beta^-) \cup y\s z)\}$
is non-empty,
so $b_{x}^{{{\bar\alpha}}}(\beta)=\min(Q,<_\phi)$.
Pick $y\in S_\beta$ such that $y\s b_{x}^{{{\bar\alpha}}} (\beta)$.
Then $y\in S_\beta=A \cap (T \restriction \beta)$, and
\[
y <_T b^{{{\bar\alpha}}}_{x}(\beta) <_T \mathbf b^{{{\bar\alpha}}}_{x},
\]
as required.
\end{proof}

Altogether, we have
\[
y <_T \mathbf b^{{{\bar\alpha}}}_{x} \leq_T  v' \leq_T v.
\]
Since $y$ is an element of the antichain $A$,
the fact that $v$ extends $y$ implies that $v \notin A$.
Since $v \in T \restriction (\kappa \setminus \alpha)$ was arbitrary,
we have shown that $A \subseteq T \restriction \alpha$.

To see that $\left|A\right| < \kappa$:
For each $\beta < \alpha$ we have $\left|T_{\beta}\right|<\kappa$.
Since $A \subseteq T \restriction \alpha$ and
$\alpha < \kappa = \cf(\kappa)$, it follows that
we have
\[
\left| A \right| \leq \left|T \restriction \alpha \right|
                    = \sum_{\beta < \alpha} \left|T_\beta\right| < \kappa,
\]
as required.

Since $A$ was an arbitrary maximal antichain in $(T,{\subset})$,
we infer that our tree has no antichains of size $\kappa$.

Any splitting tree with no antichains of size $\kappa$ also has no chains of size $\kappa$.
This completes the proof that $(T,{\subset})$ is a $\kappa$-Souslin tree.
\end{proof}

Property~(\ref{prolific}) guarantees that $T$ is prolific,
while property~(\ref{slim}) guarantees that $T$ is slim.

Properties~(\ref{ascent-injection}) and~(\ref{cofinite-ascent}) guarantee that
$\langle f_\alpha \mid \alpha < \kappa \rangle$ is an injective $\mathcal F^{\fin}_{\aleph_0}$-ascent path through $T$.
\end{proof}

\begin{theorem} \label{tree-ctbl-ascent}
Suppose that  $\kappa$ is a regular uncountable cardinal,
$U \subseteq {}^{<\kappa} \kappa$ is a slim downward-closed $\kappa$-tree,
and $\p_{14}(\kappa,  {\sqsubseteq}, 1, \{\kappa\})$ holds.

Then there exists a prolific slim $\kappa$-Souslin tree with an injective $(\mathcal F^{\fin}_{\aleph_0},U)$-ascent path.
\end{theorem}
\begin{proof}
We assume that the reader is comfortable with the proof of Theorem~\ref{ctbl-ascent-thm}.

By passing to an isomorphic slim tree if necessary,
we may assume that for all $u \in U$ and all $\beta < \dom(u)$,
$u(\beta) < \left| U_{\beta+1} \right|$.

We will construct, simultaneously by recursion over $\alpha < \kappa$,
the levels $\langle T_\alpha \mid \alpha < \kappa \rangle$ of the tree $T$
as well as the functions $\langle f_u \mid u \in U \rangle$
and the nodes $\langle\langle \mathbf b^\alpha_x \mid x \in T \restriction C_\alpha \rangle\mid \alpha \in \acc(\kappa)\rangle$
so that after each stage $\alpha$,
properties~(1)--(\ref{b-simple}) of the construction in Theorem~\ref{ctbl-ascent-thm} are satisfied,
in addition to the following:

\begin{enumerate}
\setcounterref{enumi}{ascent-is-function}
\addtocounter{enumi}{-1}
\item \label{tree-ascent-is-function}\label{tree-ascent-injection}
For every $u \in U_\alpha$,
$f_u : \omega \to T_\alpha$ is a function;
moreover, if $\alpha > 0$, then $f_u$ is injective;

\item \label{tree-cofinite-ascent}
For every $u \in U_\alpha$ and every
 $\beta < \alpha$,
\[
\{n < \omega \mid f_{u \restriction \beta}(n) <_T f_{u}(n) \} \in \mathcal F^{\fin}_{\aleph_0};
\]

\item \label{tree-ascent-on-acc}
If $\alpha$ is a limit ordinal and $\beta \in \acc (C_\alpha)$, then for every $u \in U_\alpha$,
\[
\{ n < \omega \mid f_{u \restriction \beta}(n) <_T f_{u}(n) \} = \omega;
\]

\item \label{tree-order2-nacc}
If $\alpha$ is  a limit ordinal such that $\sup (\acc(C_\alpha)) < \alpha$, then
for every $n < \omega$ and every $u \in U_\alpha$ there is some $x \in T \restriction C_\alpha$ such that
\[
f_u (n) = \mathbf b^\alpha_x;
\]

\item \label{tree-limit-ctbl-ascent}
If $\alpha$ is a limit ordinal, then
\[
T_\alpha = \{ \mathbf b^\alpha_x \mid x \in T \restriction C_\alpha \} \cup
    \{ f_u (n) \mid u \in U_\alpha, n < \omega\};
\]

\item \label{distinct-ascent}
For any two distinct nodes $u,v$ from $U_\alpha$,
\[
\{ n < \omega \mid f_{u}(n) \neq f_{v}(n) \} \in \mathcal F^{\fin}_{\aleph_0}.
\]
\end{enumerate}

We have the following instance of the Coherence Claim Template from page \pageref{coherence-template}:

\begin{claim}
Fix limit ordinals ${\bar\alpha} < \alpha < \kappa$, and suppose $T \restriction \alpha$ has been constructed
to satisfy the above properties.
If $C_{\bar\alpha} = C_\alpha \cap {\bar\alpha}$,  $x \in T \restriction C_{\bar\alpha}$,
and $b^\alpha_x \restriction (C_{\bar\alpha} \setminus \height(x))$ has already been constructed,
then
\[
b^{\bar\alpha}_x = b^\alpha_x \restriction \left( C_{\bar\alpha} \setminus \height(x) \right).
\]
\end{claim}
\begin{proof}
Follows from property~(\ref{b-simple}),
as in the proof of Claim~\ref{coherence-simple}.
\end{proof}

The recursive construction proceeds as follows:

\begin{description}
\item[Base case, $\alpha = 0$]
As always, let
$T_0 = \{ \emptyset \}$.
The only element of $U_0$ is the root $\emptyset$,
and we define $f_\emptyset : \omega \to \{\emptyset\}$ by setting
$f_\emptyset(n) = \emptyset$ for all $n < \omega$.
The required properties are automatically satisfied as there is nothing to check.

\item[Successor ordinal, $\alpha = \beta+1$]
In this case, define
\[
T_\alpha = \{ t^\smallfrown \langle i \rangle \mid t \in T_\beta, i<\max \{\omega \cdot \left| U_\alpha \right|, \alpha\} \}.
\]

Then, for every $u \in U_\alpha$,
we define
$f_u : \omega \to T_\alpha$ by setting, for all $n < \omega$,
\[
f_u (n) = f_{u \restriction \beta} (n) ^\smallfrown \langle \omega \cdot (u(\beta)) + n \rangle.
\]
The required properties are easy to verify (for~(\ref{slim}) we use the fact that $U$ is slim).

\item[Limit level, $\alpha = \sup \alpha > 0$]
We construct $\mathbf b^\alpha_x \in {}^\alpha \kappa$ for each $x \in T \restriction C_\alpha$,
just as in Theorem~\ref{ctbl-ascent-thm}.
Then, we fix $u \in U_\alpha$, and we must construct the function $f_u : \omega \to {}^\alpha \kappa$.

\begin{claim} \label{full-domination-acc-tree} For all $n<\omega$,
the sequence $\langle f_{u \restriction \beta} (n) \mid \beta \in \acc(C_\alpha) \rangle$ is
increasing in $\left(T \restriction \alpha, \subseteq\right)$.
\end{claim}
\begin{proof}
Follows from property~(\ref{tree-ascent-on-acc}),
as in the proof of Claim~\ref{full-domination-acc}.
\end{proof}

Set
\[
\alpha_0 = \sup \left( \acc(C_\alpha) \cup \{ \min \left( C_\alpha \setminus \{ 0 \} \right) \} \right).
\]

Let $n < \omega$ be arbitrary. We shall prescribe a function value $f_u (n) \in {}^\alpha \kappa$,
by considering two possibilities:

\begin{description}
\item [$\blacktriangleright \alpha_0 = \alpha$]
By Claim~\ref{full-domination-acc-tree},
the sequence $\langle f_{u \restriction \beta}(n) \mid \beta \in \acc(C_\alpha) \rangle$ is increasing,
and in this case it is cofinal in $(T \restriction \alpha, \subseteq)$.
Let
\[
f_u (n) = \bigcup\{ f_{u \restriction \beta} (n)\mid \beta \in \acc (C_\alpha)\}.
\]
It is clear that $f_u (n) \in {}^\alpha \kappa$.

\item [$\blacktriangleright 0 < \alpha_0 < \alpha$]
Enumerate $C_\alpha  \setminus \alpha_0$ as an increasing sequence
$\langle \alpha_m \mid m < \omega \rangle$ cofinal in $\alpha$, and
let
\[
m_n = \max \left\{ m \leq n  \mid \langle f_{u \restriction \alpha_i} (n) \mid i \leq m \rangle \text{ is $<_T$-increasing} \right\}.
\]
Then,
define
\[
f_{u} (n) = \mathbf b^\alpha_{f_{u \restriction \alpha_{m_n}} (n)}.
\]
As $ f_{u \restriction \alpha_{m_n}} (n) \in  T \restriction C_\alpha$,
we have that $f_u (n)$ is a well-defined element of ${}^\alpha \kappa$.

\end{description}

Having constructed $f_u (n)$, we now have the following variant of Claim~\ref{a0-goes-right}:

\begin{claim}\label{a0-goes-right-tree}
If $\alpha_0 < \alpha$ then
$f_{u \restriction \alpha_0} (n) <_T f_u (n)$.
\end{claim}

\begin{proof}
Referring back to the construction of $f_u (n)$,
we have
\begin{align*}
f_{u \restriction \alpha_0} (n)
        &\leq_T f_{u \restriction \alpha_{m_n}}(n)
            &&\text{by choice of $m_n$}             \\
        &<_T \mathbf b^\alpha_{f_{u \restriction \alpha_{m_n}} (n)}  = f_u (n)
            &&\text{by construction},
\end{align*}
as required.
\end{proof}

Finally, as promised, we set
\[
T_\alpha = \{ \mathbf b^\alpha_x \mid  x \in T \restriction C_\alpha \} \cup
\{ f_u (n) \mid u \in U_\alpha, n < \omega \}.
\]

The required properties are verified just as in the
proof of Theorem~\ref{ctbl-ascent-thm},
replacing for $\beta < \alpha$, the pair $(f_\alpha,f_\beta)$ with $(f_u,f_{u \restriction \beta})$,
modulo the following exceptions:

\begin{itemize}

\item [(\ref{slim})]
As in the proof of Theorem~\ref{ctbl-ascent-thm}, we know that
$\left|T \restriction C_\alpha \right| \leq \left|\alpha\right|$.

Since $U$ is slim and $\left|\alpha\right| \geq \aleph_0$, we have $\left|U_\alpha\right| \leq \left|\alpha\right|$.
Since every node of the form $\mathbf b^\alpha_x$
is produced from some node $x \in T \restriction C_\alpha$,
and every node of the form $f_u (n)$ comes from some $u \in U_\alpha$ and some $n < \omega$,
it follows that
\[
\left|T_\alpha \right| \leq
    \left|T \restriction C_\alpha \right| + \left|U_\alpha\right| \cdot \aleph_0
    \leq \left|\alpha\right| + \left|\alpha\right| \cdot \aleph_0
    = \left|\alpha\right| + \left|\alpha\right|
    = \left|\alpha\right|,
\]
as required.

\item [(\ref{distinct-ascent})]
Fix two distinct points $u,v \in U_\alpha$.
Since $\alpha$ is a limit ordinal and $u, v \in {}^\alpha \kappa$,
we find some $\beta < \alpha$
such that $u \restriction \beta \neq v \restriction \beta$.
Define the following sets:
\begin{align*}
F_1 &= \{ n < \omega \mid f_{u \restriction \beta}(n) \neq f_{v \restriction \beta}(n) \}    \\
F_2 &= \{ n < \omega \mid f_{u \restriction \beta}(n) <_T f_{u}(n) \}    \\
F_3 &= \{ n < \omega \mid f_{v \restriction \beta}(n) <_T f_{v}(n) \}
\end{align*}
Applying the induction hypothesis to $\beta$, we have $F_1 \in \mathcal F^{\fin}_{\aleph_0}$.
Property~(\ref{tree-cofinite-ascent}) gives $F_2, F_3 \in \mathcal F^{\fin}_{\aleph_0}$.
Define $F = F_1 \cap F_2 \cap F_3$.
Clearly, $F \in \mathcal F^{\fin}_{\aleph_0}$, as it is the intersection of finitely many sets from the filter.
Consider any $n \in F$.
Since $n \in F_1$,
$f_{u\restriction \beta}(n)$ and $f_{v \restriction \beta}(n)$ are
distinct (thus incompatible) elements of $T_\beta$.
Since $n \in F_2 \cap F_3$, we have
$f_{u \restriction \beta}(n) <_T f_u(n)$ and
$f_{v \restriction \beta}(n) <_T f_v(n)$.
It follows that $f_u(n) \neq f_v(n)$, as required.
\end{itemize}
\end{description}

\begin{claim}\label{Souslin-claim-tree-ctbl-ascent}
The tree $T = \bigcup_{\alpha < \kappa} T_\alpha$ is a $\kappa$-Souslin tree.
\end{claim}

\begin{proof}
Just as in Claim~\ref{Souslin-claim-ctbl-ascent},
replacing, in the proof of Subclaim \ref{get-b-ctbl}, the pair $(f_\alpha,f_{\bar\alpha})$ with $(f_u,f_{u \restriction\bar\alpha})$.
\end{proof}

Properties~(\ref{tree-ascent-injection}),
(\ref{tree-cofinite-ascent}), and~(\ref{distinct-ascent})
guarantee that $\langle f_u \mid u \in U \rangle$ is
an injective $(\mathcal F^{\fin}_{\aleph_0}, U)$-ascent path
through $T$.
\end{proof}

\subsection{Complete Trees}\label{subsectioncomplete}
By strengthening the fourth parameter of the principle $\p^-_{14}$, from $\{\kappa\}$ to $\{E^\kappa_{\geq\chi}\}$,
for some cardinal $\chi < \kappa$,
we can ensure that the ordinal $\alpha$ on which the ``hitting'' action takes place
in the course of proving that the tree is Souslin
has large cofinality.
Thus, the careful limitation determining which nodes of limit height $\alpha$ are placed into the tree
(as given by property~(\ref{limit-ctbl-ascent}) of the recursive construction in Theorems~\ref{ctbl-ascent-thm}
and~\ref{tree-ctbl-ascent})
needs to be observed only for ordinals $\alpha \in E^\kappa_{\geq\chi}$.
This gives us the flexibility to add as many nodes as we like at any height $\alpha$ of small cofinality,
subject only to the constraint that the tree remain a $\kappa$-tree, that is, $\left|T_\alpha\right| < \kappa$.
In particular, if for every cardinal $\lambda < \kappa$ we have $\lambda^{\cf(\alpha)} < \kappa$,
then we can add a limit of every branch at level $\alpha$ into $T$,
and if we can do this for every limit ordinal $\alpha \in E^\kappa_{< \chi}$ then we can ensure that our tree is $\chi$-complete.
Of course we must forgo the slimness of the tree obtained in Theorems~\ref{ctbl-ascent-thm}
and~\ref{tree-ctbl-ascent}, but this is obvious, as these are contradictory concepts.

Since, for height $\alpha \in E^\kappa_{<\chi}$, we will not need the nodes $\mathbf b^\alpha_x$
to determine the contents of $T_\alpha$,
it is tempting to avoid constructing the $\mathbf b^\alpha_x$ for such $\alpha$ altogether.
It this a good idea?

This idea would actually be fine if the goal were only to construct a $\chi$-complete $\kappa$-Souslin tree,
without requiring an ascent path,
because in that case we could adjust the construction so that in the proof of Subclaim~\ref{get-b-ctbl-complete} below,
the first option always holds. Moreover,
this $\p_{14}(\kappa,{\sq},\theta,\{\kappa\})$-based construction would go through as a $\p_{14}(\kappa,{\sqx},\theta,\{E^\kappa_{\ge\chi}\})$-based construction.

However, in the presence of the ascent path functions, the proof of Subclaim~\ref{get-b-ctbl-complete}
requires us to have $\mathbf b^\alpha_x$ defined at levels $\bar\alpha$ of countable cofinality,
in particular, for $\bar\alpha\in\nacc(\acc(C_\alpha))$ on which the ``hitting'' action takes place.
Therefore, we shall maintain the construction of $\mathbf b^\alpha_x$ even at limit levels of  cofinality $<\chi$.

\smallskip

The next proof will demonstrate that there is a transparent way of transforming any proxy-based construction of a slim tree,
into a construction of a complete tree. To compare,
more than ten years after Devlin's paper \cite{MR732661} with a construction of an $\aleph_2$-Souslin tree admitting an $\mathcal F_{\aleph_0}^{\fin}$-ascent path,
Cummings \cite{MR1376756} gave a construction of such a tree which is moreover $\aleph_1$-complete.

\begin{theorem} \label{tree-ctbl-ascent-complete-thm}
Suppose that $\kappa$ is any regular uncountable cardinal, $\chi < \kappa$ is an infinite cardinal,
$U \subseteq {}^{<\kappa} \kappa$ is a downward-closed $\kappa$-tree,
and $\p_{14}(\kappa,  {\sqsubseteq}, 1, \{E^{\kappa}_{\ge\chi}\})$ holds.

If $\lambda^{<\chi} < \kappa$ for all $\lambda < \kappa$,
then there exists a prolific $\chi$-complete $\kappa$-Souslin tree
that admits an injective $(\mathcal F^{\fin}_{\aleph_0},U)$-ascent path.
\end{theorem}

\begin{proof}
Most of the proof is the same as the proof of Theorem~\ref{tree-ctbl-ascent}.

Let $\langle C_\alpha \mid \alpha < \kappa \rangle$ be a witness to
$\p^-_{14}(\kappa, {\sqsubseteq}, 1, \{ E^\kappa_{\geq\chi} \})$.
Similarly to  the proof of Theorem~\ref{tree-ctbl-ascent}, we will construct,
simultaneously by recursion over $\alpha < \kappa$,
the levels $\langle T_\alpha \mid \alpha < \kappa \rangle$ of the tree $T$
as well as the functions $\langle f_u \mid u \in U \rangle$
and the nodes
$\langle\langle \mathbf b^\alpha_x \mid x \in T \restriction C_\alpha \rangle\mid \alpha \in \acc(\kappa)\rangle$
so that after each stage $\alpha$ of the construction,
properties~(1)--(\ref{prolific}),
(\ref{b-simple})--(\ref{tree-order2-nacc}), and~(\ref{distinct-ascent})
of the construction in Theorem~\ref{tree-ctbl-ascent}
are satisfied, as well as the following:

\begin{enumerate}
\setcounterref{enumi}{slim}
\addtocounter{enumi}{-1}
\item \label{width<kappa}
$\left| T_\alpha \right| < \kappa$;

\setcounterref{enumi}{limit-ctbl-ascent}
\addtocounter{enumi}{-1}
\item \label{tree-lim-ctbl-ascent-complete}
\begin{enumerate}
\item
If $\alpha \in E^\kappa_{\geq\chi}$, then
\[
T_\alpha = \left\{ \mathbf b^\alpha_x \mid x \in T \restriction C_\alpha \right\} \cup
    \left\{ f_u (n) \mid u \in U_\alpha, n < \omega \right\};
\]

\item
If $\alpha \in \acc(\kappa) \cap E^\kappa_{<\chi}$, then
the limit of every branch through $T \restriction \alpha$ is a node in $T_\alpha$.
\end{enumerate}
\end{enumerate}

The recursive construction proceeds just as in the proof of Theorem~\ref{tree-ctbl-ascent},
with the following crucial difference:  At a limit level $\alpha$,
after constructing $\mathbf b^\alpha_x \in {}^\alpha\kappa$ for each $x \in T \restriction C_\alpha$
as well as the function $f_u : \omega \to {}^\alpha \kappa$ for each $u \in U_\alpha$,
the decision as to which elements of $^\alpha \kappa$ are included in $T_\alpha$
depends further on the nature of $\alpha$, as follows:

\begin{description}
\item[$\blacktriangleright \cf(\alpha) \geq \chi$]
In this case, we set
\[
T_\alpha = \left\{ \mathbf b^\alpha_x \mid x \in T \restriction C_\alpha \right\} \cup
    \left\{ f_u (n) \mid u \in U_\alpha, n < \omega \right\}.
\]

\item[$\blacktriangleright \cf(\alpha) < \chi$]
In this case,
let $T_\alpha$ consist of the limits of all branches
through $T \restriction \alpha$.
Notice that each $f_{u}(n)$ and each $\mathbf b^\alpha_x$ is constructed as the limit of a cofinal branch through $(T \restriction \alpha,\s)$,
and hence
\[
T_\alpha \supseteq \left\{ \mathbf b^\alpha_x \mid x \in T \restriction C_\alpha \right\} \cup
    \left\{ f_u (n) \mid u \in U_\alpha, n < \omega \right\}.
\]

\end{description}

The required properties are verified just as in the
proof of Theorem~\ref{tree-ctbl-ascent},
with the exception of:

\begin{itemize}

\item [(\ref{width<kappa})]
\begin{description}
\item[$\blacktriangleright \cf(\alpha) \geq \chi$]

Applying the induction hypothesis, for each $\beta < \alpha$ we have $\left|T_{\beta}\right|<\kappa$.
Since $\alpha < \kappa = \cf(\kappa)$, it follows that
\[
\left|T \restriction C_\alpha \right| \leq
\left|T \restriction \alpha \right|
                    = \sum_{\beta < \alpha} \left|T_\beta\right| < \kappa.
\]
Since $U$ is a $\kappa$-tree, $\left|U_\alpha\right| < \kappa$.
Then, since every node of the form $\mathbf b^\alpha_x$
is produced from some node $x \in T \restriction C_\alpha$,
and every node of the form $f_u (n)$ comes from some $u \in U_\alpha$ and some $n < \omega$,
it follows that
\[
\left|T_\alpha \right| \leq
    \left|T \restriction C_\alpha \right| + \left| U_\alpha \right| \cdot \aleph_0
    < \kappa,
\]
as required.

\item[$\blacktriangleright \cf(\alpha) < \chi$]

To bound the number of nodes in $T_\alpha$,
we need a bound on the number of branches through $T \restriction \alpha$.
Choose a sequence $\langle \alpha_\iota \mid \iota < \cf(\alpha) \rangle$ cofinal in $\alpha$.
Every branch $b$ through $T \restriction \alpha$ determines a distinct sequence
$\langle b \restriction \alpha_\iota \mid \iota < \cf(\alpha) \rangle$ of nodes,
where each $b \restriction \alpha_\iota \in T_{\alpha_\iota}$.
So the number of branches through $T \restriction \alpha$ is bounded by the number of such sequences,
which is
\[
\prod_{\iota < \cf(\alpha) } \left| T_{\alpha_\iota} \right|.
\]
Define
\[
\lambda = \sup_{\iota < \cf(\alpha)} \left| T_{\alpha_\iota} \right|.
\]
Applying the induction hypothesis, for each $\iota < \cf(\alpha)$ we have $\left|T_{\alpha_\iota}\right|<\kappa$.
Since $\cf(\alpha) \leq \alpha < \kappa = \cf(\kappa)$,
it follows that $\lambda < \kappa$.
Since $\cf(\alpha) < \chi$, we then have
\[
\prod_{\iota < \cf(\alpha) } \left| T_{\alpha_\iota} \right| \leq \prod_{\iota < \cf(\alpha) } \lambda =
    \lambda^{\cf(\alpha)} \leq \lambda^{<\chi} < \kappa,
\]
where  the last inequality comes from the arithmetic hypothesis in the statement of the theorem.
Thus the number of branches through $T \restriction \alpha$ is $<\kappa$, so that $\left|T_\alpha \right| < \kappa$,
as required.

\end{description}
\end{itemize}

The fact that $T=\bigcup_{\alpha<\kappa}T_\alpha$ is $\chi$-complete is exactly what is provided by
property~(\ref{tree-lim-ctbl-ascent-complete})(b) of the recursion.

\begin{claim}\label{claim431}
The tree $(T,{\subset})$ is a $\kappa$-Souslin tree.
\end{claim}

\begin{proof}
Let $A \subseteq T$ be a maximal antichain.
By Subclaim \ref{sc621}, the following set is stationary:
\[
A_0 = \{\beta\in R_0\mid  A\cap(T\restriction\beta)=S_\beta\text{ is a maximal antichain in }T\restriction\beta \}.
\]

So we apply the last part of the proxy principle to obtain an ordinal $\alpha\in E^\kappa_{\ge\chi}$
such that
\[
\sup \{ \beta \in C_\alpha  \mid \suc_\omega (C_\alpha \setminus \beta) \subseteq A_0 \} = \alpha.
\]

Let $v'$ be an arbitrary element of $T_\alpha$.

\begin{subclaim}\label{get-b-ctbl-complete}
There are some ${\bar\alpha} \in \acc(C_\alpha) \cup \{ \alpha \}$ and ${x} \in T \restriction C_{{\bar\alpha}}$ such that
\[
\mathbf b^{{{\bar\alpha}}}_{x} \leq_T v' \text{ and }
\sup \left(\nacc (C_{{\bar\alpha}}) \cap A_0 \right) = {\bar\alpha}.
\]
\end{subclaim}

\begin{proof}
Since $\alpha\in E^\kappa_{\ge\chi}$,
by property~(\ref{tree-lim-ctbl-ascent-complete})(a) there are now two possibilities to consider:
\begin{description}
\item[$\blacktriangleright v' = \mathbf b^\alpha_{x}$ for some $x \in T \restriction C_\alpha$]
In this case, fix such an $x$, and let ${\bar\alpha} = \alpha$,
and the subclaim is satisfied.

\item[$\blacktriangleright v' = f_u(n)$ for some $u\in U_\alpha$ and $n < \omega$]
Fix such  $u$ and $n$.
By our choice of $\alpha$, we can choose $\epsilon \in C_\alpha$ such that
$\suc_\omega (C_\alpha \setminus \epsilon) \subseteq A_0$.
Define
\[
{\bar\alpha} = \sup \left( \suc_\omega (C_\alpha \setminus \epsilon) \right).
\]
It is clear that ${\bar\alpha}$ is a limit ordinal, $\epsilon < {\bar\alpha} \leq \alpha$,
and ${\bar\alpha} \in \acc (C_\alpha) \cup \{ \alpha \}$.
Thus by property~(\ref{tree-ascent-on-acc}) we have $f_{{u\restriction \bar\alpha}}(n) \leq_T f_u (n)$.
By  $C_{{\bar\alpha}} = C_\alpha \cap {\bar\alpha}$, we have
$\sup(\acc(C_{{\bar\alpha}})) = \sup(\acc(C_\alpha \cap {\bar\alpha})) \leq \epsilon < {\bar\alpha}.$
Then, by applying property~(\ref{tree-order2-nacc}) to ${\bar\alpha}$,
we must have $f_{u\restriction{\bar\alpha}}(n) = \mathbf b^{{{\bar\alpha}}}_{x}$ for some $x \in T \restriction C_{{\bar\alpha}}$.
Fix such an $x$.
It follows that
\[
\mathbf b^{{{\bar\alpha}}}_{x} = f_{u\restriction {\bar\alpha}}(n) \leq_T f_u (n) = v'.
\qedhere
\]
\end{description}
\end{proof}
As in the proof of Claim \ref{Souslin-claim-ctbl-ascent}, it then follows that $v'$ extends some element $y$ from the antichain $A$.
As $v'$ was an arbitrary element of $T_\alpha$, this shows that  $A\s T\restriction\alpha$. Of course, $A$ was an arbitrary antichain, and hence the splitting $\kappa$-tree $(T,{\subset})$ is $\kappa$-Souslin.
\end{proof}
This completes the proof.
\end{proof}

Applying Theorem~\ref{tree-ctbl-ascent-complete-thm} to the special case $U=\bigcup_{\alpha<\kappa}{}^\alpha1$,
we obtain:

\begin{corollary} \label{ctbl-ascent-complete-thm}
Suppose that $\kappa$ is any regular uncountable cardinal, $\chi < \kappa$ is an infinite cardinal,
and $\p_{14}(\kappa,  {\sqsubseteq}, 1, \{E^\kappa_{\geq\chi}\})$ holds.

If $\lambda^{<\chi} < \kappa$ for all $\lambda < \kappa$,
then there exists a prolific $\chi$-complete $\kappa$-Souslin tree with an injective $\mathcal F^{\fin}_{\aleph_0}$-ascent path.
\end{corollary}

We remark that the techniques of this section can be used to produce a slim/complete $\kappa$-Souslin tree with
an \emph{injective} $(\mathcal F^{\bd}_\theta, U)$-ascent path from $\p_{14}(\kappa,  {\sqsubseteq}, 1, \{\kappa\})$,
not only for $\theta=\aleph_0$, but also for any cardinal $\theta < \kappa$ of countable cofinality.
For the slim tree, one needs to employ the trick of distinguishing small and large successor ordinals,
as is done in the proof of Theorem~\ref{NEW-theta-ascent-thm} below.

\begin{Q} Can a $\kappa$-Souslin tree with an $\mathcal F^\fin_{\aleph_0}$-ascent path be constructed from
$\p_{14}(\kappa, {\sq^*}, 1, \{\kappa\})$?\footnote{Here, $D \sqsubseteq^* C$ iff
$D \setminus \gamma \sqsubseteq C \setminus \gamma$ for some  $\gamma < \sup(D)$.}
If not, how about some simple intermediate relation $\mathcal R$ such that
${\sqsubseteq} \subseteq \mathcal R \subseteq {\sqsubseteq}^*$?
\end{Q}

\section{Wider ascent paths}\label{sectionwiderwidth}

As mentioned in the introduction, a $\lambda^+$-tree that admits an $\mathcal F^{\bd}_\theta$-ascent path is nonspecial (and hard to specialize),
unless $\cf(\theta)=\cf(\lambda)$. So, if $\lambda$ is a singular cardinal of countable cofinality, then one may be interested in $\lambda^+$-trees with an $\mathcal F^{\bd}_\theta$-ascent path for some regular uncountable cardinal $\theta$.
In fact, there are reasons to study $\mathcal F^{\bd}_\theta$-ascent paths for uncountable $\theta$, even for $\kappa$-trees where $\kappa$ is not a successor cardinal.
For instance, a simplified form of a theorem of L\"ucke from \cite{lucke} asserts that the Proper Forcing Axiom ($\pfa$) implies that for every regular cardinal $\kappa\ge\aleph_2$,
no $\kappa$-Aronszajn tree admits an $\mathcal F_{\aleph_0}$-ascent path, let alone an $\mathcal F^{\bd}_{\aleph_0}$-ascent path.

Therefore, in this section, we shall address the task of constructing $\kappa$-Souslin trees that admit $\mathcal F^{\bd}_\theta$-ascent paths for various values of $\theta$.
Of course, whenever possible, we shall want to obtain $\mathcal F^{\fin}_\theta$-ascent paths.
To better understand what is possible and what is not possible,
let $\nu<\kappa$ denote infinite regular cardinals, and $(T,<_T)$ some $\kappa$-Aronszajn tree.
An adaptation of the argument of \cite[Theorem 41H]{MR780933} entails that
if $\nu$ is a supercompact cardinal, then $(T,<_T)$ admits no $\mathcal F^{\bd}_\theta$-ascent path for every infinite cardinal $\theta\in E^\kappa_{<\nu}$.
So, the best one can hope for in this scenario is the existence of an $\mathcal F^{\bd}_\theta$-ascent path for  $\theta\in E^\kappa_{\ge\nu}$.
For this, we define the following $\nu$-complete filter over $\theta$:
\[\label{defnfnutheta}
\mathcal F^\nu_\theta=\{ Z\s \theta\mid |\theta\setminus Z|<\nu\}.
\]

Clearly, $\mathcal F^{\aleph_0}_\theta = \mathcal F^{\fin}_\theta$.
More importantly, $\mathcal F^{\nu}_\theta$ projects to a subfilter of $\mathcal F^{\bd}_\mu$,
for all $\mu\in[\aleph_0,\theta]\cap\cof(\ge\nu)$.
Let us demonstrate how this helps.

By \cite{axioms}, $\p_{14}(\lambda^+,{\sq_\nu},\lambda^+,\{\lambda^+\})$ is consistent together with $\nu$ being supercompact and, say,  $\lambda=\nu^{+\omega}$.
By the upcoming theorem, this entails a $\lambda^+$-Souslin tree $(T,{\subset})$ with an $\mathcal F_\lambda^\nu$-ascent path.
Consequently, for all $\theta\le\lambda$: the tree $(T,{\subset})$ admits an $\mathcal F^{\bd}_\theta$-ascent path iff $\cf(\theta)\ge\nu$.
So the results of this section are sharp.

Coming back to L\"ucke's theorem, we mention that the proof of Corollary \ref{cor116}
demonstrates the consistency of $\pfa$ together with principles of the form $\p_{14}(\kappa,{\sq_{\aleph_2}},\ldots)$.

\medskip

The constructions of $\kappa$-Souslin trees in this section will be from the principle $\p_{14}(\kappa,\mathcal R,\theta,\mathcal S)$,
where $\theta<\kappa$ is the width of the ascent path. In the previous section, we managed to get by, assuming merely $\p_{14}(\kappa,\mathcal R,1,\mathcal S)$.
This was possible, because the ordinal $\bar\alpha$ from Subclaim \ref{get-b-ctbl} that was responsible for sealing antichains is of countable cofinality.
For $\theta$ of uncountable cofinality, we occur into a situation of \emph{mismatch of cofinalities},
that prevents addressing $f_\alpha(\iota)$ for all $\iota<\theta$ at once. This is resolved by increasing the third parameter into $\theta$,
and handling $f_\alpha(\iota)$ for each $\iota<\theta$ separately.

\begin{theorem}\label{NEW-theta-ascent-thm}
Suppose that $\nu < \kappa$ are regular infinite cardinals, $\theta < \kappa$ is any infinite cardinal,
$U \subseteq {}^{<\kappa} \kappa$ is a slim downward-closed $\kappa$-tree,
and $\p_{14}(\kappa,  {\sqsubseteq_{\nu}}, \theta, \{\kappa\})$ holds.

Then there exists a prolific slim $\kappa$-Souslin tree with an injective $(\mathcal F^{\nu}_{\theta}, U)$-ascent path.
\end{theorem}

\begin{proof}
By passing to an isomorphic slim tree if necessary,
we may assume that for all $u \in U$ and all $\beta < \dom(u)$,
$u(\beta) < \left| U_{\beta+1} \right|$.

We commence, using $\diamondsuit(H_\kappa)$, by fixing the
functions $\phi:\kappa\leftrightarrow H_\kappa$, $\psi:\kappa\rightarrow H_\kappa$,
sequences $\langle S_\beta \mid \beta<\kappa\rangle$, $\langle R_i  \mid i < \kappa \rangle$,
well-ordering $<_\phi$,
notation $\beta(T)$,
and the functions $\defaultaction:H_\kappa\times H_\kappa\rightarrow H_\kappa$ and
$\sealantichain:H_\kappa\times H_\kappa\rightarrow H_\kappa$
as described in Section~\ref{section:basicsetup}.
Let $\langle C_\alpha \mid \alpha < \kappa \rangle$ be a witness to $\p^-_{14}(\kappa,  {\sqsubseteq_{\nu}}, \theta, \{\kappa\})$.
Without loss of generality, we may assume that  $C_\alpha=C_\alpha\setminus(\theta+1)$ whenever $\theta<\alpha<\kappa$.

Having weakened the second parameter from $\sqsubseteq$ to $\sqsubseteq_{\nu}$,
there is no way to guarantee that the sequence $b^\alpha_x$ can always be constructed.
Thus, unlike the constructions from the previous section,
we are not going to define $b_x^\alpha$ for every limit $\alpha<\kappa$ and
every $x\in T\restriction C_\alpha$.
Rather, this time, we shall define $b_x^\alpha$ only when $\alpha\in\Gamma$ for a particular stationary subset $\Gamma$ of $\kappa$.

However, the lack of sequences $b^\alpha_x$ for limit ordinals $\alpha \notin \Gamma$
leaves us with another problem:
How do we guarantee normality at these levels?
One of the main uses of the sequences $b^\alpha_x$ was to ensure
the existence of a node at level $\alpha$ above $x$
(see proof of property~(\ref{normal}) in the limit-level construction of Theorem~\ref{ctbl-ascent-thm}),
which in turn was necessary in order to apply the Extension Lemma from page~\pageref{extendfact} during the course of the construction.
In the absence of some of the sequences in this construction,
we shall obtain normality by another means: the ascent path.
Instead of constructing a single ascent path $\langle f_\alpha \mid \alpha < \kappa \rangle$ as in
Theorem \ref{ctbl-ascent-thm},
we shall construct an ascent path $\langle f_{x, \alpha} \mid \alpha < \kappa \rangle$
for every node $x$ of the tree,
where each value $f_{x, \alpha} (\iota)$ of the ascent path will always be compatible with the node $x$.
This way, whenever $\alpha>\height(x)$, we shall have $f_{x,\alpha}(\iota)$ serving as an extension of $x$ to level $\alpha$.

But the above-mentioned ascent paths serve only to ensure the normality of the tree,
and have nothing to do with the task of injecting an $(\mathcal F^\nu_\theta,U)$-ascent path through the tree.
While true, it turned out to be convenient to address these tasks in a uniform way.
Therefore, we shall simply construct $\langle f_{x, u} \mid u\in U\rangle$ for each $x\in T$.

A last remark before hitting the construction. As $f_{x,u}(\iota)$ is required to be compatible with $x$,
this means that for every $u\in U$ with $\height_U(u)\le\height_T(x)$, we must obviously have $f_{x,u}(\iota)=x\restriction\height_U(u)$.
For this reason, we shall only be explicitly specifying $f_{x,u}$ for $u\in U$ with $\height_U(u)>\height_T(x)$.

\begin{claim} \label{Gamma-includes}\label{Gamma-works}
Define $\Gamma = \{\alpha \in \acc(\kappa) \setminus(\theta+1)\mid \left( \forall \beta \in \acc(C_\alpha) \right) C_\beta \sqsubseteq C_\alpha\}$. Then:
\begin{enumerate}
\item If $\alpha\in\Gamma$ and $\bar\alpha\in\acc(C_\alpha)$, then ${\bar\alpha}\in\Gamma$;
\item $\Gamma \supseteq \{ \alpha\in \acc(\kappa) \setminus(\theta+1)\mid \otp(C_\alpha) \geq \nu \text{ or $\nacc(C_\alpha)$ contains a limit ordinal} \}$.
\end{enumerate}

In particular, $\Gamma$ covers the stationary set $E^\kappa_{\ge\nu}\setminus (\theta +1)$.
\end{claim}
\begin{proof}
(1) Fix $\alpha \in \Gamma$ and ${\bar\alpha} \in \acc(C_\alpha)$. We must show that ${\bar\alpha} \in \Gamma$.
Clearly ${\bar\alpha}$ is a limit ordinal $>\theta$, and from the fact that $\alpha \in \Gamma$
it follows that $C_{\bar\alpha} \sqsubseteq C_\alpha$.
Consider any $\beta \in \acc(C_{\bar\alpha})$.
Then also $\beta \in \acc(C_\alpha)$,
and it follows (again from $\alpha \in \Gamma$) that $C_\beta \sqsubseteq C_\alpha$.
Then $C_\beta = C_\alpha \cap \beta = C_\alpha \cap {\bar\alpha} \cap \beta = C_{\bar\alpha} \cap \beta$,
so that $C_\beta \sqsubseteq C_{\bar\alpha}$,
as required to show that ${\bar\alpha} \in \Gamma$.

(2) Suppose $\alpha\in \acc(\kappa) \setminus(\theta+1)$ and $\otp(C_\alpha)\ge\nu$.
Then for all ${\bar\alpha}\in\acc(C_\alpha)$, we have $C_{\bar\alpha} \sqsubseteq_{\nu} C_\alpha$,
which must mean that $C_{\bar\alpha}\sqsubseteq C_\alpha$. Consequently, $\alpha\in\Gamma$.

Suppose $\alpha\in \acc(\kappa) \setminus(\theta+1)$ and $\nacc(C_\alpha)$ contains a limit ordinal.
Then for all ${\bar\alpha}\in\acc(C_\alpha)$, we have $C_{\bar\alpha} \sqsubseteq_{\nu} C_\alpha$,
which must mean that $C_{\bar\alpha}\sqsubseteq C_\alpha$. Consequently, $\alpha\in\Gamma$.
\end{proof}

As always, the tree $T$ will be a downward-closed subset of $^{<\kappa} \kappa$,
so that each level $T_\alpha$ will be a subset of $^\alpha \kappa$,
and the tree relation $\leq_T$ will simply be extension of sequences.
We will construct,
simultaneously by recursion over $\alpha < \kappa$,
the levels $\langle T_\alpha \mid \alpha < \kappa \rangle$ of the tree $T$
as well as the functions $\langle\langle f_{x, u} \mid x\in T\restriction\alpha, u \in U_\alpha \rangle\mid \alpha<\kappa\rangle$
and the nodes $\langle\langle \mathbf b^\alpha_x \mid x \in T \restriction C_\alpha \rangle\mid \alpha \in \Gamma \rangle$,
so that after each stage $\alpha$,
properties~(1)--(\ref{slim}) of the construction in Theorem~\ref{ctbl-ascent-thm}
are satisfied, as well as the following:

\begin{enumerate}
\setcounterref{enumi}{slim}

\item \label{b-theta-ascent}
If $\alpha \in \Gamma$,
then
for every $x \in T \restriction C_\alpha$,
$\mathbf b^\alpha_x \in T_\alpha$ is the limit of the
increasing, continuous, cofinal sequence $b^\alpha_x$
in $(T \restriction \alpha, \subseteq )$,
satisfying the same properties (a)--(d) as in the corresponding
property~(\ref{b-simple}) of Theorem~\ref{ctbl-ascent-thm};

\item \label{tree-theta-ascent-is-function}\label{tree-successor-ascent-injection-theta}
For every $x \in T \restriction \alpha$ and every $u \in U_\alpha$,
$f_{x, u} : \theta \to T_\alpha \cap \cone{x}$ is a function.
Moreover:
\begin{enumerate}
\item
If $\alpha \leq \theta$, then $f_{x,u}$ is a constant function;
\item
If $\alpha > \theta$ is a successor ordinal,
then $f_{x,u}$ is injective;
\end{enumerate}

\item \label{tree-cobounded-ascent-theta}
For every $\beta < \alpha$, every $x \in T \restriction \beta$, and every $u \in U_\alpha$,
\[
\{ \iota < \theta \mid f_{x, u \restriction \beta}(\iota) <_T f_{x, u}(\iota) \} \in \mathcal F^{\nu}_\theta;
\]

\item \label{tree-ascent-on-acc-theta}
\begin{enumerate}
\item
If $\alpha \leq \theta$,
then for every $\beta < \alpha$, every $x \in T \restriction \beta$, and every $u \in U_\alpha$,
\[
\{ \iota < \theta \mid f_{x, u \restriction \beta}(\iota) <_T f_{x, u}(\iota) \} = \theta;
\]
\item
If $\alpha\in\Gamma$ and $\beta \in \acc (C_\alpha)$,
then for every $x \in T \restriction \beta$ and every $u \in U_\alpha$,
\[
\{ \iota < \theta \mid f_{x, u \restriction \beta}(\iota) <_T f_{x, u}(\iota) \} = \theta;
\]
\end{enumerate}

\item \label{order2-nacc-theta-many-tree}
If $\alpha \in \Gamma$ satisfies $\sup (\acc(C_\alpha)) < \alpha$,
and if $\iota^* < \theta$ is such that, for some $\beta < \alpha$, $\psi[C_\alpha \setminus \beta] = \{ \iota^* \}$,
then for every $x \in T \restriction \alpha$ and every $u \in U_\alpha$
there is some $y \in T \restriction C_\alpha$ such that
\[
f_{x, u} (\iota^*) = \mathbf b^\alpha_{y};
\]

\item \label{tree-Gamma}
If $\alpha \in \Gamma$,
then
\[
T_\alpha = \{ \mathbf b^\alpha_x \mid x \in T \restriction C_\alpha \} \cup
                \{ f_{x, u} (\iota) \mid x \in T \restriction \alpha, u \in U_\alpha, \iota < \theta \};
\]

\item \label{tree-not-Gamma-level}
If $\alpha \in \acc(\kappa) \setminus \Gamma$, then
\[
T_\alpha = \{ f_{x, u} (\iota) \mid x \in T \restriction \alpha, u \in U_\alpha, \iota < \theta \};
\]

\item \label{theta-distinct-ascent}
For any two distinct nodes $u,v$ from $U_\alpha$,
\[
\{ \iota < \theta \mid f_{\emptyset, u} (\iota) \neq f_{\emptyset, v} (\iota) \} \in \mathcal F^{\nu}_\theta.
\]
\end{enumerate}

Notice that property~(\ref{tree-successor-ascent-injection-theta})(b) is weaker than in previous theorems,
in that we do not require $f_{x, u}$ to be an injection when $\alpha$ is a limit ordinal.

\smallskip

We leave for the reader to verify that
the following instance of the Coherence Claim Template from page \pageref{coherence-template} holds:
\begin{claim}\label{coherence-Gamma}
Fix limit ordinals ${\bar\alpha} < \alpha$ both in $\Gamma$, and suppose $T \restriction \alpha$ has been constructed
to satisfy the above properties.
If $C_{\bar\alpha} = C_\alpha \cap {\bar\alpha}$,  $x \in T \restriction C_{\bar\alpha}$,
and $b^\alpha_x \restriction (C_{\bar\alpha} \setminus \height(x))$ has already been constructed,
then
\[
b^{\bar\alpha}_x = b^\alpha_x \restriction \left( C_{\bar\alpha} \setminus \height(x) \right).
\]
\end{claim}

The recursive construction proceeds as follows:

\begin{description}
\item[Base case, $\alpha = 0$]
As always, let
$T_0 = \{ \emptyset \}$.
The required properties are automatically satisfied as there is nothing to check.
We do not define any ascent-path function here, since our commitment is to define $f_{x, u}$
only when $x \in T \restriction \alpha$, and of course $T \restriction 0$ is empty.

\item[Small successor ordinal, $\alpha = \beta+1 < \theta$]
In this case, define
\[
T_\alpha = \{ t^\smallfrown \langle \iota \rangle  \mid t \in T_\beta, \iota< \max \{\left| U_\alpha \right|, \omega, \alpha\} \}.
\]
In addition, for every $x \in T \restriction \alpha$ and every $u \in U_\alpha$,
define the constant function $f_{x, u} : \theta \to T_\alpha \cap \cone{x}$ by setting,
for all $\iota < \theta$,
\[
f_{x, u} (\iota) =
    \begin{cases}
        x^\smallfrown \langle u(\beta) \rangle,   &\text{if $x \in T_\beta$;} \\
        f_{x, u \restriction \beta}(\iota)^\smallfrown \langle u(\beta) \rangle,
                                        &\text{if $x \in T \restriction \beta$.}
    \end{cases}
\]
The required properties are easy to verify.

\item[Large successor ordinal, $\alpha = \beta+1 > \theta$]
In this case, define
\[
T_\alpha = \{ t^\smallfrown \langle\iota\rangle\mid
            t \in T_\beta, \iota< \max \{ \theta \cdot \left| U_\alpha \right|, \alpha\} \}.
\]
In addition, for every $x \in T \restriction \alpha$ and every $u \in U_\alpha$,
define the injective function $f_{x, u} : \theta \to T_\alpha \cap \cone{x}$ by setting,
for all $\iota < \theta$,
\[
f_{x, u} (\iota) =
    \begin{cases}
        x^\smallfrown \langle\theta \cdot (u(\beta)) + \iota \rangle,   &\text{if $x \in T_\beta$;} \\
        f_{x, u \restriction \beta}(\iota)^\smallfrown \langle \theta \cdot (u(\beta)) + \iota \rangle,
                                        &\text{if $x \in T \restriction \beta$.}
    \end{cases}
\]
The required properties are easy to verify.

\item[Limit level, $\alpha \notin \Gamma$]
Since $\alpha \notin \Gamma$, we do not define any nodes of the form $\mathbf b^\alpha_x$.

Let $x\in T\restriction\alpha$ and $u \in U_\alpha$ be arbitrary.
We need to define $f_{x, u} : \theta\rightarrow{}^\alpha\kappa$.
For every $\beta_0<\beta_1$ in $C_\alpha\setminus(\height(x)+1)$,
define
\[
F_{x,u}^{\beta_0, \beta_1} =
    \left\{ \iota < \theta \mid \left( f_{x, u \restriction \beta_0}(\iota) <_T f_{x, u \restriction \beta_1}(\iota) \right) \right\},
\]
and then let
\[
F^*_{x, u} = \bigcap \left\{ F^{x,u}_{\beta_0, \beta_1} \mid
    \beta_0<\beta_1\text{ in }C_\alpha \setminus (\height(x) +1) \right\}.
\]

\begin{claim}\label{F-in-filter}
$F^*_{x, u} \in \mathcal F^{\nu}_\theta$.
\end{claim}

\begin{proof}
There are two cases to consider:
\begin{description}
\item[$\blacktriangleright \alpha \leq \theta$]
For every $\beta_0<\beta_1$ in $C_\alpha\setminus(\height(x)+1)$,
applying property~(\ref{tree-ascent-on-acc-theta})(a) to $\beta_1$ gives $F_{x,u}^{\beta_0, \beta_1} = \theta$.
Thus in fact $F^*_{x, u} = \theta \in \mathcal F^{\nu}_\theta$ in this case.

\item[$\blacktriangleright \alpha > \theta$]
By property~(\ref{tree-cobounded-ascent-theta}),
$F_{x,u}^{\beta_0, \beta_1} \in \mathcal F^{\nu}_\theta$
for every $\beta_0<\beta_1$ in $C_\alpha\setminus(\height(x)+1)$.
Since $\alpha \notin \Gamma$,
Claim~\ref{Gamma-includes}(2) gives $\otp(C_\alpha) < \nu$ in this case.
The result now follows from the $\nu$-completeness of
the filter $\mathcal F^{\nu}_\theta$.
\qedhere
\end{description}
\end{proof}

By definition of $F^*_{x, u}$, for all $\iota \in F^*_{x, u}$,
the sequence $\langle f_{x, u \restriction \beta}(\iota) \mid \height(x)\in \beta \in C_\alpha  \rangle$ is increasing
and cofinal in $( T \restriction \alpha, \subseteq )$.
Denote $\iota^*_{x, u} = \min (F^*_{x, u})$.
Then, define $f_{x, u} : \theta \to {}^\alpha \kappa$ by stipulating
\[
f_{x, u} (\iota) =
    \begin{cases}
        \bigcup \{ f_{x, u \restriction \beta} (\iota) \mid \height(x)\in \beta \in C_\alpha \},
                                &\text{if } \iota \in F^*_{x, u};       \\
        \bigcup \{ f_{x, u \restriction \beta} (\iota^*_{x, u}) \mid \height(x)\in \beta \in C_\alpha \},
                        &\text{otherwise.}
    \end{cases}
\]

Clearly, $f_{x, u} (\iota) \in {}^\alpha \kappa$ for each $\iota < \theta$.

Finally, as promised, we set
\[
T_\alpha = \{ f_{x, u} (\iota) \mid x \in T \restriction \alpha, u \in U_\alpha, \iota < \theta \};
\]

To verify some of the required properties:

\begin{itemize}
\item [(\ref{right-level})]
Each $f_{x, u}(\iota)$ is
the limit of some cofinal branch through $(T \restriction \alpha,\s)$, so it is in ${}^\alpha \kappa$.

\item[(\ref{normal})]
Since $(U,\subset)$ is a $\kappa$-tree,
in particular $U_\alpha \neq \emptyset$, so by picking an arbitrary $u \in U_\alpha$,
we see that for every $x \in T \restriction \alpha$,
we have defined some node $f_{x, u}(0) \in T_\alpha$ above $x$.

\item[(\ref{tree-successor-ascent-injection-theta})(a)]
Assuming $\alpha \leq \theta$:
By the induction hypothesis,
each $f_{x, u \restriction \beta}$ for $\beta \in C_\alpha \setminus (\height(x)+1)$ is a constant function,
and by Claim~\ref{F-in-filter}, $F^*_{x,u}$ is nonempty, so
the sequence consisting of their constant values,
$\langle f_{x, u \restriction \beta}(0) \mid \beta \in C_\alpha \setminus (\height(x)+1) \rangle$,
must be increasing and cofinal in $(T \restriction \alpha, \subseteq)$.
Our definition of $f_{x, u}$ in this case then gives, for every $\iota < \theta$,
\[
f_{x, u} (\iota) =
\bigcup \{ f_{x, u \restriction \beta} (0) \mid \height(x)\in \beta \in C_\alpha \},
\]
so that the function $f_{x, u}$ is constant.

\item [(\ref{slim})]
Just as in Theorem~\ref{ctbl-ascent-thm},
we have $\left|T \restriction \alpha \right| \leq \left|\alpha\right|$.
Since $(U,\subset)$ is slim and $\left|\alpha\right| \geq \aleph_0$, we have $\left|U_\alpha\right| \leq \left|\alpha\right|$.
For every pair of nodes $x \in T \restriction \alpha$ and $u \in U_\alpha$,
we show that $\left|\range(f_{x, u})\right| \leq \left|\alpha\right|$,
by considering two cases:

\begin{description}
\item[$\blacktriangleright \alpha \leq \theta$]
In this case, property~(\ref{tree-successor-ascent-injection-theta})(a) tells us that $f_{x,u}$ is a constant function, so that
\[
\left|\range(f_{x, u})\right|  = 1 < \left|\alpha\right|.
\]

\item[$\blacktriangleright \alpha > \theta$]
In this case, we have
\[
\left|\range(f_{x, u})\right| \leq \left|\dom(f_{x, u})\right| = \theta \leq \left|\alpha\right|.
\]
\end{description}

In both cases, we then have
\[
\left|T_\alpha \right| \leq
    \left|T \restriction \alpha \right| \cdot \left| U_\alpha \right| \cdot \sup \{ \left|\range(f_{x, u})\right|  \mid x \in T \restriction \alpha, u \in U_\alpha \}
    \leq \left|\alpha\right| \cdot \left|\alpha\right| \cdot \left|\alpha\right|
    = \left|\alpha\right|,
\]
as required.

\item [(\ref{tree-cobounded-ascent-theta})]
Fix $ \beta < \alpha$, $x \in T \restriction \beta$, and $u \in U_\alpha$.
Since $\sup(C_\alpha) = \alpha$, find some $\beta' \in C_\alpha$ such that $\beta  < \beta' < \alpha$.
Define $F = \{ \iota < \theta \mid f_{x, u \restriction \beta}(\iota) <_T f_{x, u \restriction \beta'}(\iota) \}$.
Applying the induction hypothesis to $u \restriction \beta'$,
we know that $F \in \mathcal F^{\nu}_\theta$.
Then also $F \cap F^*_{x, u} \in \mathcal F^{\nu}_\theta$.

For every $\iota \in F^*_{x, u}$, since $\beta' \in C_\alpha$, we have defined
$f_{x, u} (\iota)$ to be above $f_{x, u \restriction \beta'}(\iota)$.
Thus for any $\iota \in F \cap F^*_{x, u}$, we have
$f_{x, u \restriction \beta}(\iota) <_T f_{x, u \restriction \beta'}(\iota) <_T f_{x, u}(\iota)$,
as required.

\item[(\ref{tree-ascent-on-acc-theta})(a)]
Assuming $\alpha \leq \theta$:
Follow the same proof as for (\ref{tree-cobounded-ascent-theta}),
but this time we have $F = F^*_{x, u} = \theta$.

\item[(\ref{theta-distinct-ascent})]
As in property~(\ref{distinct-ascent}) of Theorem~\ref{tree-ctbl-ascent}.

\end{itemize}

\item[Limit level, $\alpha \in\Gamma$]
We begin by constructing $\mathbf b^\alpha_x \in {}^\alpha \kappa$ for each $x \in T \restriction C_\alpha$,
just as in Theorem~\ref{ctbl-ascent-thm}.
Of course, when defining $b^\alpha_x (\beta)$ for some $\beta \in \acc(C_\alpha)$,
it is crucial that $b^\alpha_x (\beta) \in T_\beta$.
The latter is indeed the case, thanks to
Claims~\ref{Gamma-works} and \ref{coherence-Gamma}
and the fact that $\alpha \in \Gamma$.

Next, we fix $x \in T \restriction \alpha$ and $u \in U_\alpha$, and
we must construct a function $f_{x, u}: \theta\rightarrow{}^\alpha \kappa$.
Fix $\iota < \theta$, and let us prescribe a function value $f_{x, u}(\iota) \in {}^\alpha \kappa$.
We shall need the following variant of Claim~\ref{full-domination-acc}:

\begin{claim} \label{full-domination-acc-theta-tree}
The sequence $\langle f_{x, u \restriction \beta}(\iota) \mid \beta \in \acc(C_\alpha) \setminus (\height(x)+1) \rangle$ is increasing
in $( T \restriction \alpha, \subseteq )$.
\end{claim}
\begin{proof}
Consider any $\beta_1, \beta_2 \in \acc(C_\alpha) \setminus (\height(x)+1)$ with $\beta_1 < \beta_2$.
Since $\alpha \in \Gamma$ and $\beta_2 \in \acc(C_\alpha)$,
it follows that $C_{\beta_2} \sqsubseteq C_\alpha$.
Since $\beta_1 < \beta_2$ and $\beta_1 \in \acc(C_\alpha)$, we must then have $\beta_1 \in \acc (C_{\beta_2})$.
Claim~\ref{Gamma-works} gives $\beta_2 \in \Gamma$.
By property~(\ref{tree-ascent-on-acc-theta})(b) of the induction hypothesis applied to $\beta_2$,
it follows that $f_{x, \beta_1}(\iota) <_T f_{x, \beta_2}(\iota)$, as required.
\end{proof}

Let
\[
\alpha_x = \sup \left( \acc(C_\alpha) \cup \{ \min (C_\alpha \setminus (\height(x)+1)) \} \right).
\]
It is clear from the definition that $\height(x)<\alpha_x\le\alpha$,
that $\alpha_x \in C_\alpha \cup \{ \alpha \}$,
and that $\alpha_x = \sup(\acc (C_\alpha))$ iff $\height(x) < \sup(\acc(C_\alpha))$.
Notice also that $\alpha_\emptyset$ coincides with $\alpha_0$ of Theorem~\ref{ctbl-ascent-thm}.

The definition of $f_{x, \alpha} (\iota)$ splits into two possibilities:

\begin{description}%[style=unboxed]
\item [$\blacktriangleright \alpha_x = \alpha$]
In particular, $\sup(\acc(C_\alpha)) = \alpha$.
By Claim~\ref{full-domination-acc-theta-tree},
the sequence $\langle f_{x, u \restriction \beta}(\iota) \mid \height(x)\in \beta \in \acc(C_\alpha)  \rangle$ is
increasing, and in this case it is cofinal in $(T \restriction \alpha, \subseteq)$,
so we let
\[
f_{x, u} (\iota) = \bigcup \left\{ f_{x, u \restriction \beta} (\iota) \mid \height(x)\in \beta \in \acc (C_\alpha)  \right\}.
\]
Clearly, $f_{x, u} (\iota) \in {}^\alpha \kappa$.

\item [$\blacktriangleright \alpha_x < \alpha$]
In this case,
$C_\alpha \setminus \alpha_x$ is an $\omega$-type cofinal subset of $\alpha$.
Let $\langle \alpha^m_x \mid m<\omega \rangle$ denote the
increasing enumeration of $C_\alpha \setminus \alpha_x$,
so that $\alpha_x = \alpha^0_x$.

Let us  define $f_{x, u}(\iota)$
by considering several possibilities:

\begin{itemize}
\item
If there exists some $m<\omega$ such that
$\psi (\alpha^k_x) =\iota$ whenever $m<k<\omega$,
then let
\[
f_{x, u}(\iota)=\mathbf{b}^\alpha_{f_{x, u\restriction \alpha_x} (\iota)}.
\]

\item
Otherwise, consider the ordinal
\[
m_\iota=\sup \left\{ m<\omega \mid \langle f_{x, u \restriction \alpha^n_x}(\iota) \mid n\le m \rangle\text{ is $<_T$-increasing}\right\},
\]
and let
\[
f_{x, u}(\iota)=
    \begin{cases}
        \mathbf{b}^\alpha_{f_{x, u \restriction \alpha_x^{m_\iota}}(\iota)},    &\text{if } m_\iota<\omega; \\
        \bigcup \{ f_{x, u \restriction \alpha^n_x}(\iota) \mid n<\omega \},        &\text{if } m_\iota=\omega.
    \end{cases}
\]

\end{itemize}

In all cases, it is clear that $f_{x, u} (\iota) \in {}^\alpha \kappa$,
as it is the limit of a cofinal branch through $(T \restriction \alpha,\s)$.

\end{description}

Having constructed $f_{x, u} (\iota)$,
we now have the following variant of Claim~\ref{a0-goes-right}:

\begin{claim} \label{a0-goes-right-theta-many-tree}
If $\alpha_x < \alpha$ then
$f_{x, u \restriction \alpha_x} (\iota) <_T f_{x, u} (\iota)$.
\end{claim}

\begin{proof}
Referring back to the construction of $f_{x, u} (\iota)$,
we see that in all subcases, either
\[
f_{x, u \restriction \alpha_x} (\iota) <_T \mathbf b^\alpha_{f_{x, u \restriction \alpha_x} (\iota)} = f_{x, u}(\iota),
\]
or
\[
f_{x, u \restriction \alpha_x}(\iota) \leq_T f_{x, u \restriction \alpha_x^{m_\iota}}(\iota)
    <_T \mathbf{b}^\alpha_{f_{x, u \restriction \alpha_x^{m_\iota}}(\iota)} = f_{x, u}(\iota),
\]
or
\[
f_{x, u \restriction \alpha_x}(\iota) <_T  \bigcup \{ f_{x, u \restriction \alpha^n_x}(\iota) \mid n<\omega \}
     = f_{x, u}(\iota),
\]
so that
the required condition is satisfied.
\end{proof}

Finally, as promised, we set
\[
T_\alpha = \{ \mathbf b^\alpha_x \mid x \in T \restriction C_\alpha \} \cup
                \{ f_{x, u} (\iota) \mid x \in T \restriction \alpha, u \in U_\alpha, \iota < \theta \}.
\]
To verify some of the required properties:

\begin{itemize}
\item [(\ref{right-level})]
Each sequence $b^\alpha_x$ defines a cofinal branch through $(T \restriction \alpha,\s)$,
so that its limit $\mathbf b^\alpha_x \in {}^\alpha \kappa$.

Each $f_{x, \alpha}(\iota)$ is either equal to some $\mathbf b^\alpha_y$ or
the limit of some other cofinal branch through $(T \restriction \alpha,\s)$, so it is in ${}^\alpha \kappa$.

\item[(\ref{normal})]
As $(U,\subset)$ is a $\kappa$-tree,
by picking an arbitrary  $u \in U_\alpha$,
we see that for every $x \in T \restriction \alpha$,
we have defined some node $f_{x, u}(0) \in T_\alpha$ above $x$.%
\footnote{In this case, $\alpha\in\Gamma$, so we also have normality using the nodes $\mathbf b^\alpha_x$, as in Theorem~\ref{ctbl-ascent-thm}.}

\item [(\ref{slim})]
Following the proof of the same property in the case $\alpha \notin \Gamma$,
we have
\[
\left| T \restriction C_\alpha \right|
    \leq \left| T \restriction \alpha \right|
    \leq \left| \alpha \right|
\]
as well as $\left| U_\alpha \right| \leq \left| \alpha \right|$.
Since $\alpha \in \Gamma$, it follows that $\alpha > \theta$.
In this case, every node of the form $\mathbf b^\alpha_x$
is produced from some node $x \in T \restriction C_\alpha$,
and every node of the form $f_{x, u} (\iota)$ comes from
some pair of nodes $x \in T \restriction \alpha$ and $u \in U_\alpha$
as well as some $\iota < \theta$,
so it follows that
\[
\left| T_\alpha \right|
    \leq \left| T \restriction C_\alpha \right| +
        \left| T \restriction \alpha \right| \cdot \left| U_\alpha \right| \cdot \theta
    \leq \left| \alpha \right| + \left| \alpha \right| \cdot \left| \alpha \right| \cdot \left| \alpha \right|
    = \left| \alpha \right|,
\]
as required.

\item [(\ref{tree-ascent-on-acc-theta})(b)]
Fix $\beta \in \acc(C_\alpha)$, $x \in T \restriction \beta$, $u\in U_\alpha$ and $\iota < \theta$.
We must show that $f_{x, u \restriction \beta} (\iota) <_T f_{x, u} (\iota)$.
Referring back to the construction of $f_{x, u}$, there are two cases to check:

\begin{description}%[style=unboxed]
\item[$\blacktriangleright \alpha_x = \alpha$]
In this case, $f_{x, u} (\iota)$ was constructed to be above $f_{x, u \restriction \beta} (\iota)$.

\item[$\blacktriangleright  \alpha_x < \alpha$]
Since $\height(x) < \beta \in \acc(C_\alpha)$, in particular $\height(x) < \sup (\acc (C_\alpha))$,
so that
$\alpha_x = \sup(\acc(C_\alpha))$.
Thus
$\beta \leq \alpha_x$ and (since $C_\alpha$ is club in $\alpha > \alpha_x$) $\alpha_x \in \acc (C_\alpha)$.
We then have
\begin{align*}
f_{x, u \restriction \beta} (\iota)  &\leq_T f_{x, u \restriction \alpha_x} (\iota)
                &&\text{from Claim~\ref{full-domination-acc-theta-tree}}                 \\
    &<_T f_{x, u} (\iota)      &&\text{by Claim~\ref{a0-goes-right-theta-many-tree}},
\end{align*}
as required.

\end{description}

\item [(\ref{tree-cobounded-ascent-theta})]
Fix $\beta < \alpha$, $x \in T \restriction \beta$, and $u \in U_\alpha$.
Again referring back to the construction of $f_{x, u}$, there are two cases to check:

\begin{description}
\item[$\blacktriangleright \alpha_x = \alpha$]
Just as in the proof of property~(\ref{cofinite-ascent}) of Theorem~\ref{ctbl-ascent-thm}.

\item[$\blacktriangleright  \alpha_x < \alpha$]
In this case we have identified a sequence $\langle \alpha_x^m \mid m < \omega \rangle$ cofinal in $\alpha$,
so we fix some $m < \omega$ such that $\beta < \alpha_x^m$.
For each natural number $n < m$, let
\[
F_n = \{ \iota < \theta \mid f_{x, u \restriction \alpha_x^n}(\iota) <_T f_{x, u \restriction \alpha_x^m}(\iota) \}.
\]
Also let
\[
G = \{ \iota < \theta \mid f_{x, u \restriction \beta}(\iota) <_T f_{x, u \restriction \alpha_x^m}(\iota) \}.
\]
Applying the induction hypothesis to $\alpha_x^m$,
we have $F_n \in \mathcal F^{\nu}_{\theta}$ for all $n < m$,
and also $G \in \mathcal F^{\nu}_{\theta}$.
Define
\[
F = G \cap \bigcap_{n < m} F_n .
\]
If there exists some $m < \omega$ and some $\iota^* < \theta$ such that
$\psi (\alpha_x^k) = \iota^*$ whenever $m < k < \omega$,
then replace $F$ with $F \setminus \{ \iota^* \}$.
(There can be at most one such $\iota^*$.)

Clearly, $F \in \mathcal F^{\nu}_{\theta}$ as it is the intersection of finitely many sets from that filter.
Now, fix any $\iota \in F$, and we will show that $f_{x, u \restriction \beta} (\iota) <_T f_{x, u} (\iota)$:
First, by construction of $F$, it is not the case that there exists some $m<\omega$ such that
$\psi (\alpha_x^k) =\iota$ whenever $m<k<\omega$.

Then, for every $n<m$, we have $\iota \in F \subseteq F_n$, so that
$f_{x, u \restriction \alpha_x^n} (\iota) <_T f_{x, u \restriction \alpha_x^m} (\iota)$.
As $\pred{f_{x, u \restriction \alpha_x^m}(\iota)}$ is  a linear order and
$f_{x, u \restriction \alpha_x^n}(\iota) \in T_{\alpha_n}$ for each $n$,
it follows that $\langle f_{x, u \restriction \alpha_x^n}(\iota) \mid  n \leq m\rangle$ is $<_T$-increasing.
Thus $m$ satisfies the defining property for membership in the set whose supremum is $m_\iota$, so
it follows that $m \leq m_\iota$.
We also have $\iota \in G$, so that $f_{x, u \restriction \beta} (\iota) <_T f_{x, u \restriction \alpha_x^m} (\iota)$.
Putting everything together, we have
\[
f_{x, u \restriction \beta} (\iota) <_T f_{x, u \restriction \alpha_x^m} (\iota) \leq_T
    \begin{cases}
        f_{x, u \restriction \alpha_x^{m_\iota}} (\iota) <_T
                \mathbf b^\alpha_{f_{x, u \restriction \alpha_x^{m_\iota}} (\iota)},
                &\text{if } m_\iota < \omega;           \\
        \bigcup \{ f_{x, u \restriction  \alpha_x^n}(\iota) \mid n<\omega \},
                &\text{otherwise,}
    \end{cases}
\]
and the right-hand side of the above equation is equal to $f_{x, u} (\iota)$. So, we are done.
\end{description}

\item [(\ref{order2-nacc-theta-many-tree})]
Consider any $x \in T \restriction \alpha$.
In this case,
$\alpha_x < \alpha$, and
there exists some $m<\omega$ such that
$\psi (\alpha_x^k) =\iota^*$ whenever $m<k<\omega$,
so that $f_{x, u}(\iota^*)$ was defined to be equal to $\mathbf b^\alpha_{f_{x, u \restriction \alpha_x}(\iota^*)}$.

\end{itemize}
\end{description}

Now we let
\[
T = \bigcup_{\alpha < \kappa} T_\alpha.
\]

\begin{claim}\label{Souslin-claim-theta-ascent}
The tree $(T,{\subset})$ is a $\kappa$-Souslin tree.
\end{claim}
\begin{proof}
Let $A \subseteq T$ be a maximal antichain, and
we will show that $A \subseteq T \restriction \alpha$ for some $\alpha < \kappa$.
By Subclaim \ref{sc621}, for every ordinal $\iota < \theta$, the set
\[
A^\iota = A_{\phi^{-1}(\iota)}=\{ \beta<\kappa \mid \psi(\beta) = \iota \text{ and }
    A\cap(T\restriction\beta)=S_\beta\text{ is a maximal antichain in }T\restriction\beta
\}
\]
is stationary. In particular, $A^\iota \cap \acc(\kappa)$
is cofinal in $\kappa$.
Thus we can apply the last part of the proxy principle to the sequence
$\langle A^\iota \cap \acc(\kappa) \mid \iota < \theta \rangle$
to obtain a limit ordinal $\alpha$ with $\theta < \alpha < \kappa$
such that for every $\iota < \theta$,
\[
\sup \{ \beta \in C_\alpha  \mid \suc_\omega (C_\alpha \setminus \beta) \subseteq A^\iota \cap \acc(\kappa) \}
    = \alpha.
\]

\begin{subclaim}\label{alpha-in-Gamma}
$\alpha \in \Gamma$.
\end{subclaim}

\begin{proof}
By the choice of $\alpha$, we know that $\alpha>\theta$ is a limit ordinal, and we can find $\beta \in C_\alpha$ such that
$\suc_\omega (C_\alpha \setminus \beta) \subseteq A^0\cap\acc(\kappa)$.
As $\suc_\omega (C_\alpha \setminus \beta)\s\nacc(C_\alpha)$, we infer that the latter contains a limit ordinal,
and so Claim~\ref{Gamma-includes}(2) gives $\alpha \in \Gamma$.
\end{proof}

To see that the antichain $A$ is a subset of $T \restriction \alpha$,
consider any $v' \in T_\alpha$,
and we will find some $y \in A\cap(T\restriction \alpha)$ compatible with $v'$.

\begin{subclaim}\label{get-b}
There are some $\iota < \theta$,
${\bar\alpha} \in \acc(C_\alpha) \cup \{ \alpha \} \subseteq \Gamma$,
and ${x} \in T \restriction C_{{\bar\alpha}}$ such that
\[
\mathbf b^{{{\bar\alpha}}}_{x} \leq_T v' \text{ and }
\sup \left(\nacc (C_{{\bar\alpha}}) \cap A^\iota \right) = {\bar\alpha}.
\]
\end{subclaim}

\begin{proof}
By Subclaim~\ref{alpha-in-Gamma} we have $\alpha \in \Gamma$, so that
by property~(\ref{tree-Gamma}) there are two possibilities to consider:

\begin{description}%[style=unboxed]
\item[$\blacktriangleright v' = \mathbf b^\alpha_{x}$ for some ${x} \in T \restriction C_\alpha$]

In this case,
fix such an $x$, set ${\bar\alpha} = \alpha$ and the subclaim is satisfied for any choice of $\iota < \theta$.

\item[$\blacktriangleright v' = f_{x, u} (\iota)$ for some $x \in T \restriction \alpha$, $u \in U_\alpha$, and $\iota < \theta$]
Fix such $x$, $u$, and $\iota$.
By our choice of $\alpha$, let us pick $\epsilon \in C_\alpha$ with $\max\{\theta,\height (x)\} < \epsilon$ such that
$\suc_\omega (C_\alpha \setminus \epsilon) \subseteq A^\iota$.
Let
\[
{\bar\alpha} = \sup \left( \suc_\omega (C_\alpha \setminus \epsilon) \right).
\]
Clearly ${\bar\alpha} \in \acc(C_\alpha) \cup \{ \alpha \}$.
Since $\alpha \in \Gamma$, property~(\ref{tree-ascent-on-acc-theta})(b) gives us $f_{x, u \restriction {\bar\alpha}} (\iota) \leq_T f_{x, u}(\iota)$.
Also,
it follows from Claim~\ref{Gamma-works}(1) that
${\bar\alpha} \in \Gamma$ and $C_{{\bar\alpha}} = C_\alpha \cap {\bar\alpha}$,
so that
\[
C_{{\bar\alpha}} \setminus (\epsilon +1) = (C_\alpha \cap {\bar\alpha}) \setminus (\epsilon+1)
        = \suc_\omega (C_\alpha \setminus \epsilon) \subseteq A^\iota,
\]
and also
\[
\sup(\acc(C_{{\bar\alpha}})) = \sup(\acc(C_\alpha \cap {\bar\alpha})) \leq \epsilon < {\bar\alpha}.
\]
For every $\beta \in A^\iota$ we have $\psi(\beta) = \iota$.
Thus
\[
\psi \left[ C_{{\bar\alpha}} \setminus (\epsilon+1) \right] = \psi \left[ \suc_\omega (C_\alpha \setminus \epsilon) \right]
    = \psi \left[ A^\iota \right] = \{ \iota \},
\]
so that by applying property~(\ref{order2-nacc-theta-many-tree}) to ${\bar\alpha}$,
we must have $f_{x, u \restriction {\bar\alpha}} (\iota) = \mathbf b^{{{\bar\alpha}}}_{x}$
for some ${x} \in T \restriction C_{{\bar\alpha}}$.
Fix such an $x$. It follows that
\[
\mathbf b^{{{\bar\alpha}}}_{x} = f_{x, u \restriction {\bar\alpha}}(\iota) \leq_T f_{x, u} (\iota) = v'.
\]
Notice that
\[
{\bar\alpha} = \sup \left( \suc_\omega (C_\alpha \setminus \epsilon) \cap A^\iota \right)
    = \sup \left( \nacc (C_{{\bar\alpha}}) \cap A^\iota \right),
\]
giving the required conclusion.
\qedhere
\end{description}
\end{proof}

We now fix $\iota$, $\bar\alpha$, and $x$ as in Subclaim~\ref{get-b}.
Then we can find some $y \in A$ such that $y <_T \mathbf b^{\bar\alpha}_{x}$,
just as in the proof of Subclaim~\ref{point-in-antichain},
replacing $A_0$ with $A^\iota$.

Altogether, we have
\[
y <_T \mathbf b^{{\bar\alpha}}_{x} \leq_T  v'.
\]
Thus we have shown that
every $v' \in T_\alpha$ extends some element $y$ of the antichain $A$,
so it follows that
$A \subseteq T \restriction \alpha$.
Consequently,
$\left|A\right| \leq \left| T \restriction \alpha \right| < \kappa$,
and this shows that there are no antichains of size $\kappa$ in the splitting tree  $(T,{\subset})$.
Therefore, it is $\kappa$-Souslin.
\end{proof}

Property~(\ref{prolific}) guarantees that the tree $(T,{\subset})$ is prolific,
while property~(\ref{slim}) guarantees that it is slim.

Let $f_{\emptyset, \emptyset} : \theta \to \{\emptyset\}$ be the constant function.
Then properties~(\ref{tree-successor-ascent-injection-theta})(b),
(\ref{tree-cobounded-ascent-theta}) and~(\ref{theta-distinct-ascent}) guarantee that
$\langle f_{\emptyset, u} \mid u \in U \rangle$ forms an injective
$(\mathcal F^{\nu}_\theta, U)$-ascent path through $(T,{\subset})$.
\end{proof}

Applying Theorem~\ref{NEW-theta-ascent-thm}
to the special case $\nu=\aleph_0$ and $U=\bigcup_{\alpha<\kappa}{}^\alpha1$,
we obtain:

\begin{corollary}\label{cor52}
Suppose that $\theta<\kappa=\cf(\kappa)$ are any infinite cardinals,
and $\p_{14}(\kappa,  {\sq}, \theta, \{\kappa\})$ holds.

Then there exists a prolific slim $\kappa$-Souslin tree with an injective $\mathcal F^{\fin}_{\theta}$-ascent path.
\end{corollary}

We now turn to the $\chi$-complete counterpart of Theorem \ref{NEW-theta-ascent-thm}.

\begin{theorem} \label{tree-theta-ascent-complete-thm}
Suppose that $\nu < \kappa$ are regular infinite cardinals, $\theta,\chi < \kappa$ are infinite cardinals,
and $\p_{14}(\kappa,  {\sqsubseteq_{\nu}}, \theta, \{E^\kappa_{\ge\chi}\})$ holds.

Suppose that $\lambda^{<\chi} < \kappa$ for all $\lambda < \kappa$,
and $U \subseteq {}^{<\kappa}\kappa$ is a given downward-closed $\kappa$-tree.

Then, there exists  a prolific $\chi$-complete $\kappa$-Souslin tree that admits an injective $(\mathcal F^{\nu}_\theta,U)$-ascent path.
\end{theorem}
\begin{proof}
Let $\langle C_\alpha \mid \alpha < \kappa \rangle$ be a witness to
$\p^-_{14}(\kappa,  {\sqsubseteq}_{\nu}, \theta, \{E^\kappa_{\ge\chi}\})$.
Define $\Gamma = \{\alpha \in\acc(\kappa)\setminus(\theta+1)\mid \left( \forall \beta \in \acc(C_\alpha) \right) C_\beta \sqsubseteq C_\alpha\}$.
By recursion over $\alpha < \kappa$, construct
the levels $\langle T_\alpha \mid \alpha < \kappa \rangle$ of the tree $T$
as well as the functions $\langle\langle f_{x, u} \mid x\in T\restriction\alpha, u \in U_\alpha \rangle\mid \alpha<\kappa\rangle$
and the nodes $\langle\langle \mathbf b^\alpha_x \mid x \in T \restriction C_\alpha \rangle\mid \alpha \in \Gamma \rangle$
so that after each stage $\alpha$,
properties~(1)--(\ref{width<kappa}) of the construction in Theorem~\ref{tree-ctbl-ascent-complete-thm}
together with properties~(\ref{b-theta-ascent})--(\ref{order2-nacc-theta-many-tree}) and~(\ref{theta-distinct-ascent})
of the construction in Theorem~\ref{NEW-theta-ascent-thm}
are satisfied, as well as the following:

\begin{enumerate}

\setcounterref{enumi}{order2-nacc-theta-many-tree}

\item \label{tree-Gamma-and-large}
If $\alpha \in \Gamma \cap E^\kappa_{\geq\chi}$,
then
\[
T_\alpha = \{ \mathbf b^\alpha_x \mid x \in T \restriction C_\alpha \} \cup
                \{ f_{x, u} (\iota) \mid x \in T \restriction \alpha, u \in U_\alpha, \iota < \theta \};
\]

\item \label{tree-large-not-Gamma}
If $\alpha \in E^\kappa_{\geq\chi} \setminus \Gamma$, then
\[
T_\alpha = \{ f_{x, u} (\iota) \mid x \in T \restriction \alpha, u \in U_\alpha, \iota < \theta \};
\]

\setcounterref{enumi}{theta-distinct-ascent}

\item \label{full-small-ascent}
If $\alpha \in \acc(\kappa) \cap E^\kappa_{<\chi}$,
then the limit of every branch through $T \restriction \alpha$ is a node in $T_\alpha$.
\end{enumerate}

\smallskip

By now, it should be clear that the resulting tree $T=\bigcup_{\alpha<\kappa}T_\alpha$ will
admit an injective $(\mathcal F^{\nu}_\theta,U)$-ascent path.
The proof that $(T,{\subset})$ is $\kappa$-Souslin is the outcome of modifying the proof of Claim \ref{Souslin-claim-theta-ascent}
in the same way that the proof of Claim \ref{Souslin-claim-ctbl-ascent} was modified in Claim \ref{claim431}.
Finally, the fact that $(T,{\subset})$ is $\chi$-complete is exactly what is provided by property~(\ref{full-small-ascent}) of the recursion.
\end{proof}

Applying Theorem~\ref{tree-theta-ascent-complete-thm}
to the special case $\nu\le \cf(\theta)$ and $U=\bigcup_{\alpha<\kappa}{}^\alpha1$,
we obtain:

\begin{corollary}\label{theta-ascent-complete-thm}
Suppose that $\theta,\chi < \kappa$ are infinite cardinals,
$\lambda^{<\chi} < \kappa$ for all $\lambda < \kappa$,
 $\nu\le \cf(\theta) < \kappa$ are regular cardinals,
and $\p_{14}(\kappa,  {\sqsubseteq_{\nu}}, \theta, \{E^\kappa_{\ge\chi}\})$ holds.

Then there exists a prolific $\chi$-complete $\kappa$-Souslin tree with an injective $\mathcal F^{\bd}_{\theta}$-ascent path.
\end{corollary}

\section{Free Souslin Trees with Ascent Paths}
\label{section:free}

In this section, we shall present constructions of $(\chi,\eta)$-free $\kappa$-Souslin trees from $\p_{14}(\kappa,\mathcal R,\theta,\mathcal S)$,
where  $\theta=\kappa$.
Therefore, we note that by \cite{axioms}, for every uncountable cardinal $\lambda$,
$\sd_\lambda$ entails $\p_{14}(\lambda^+,{\sq},\lambda^+,\{ E^{\lambda^+}_{\cf(\lambda)}\})$:
\begin{itemize}
\item By \cite{MR3347468}, for $\lambda$ regular, $\p_{14}(\lambda^+,{\sq},\lambda^+,\{ E^{\lambda^+}_{\lambda}\})$ does not imply $\diamondsuit(E^{\lambda^+}_\lambda)$,
let alone $\sd_\lambda$;
\item By \cite{MR3354439}, for $\lambda$ singular, $\sd_\lambda$ is equivalent to
$\square_\lambda+\ch_\lambda$.
\end{itemize}

To motivate the arithmetic hypotheses in the statement of the  theorems of this section,
we point out that by Lemma \ref{lemma85} below, if there exists a $\chi$-free $\kappa$-Souslin tree,
then $\lambda^{<\chi}<\kappa$ for all $\lambda<\kappa$.

\begin{theorem}\label{theta-ascent-free-thm1}
Suppose that $\kappa$ is any regular uncountable cardinal, $\p_{14}(\kappa,  {\sqsubseteq}, \kappa, \{E^\kappa_{\ge\chi}\})$ holds,
and $\lambda^{<\chi}<\kappa$ for all $\lambda<\kappa$.

If $\theta$ is an infinite cardinal and $\theta^+<\chi$,
then there exists  a prolific slim $(\chi,\theta^+)$-free $\kappa$-Souslin tree with an injective $\mathcal F^{\fin}_{\theta}$-ascent path.
\end{theorem}
\begin{proof} We commence, using $\diamondsuit(H_\kappa)$,
by fixing the
functions $\phi:\kappa\leftrightarrow H_\kappa$, $\psi:\kappa\rightarrow H_\kappa$,
sequences $\langle S_\beta \mid \beta<\kappa\rangle$, $\langle R_i  \mid i < \kappa \rangle$,
well-ordering $<_\phi$,
notation $\beta(T)$,
and the functions $\sealantichain:H_\kappa\times H_\kappa\rightarrow H_\kappa$ and
$\free: H_\kappa \times H_\kappa\times H_\kappa\rightarrow H_\kappa$
as described in Section~\ref{section:basicsetup}.
Let $\langle C_\alpha \mid \alpha < \kappa \rangle$ be a witness to
$\p^-_{14}(\kappa,  {\sqsubseteq}, \kappa, \{E^\kappa_{\ge\chi}\})$.
Without loss of generality, $C_\alpha=C_\alpha\setminus\{0\}$ for all $\alpha<\kappa$.

As always, the tree $T$ will be a downward-closed subset of $^{<\kappa} \kappa$,
so that each level $T_\alpha$ will be a subset of $^\alpha \kappa$,
and the tree relation $\leq_T$ will simply be extension of sequences.
We will construct,
simultaneously by recursion over $\alpha < \kappa$,
the levels $\langle T_\alpha \mid \alpha < \kappa \rangle$ of the tree $T$,
as well as the functions $\langle f_{\alpha} \mid \alpha<\kappa \rangle$
and the nodes
$\langle\langle \mathbf b^\alpha_x \mid x \in T \restriction C_\alpha \rangle\mid \alpha \in \acc(\kappa)\rangle$,
so that after each stage $\alpha$ of the construction,
properties~(1)--(\ref{slim}) of the construction in Theorem~\ref{ctbl-ascent-thm}
are satisfied, as well as:

\begin{enumerate}
\setcounterref{enumi}{b-simple}
\addtocounter{enumi}{-1}
\item \label{b-ascent-free}
If $\alpha$ is a limit ordinal, then
for every $x \in T \restriction C_\alpha$,
$\mathbf b^\alpha_x \in T_\alpha$ is the limit of the
increasing, continuous, cofinal sequence $b^\alpha_x$
in $( T \restriction \alpha, \subseteq)$,
satisfying the following properties:

\begin{enumerate}
\item
$\dom (b^\alpha_x) = C_\alpha \setminus \height(x)$;
\item
$b^\alpha_x (\height(x)) =x$;
\item
For all $\beta \in \dom(b^\alpha_x)$, $b^\alpha_x (\beta) \in T_\beta$;
\item
If $\beta^-<\beta$ are two consecutive points in $\dom(b^\alpha_x)$,
then
\[
b_x^\alpha(\beta) = \begin{cases}
    \free ( b_x^\alpha(\beta^-), T\restriction(\beta),\sigma_{\beta}^\alpha),
                &\text{if } \psi(\beta)\in {}^{<\kappa}(T\restriction ((C_\alpha \cup \{0\}) \cap \beta^-))\setminus\{\emptyset\};  \\
    \sealantichain( b_x^\alpha(\beta^-), T\restriction(\beta)),    &\text{otherwise,}
\end{cases}
\]
\end{enumerate}
where $\sigma_{\beta}^\alpha:\dom(\psi(\beta))\rightarrow T\restriction \{0,\beta^-\}$ is defined by stipulating:\label{61sigmabetaalpha}
\[
\sigma_{\beta}^\alpha(\xi)=\begin{cases}
b^\alpha_{\psi(\beta)(\xi)}(\beta^-),&\text{if } \psi(\beta)(\xi)\neq\emptyset;\\
\emptyset,&\text{otherwise.}
\end{cases}
\]

\item \label{simple-theta-ascent-is-function}\label{simple-theta-ascent-injection}
$f_\alpha : \theta \to T_\alpha$ is a function. Moreover:

\begin{enumerate}
\item
If $\alpha \leq \theta$, then $f_\alpha$ is a constant function;
\item
If $\alpha > \theta$,
then $f_\alpha$ is injective;
\end{enumerate}

\item \label{cofinite-ascent-theta}
For every $\beta < \alpha$,
\[
\{ \iota < \theta \mid f_{\beta}(\iota) <_T f_{\alpha}(\iota) \} \in \mathcal F^{\fin}_{\theta};
\]

\item \label{ascent-on-acc-theta-strong}
\begin{enumerate}
\item
If $\alpha \leq \theta$, then for every $\beta < \alpha$,
\[
\{ \iota < \theta \mid f_{\beta}(\iota) <_T f_{\alpha}(\iota) \} = \theta;
\]
\item
If $\alpha$ is a limit ordinal and $\beta \in \acc (C_\alpha)$, then
\[
\{ \iota < \theta \mid f_{\beta}(\iota) <_T f_{\alpha}(\iota) \} = \theta;
\]
\end{enumerate}

\item \label{order2-nacc-theta}
If $\alpha > \theta$ is  a limit ordinal such that $\sup (\acc(C_\alpha)) < \alpha$,
and if $\iota^* < \theta$ is such that, for some $\beta < \alpha$, $\psi[C_\alpha \setminus \beta] = \{ \iota^* \}$,
then
there is some $x \in T \restriction C_\alpha$ such that
\[
f_\alpha (\iota^*) = \mathbf b^\alpha_x;
\]

\item \label{limit-theta-ascent}
If $\alpha$ is a limit ordinal, then
\[
T_\alpha = \{ \mathbf b^\alpha_x \mid x \in T \restriction C_\alpha \} \cup \{ f_\alpha (\iota) \mid \iota < \theta \}.
\]

\end{enumerate}

The following instance of the Coherence Claim from page \pageref{coherence-template} shows how we will ensure
that the sequences described in property~(\ref{b-ascent-free}) can always be constructed:

\begin{claim}\label{coherence-free}
Fix limit ordinals $\bar\alpha < \alpha < \kappa$ such that $C_{\bar\alpha} = C_\alpha \cap\bar\alpha$.
Suppose $T \restriction \alpha$ has been constructed to satisfy the above properties.
Suppose also that, for every $y \in T \restriction C_{\bar\alpha}$,
$b^\alpha_y \restriction (C_{\bar\alpha} \setminus \height(y))$ has already been constructed.
Then for every $x \in T \restriction C_{\bar\alpha}$,
\[
b^{\bar\alpha}_x = b^\alpha_x \restriction \left( C_{\bar\alpha} \setminus \height(x) \right).
\]
\end{claim}
\begin{proof}
First, notice that for every $x \in T \restriction C_{\bar\alpha}$, the sequence on each side of the equation has domain
$C_{\bar\alpha} \setminus \height(x)$.
Next, we will prove, by induction over $\beta \in C_{\bar\alpha}$,
that for every $x \in T \restriction (C_{\bar\alpha} \cap (\beta+1))$, $b^{\bar\alpha}_x(\beta) = b^\alpha_x(\beta)$.
Fix $\beta \in C_{\bar\alpha}$, and suppose that for every $x \in T \restriction (C_{\bar\alpha} \cap \beta)$
we have shown that
$b^{\bar\alpha}_x \restriction (C_{\bar\alpha} \cap \beta \setminus \height(x)) =
b^\alpha_x \restriction (C_{\bar\alpha} \cap \beta \setminus \height(x))$.
We must show that $b^{\bar\alpha}_x(\beta) = b^\alpha_x(\beta)$
for every $x \in T \restriction (C_{\bar\alpha} \cap (\beta+1))$.

Consider the nodes $x \in T_\beta$ separately.
For such $x$, property~(\ref{b-ascent-free})(b) gives
\[
b^{\bar\alpha}_x (\beta) = x = b^\alpha_x(\beta).
\]

It remains to check all nodes $x \in T \restriction (C_{\bar\alpha} \cap \beta)$.
For all such $x$, we have $C_{\bar\alpha} \cap \beta \setminus \height(x) \neq \emptyset$.
We now consider two cases:

\begin{description}
\item[$\blacktriangleright \beta \in \nacc({C_{\bar\alpha}})$]
Since $C_{\bar\alpha}$ is a club, $\beta$ must have an immediate predecessor in $C_{\bar\alpha} \setminus \height(x)$,
which we call $\beta^-$.
Then $\beta^-<\beta$ are two consecutive points in $\dom (b^{\bar\alpha}_x)$, and also in $\dom (b^\alpha_x)$.

\begin{description}
\item [$\blacktriangleright \blacktriangleright$ Suppose
    $\psi(\beta) \in {}^{<\kappa} (T\restriction ((C_\alpha \cup \{0\}) \cap \beta^-)) \setminus\{\emptyset\}$]
Recall that $C_{\bar\alpha} \cap \beta = C_\alpha \cap \beta$ and $C_{\bar\alpha} \cap \beta^- = C_\alpha \cap \beta^-$.
Since the definition of $\sigma_{\beta}^\alpha$ depends only on $\psi(\beta)$ and on values of $b^\alpha_y(\beta^-)$
for $y\in T\restriction(C_\alpha\cap\beta^-)=T\restriction(C_{\bar\alpha}\cap\beta^-)$,
it follows from the induction hypothesis that $\sigma^\alpha_\beta=\sigma^{\bar\alpha}_\beta$.
Consequently, for every $x \in T \restriction (C_{\bar\alpha} \cap \beta)$,
by property~(\ref{b-ascent-free})(d) and the induction hypothesis,
we have:
\[
b^{\bar\alpha}_x (\beta)
    = \free ( b_x^{\bar\alpha}(\beta^-), T\restriction(\beta+1), \sigma_\beta^{\bar\alpha} )
    = \free ( b_x^\alpha(\beta^-), T\restriction(\beta+1), \sigma_\beta^\alpha )
        = b^\alpha_x (\beta).
\]

\item [$\blacktriangleright \blacktriangleright$ Otherwise]
For every $x \in T \restriction (C_{\bar\alpha} \cap \beta)$,
by property~(\ref{b-ascent-free})(d) and the induction hypothesis, we have,
\[
b^{\bar\alpha}_x (\beta) = \sealantichain ( b^{\bar\alpha}_x (\beta^-), T \restriction (\beta+1))
                = \sealantichain ( b^\alpha_x (\beta^-), T \restriction (\beta+1) )
                = b^\alpha_x (\beta).
\]
\end{description}

\item[$\blacktriangleright \beta \in \acc({C_{\bar\alpha}})$]
Then also $\beta \in \acc (C_\alpha)$,
and by continuity of each sequence and application of the induction hypothesis,
we have, for each $x \in T \restriction (C_{\bar\alpha} \cap \beta)$,
\[
b^{\bar\alpha}_x(\beta)
    = \bigcup \{ b^{\bar\alpha}_x (\delta)  \mid \delta \in C_{\bar\alpha} \cap \beta \setminus \height(x) \}
    = \bigcup \{ b^\alpha_x (\delta) \mid \delta \in C_{\bar\alpha} \cap \beta \setminus \height(x) \}
    = b^\alpha_x (\beta).
\qedhere
\]
\end{description}
\end{proof}

The recursive construction proceeds as follows:

\begin{description}
\item[Base case, $\alpha = 0$]
As always, let
$T_0 = \{ \emptyset \}$.
Define $f_0 : \theta \to \{\emptyset\}$ by setting $f_0(\iota) = \emptyset$ for all $\iota < \theta$.
The required properties are automatically satisfied as there is nothing to check.

\item[Successor ordinal, $\alpha = \beta+1$]
As in Theorem~\ref{ctbl-ascent-thm}, define
\[
T_\alpha = \{ s^\smallfrown \langle \iota \rangle \mid s \in T_\beta, \iota< \max \{\omega, \alpha\} \}.
\]
Then, for the sake of slimness, define $f_\alpha : \theta \to T_\alpha$ by setting,
for $\iota < \theta$,
\[
f_\alpha(\iota) =
    \begin{cases}
        f_\beta(\iota)^\smallfrown \langle 0 \rangle,    &\text{if } \alpha < \theta; \\
        f_\beta(\iota)^\smallfrown \langle\iota\rangle, &\text{if } \theta < \alpha < \kappa.
    \end{cases}
\]
The required properties are easy to verify.

\item[Limit level]
We begin by constructing $\mathbf b^\alpha_x$ for each $x \in T \restriction C_\alpha$.

For each $x \in T \restriction C_\alpha$,
we will use $C_\alpha$ to determine a cofinal branch through $(T \restriction \alpha,\s)$, containing $x$,
by defining an increasing, continuous sequence $b^\alpha_x$
of nodes.
The domain of the sequence $b^\alpha_x$ will be
${C_\alpha} \setminus \height(x)$.

We
define the values $b^\alpha_x (\beta)$ of the sequences
by recursion over $\beta \in {C_\alpha}$,
\emph{simultaneously} for all $x \in T \restriction {C_\alpha}$,
where for every
$\beta \in \dom(b^\alpha_x)$ we will have $b^\alpha_x(\beta) \in T_\beta$.
So we fix $\beta \in {C_\alpha}$, and we assume that
for every $x \in T \restriction {C_\alpha}$ we have already defined
$b^\alpha_x \restriction (\dom(b^\alpha_x) \cap \beta)$.
We must define the value of $b^\alpha_x(\beta)$ for all $x$ such that $\beta \in \dom(b^\alpha_x)$,
that is, for all $x \in T \restriction ({C_\alpha} \cap (\beta+1))$:

\begin{description}
\item[$\blacktriangleright x \in T_\beta$]
We take care of these nodes separately, because for these nodes we have $\dom(b^\alpha_x) \cap \beta = \emptyset$,
so that the sequence is just starting here.
Let $b^\alpha_x(\beta) = x$.

It remains to define $b^\alpha_x(\beta)$ for all $x \in T \restriction ({C_\alpha} \cap \beta)$.
For all such $x$ we have $\dom(b^\alpha_x) \cap \beta \neq \emptyset$.

\item[$\blacktriangleright \beta \in \nacc({C_\alpha})$]
Let $\beta^-$ denote the predecessor of $\beta$ in ${{C_\alpha}}$.
For every $x \in T \restriction ({C_\alpha} \cap \beta)$, we have $ \beta^- \in \dom(b^\alpha_x) \cap \beta$,
so that $b^\alpha_x (\beta^-)$ has already been defined.

The definition of $b_x^\alpha(\beta)$ splits into two possibilities:

\begin{description}
\item [$\blacktriangleright \blacktriangleright \psi(\beta) \in
{}^{<\kappa} (T\restriction ((C_\alpha \cup \{0\}) \cap \beta^-)) \setminus \{\emptyset\}$]
Define $\sigma_\beta^\alpha$ as in page \pageref{61sigmabetaalpha}, and then
for every $x \in T \restriction (C_\alpha \cap \beta)$ let
\[
b_x^\alpha (\beta)=\free
    ( b_x^\alpha (\beta^-), T\restriction(\beta+1), \sigma_\beta^\alpha).
\]

Since $b_x^\alpha (\beta^-)$ belongs to the normal tree $T\restriction (\beta+1)$,
we get from the Extension Lemma (page \pageref{extendfact}) that
$b_x^\alpha (\beta)$ is an element of $T_\beta$ extending $b_x^\alpha (\beta^-)$.

\item [$\blacktriangleright \blacktriangleright$ Otherwise]
For every $x \in T \restriction ({C_\alpha} \cap \beta)$, let
\[
b_x^\alpha(\beta)=\sealantichain\left( b_x^\alpha(\beta^-), T\restriction(\beta+1) \right).
\]
In this case, we get from the Extension Lemma that
$b_x^\alpha(\beta)$ is an element of $T_\beta$ extending $b_x^\alpha(\beta^-)$.

\end{description}

\item[$\blacktriangleright \beta \in \acc({C_\alpha})$]
In this case, for every $x \in T \restriction ({C_\alpha} \cap \beta)$
we define as in previous proofs
\[
b^\alpha_x(\beta) = \bigcup \{ b^\alpha_x(\gamma) \mid \gamma \in \dom(b^\alpha_x) \cap \beta \}.
\]

It is clear that $b^\alpha_x(\beta) \in {}^\beta \kappa$,
and in fact we have $b^\alpha_x (\beta) \in T_\beta$
using Claim~\ref{coherence-free},
just as explained in the proof of Theorem~\ref{ctbl-ascent-thm}.

\end{description}

We then let
\[
\mathbf b^\alpha_x = \bigcup \{ b^\alpha_x (\beta) \mid \beta \in \dom(b^\alpha_x) \},
\]
and it is clear that $\mathbf b^\alpha_x \in {}^{\alpha}\kappa$.

Next, we shall define a function $f_\alpha : \theta \to {}^\alpha\kappa$, by considering several cases:

\begin{description}
\item[$\blacktriangleright \alpha \leq \theta$]
By property~(\ref{simple-theta-ascent-injection})(a) of the induction hypothesis,
each $f_\beta$ for $\beta < \alpha$ is a constant function,
and by property~(\ref{ascent-on-acc-theta-strong})(a) the sequence consisting of their constant values,
$\langle f_\beta(0) \mid \beta < \alpha \rangle$, is increasing and cofinal in $(T \restriction \alpha, \subseteq)$.
Thus, we can define the constant function $f_\alpha : \theta \to {}^\alpha \kappa$ by setting,
for all $\iota < \theta$,
\[
f_\alpha(\iota) = \bigcup \{ f_\beta(0) \mid \beta < \alpha \}.
\]

\item[$\blacktriangleright \alpha > \theta$]
Fix $\iota < \theta$, and we must prescribe a function value $f_\alpha (\iota) \in {}^\alpha \kappa$.
Let
\[
\alpha_0 = \sup \left( \acc(C_\alpha) \cup \{ \min \left( C_\alpha \setminus (\theta+1) \right) \} \right).
\]
It is clear from the definition that $\theta < \alpha_0 \leq \alpha$,
and that $\alpha_0 = \sup(\acc(C_\alpha))$ iff $\acc(C_\alpha) \setminus (\theta+1) \neq\emptyset$.
Again, the definition of $f_\alpha(\iota)$ splits into two possibilities:

\begin{description}
\item [$\blacktriangleright \alpha_0 = \alpha$]
By replacing $n$ with $\iota$ in the proof of Claim~\ref{full-domination-acc},
we see that the sequence $\langle f_\beta(\iota) \mid \beta \in \acc(C_\alpha) \rangle$ is
increasing. As $\sup(\acc(C_\alpha))=\alpha$, it is also cofinal in $( T \restriction \alpha, \subseteq )$.
Therefore, as in the construction of Theorem~\ref{ctbl-ascent-thm}, let
\[
f_\alpha (\iota) = \bigcup\{ f_\beta (\iota)\mid \beta \in \acc (C_\alpha)\}.
\]

\item [$\blacktriangleright \theta < \alpha_0 < \alpha$]
Enumerate $C_\alpha  \setminus \alpha_0$ as an increasing sequence
$\langle \alpha_m \mid m < \omega \rangle$ cofinal in $\alpha$.

This time, to define $f_{\alpha}(\iota)$,
we consider the following possibilities:

\begin{itemize}
\item
If there exists some $m<\omega$ such that
$\psi (\alpha_k) =\iota$ whenever $m<k<\omega$,
then let
\[
f_{\alpha}(\iota) = \mathbf{b}^\alpha_{f_{\alpha_0}(\iota)}.
\]
Of course, $\alpha_0 \in C_\alpha$,
thus $f_{\alpha_0}(\iota) \in T_{\alpha_0} \subseteq T \restriction C_\alpha$,
and hence $f_\alpha(\iota) \in {}^\alpha \kappa$ is well-defined.

\item
Otherwise, consider the ordinal
\[
m_\iota=\sup \left\{ m<\omega \mid \langle f_{\alpha_n}(\iota) \mid n\le m \rangle\text{ is $<_T$-increasing} \right\},
\]
and let
\[
f_{\alpha}(\iota)=
    \begin{cases}
        \mathbf{b}^\alpha_{f_{\alpha_{m_\iota}}(\iota)},    &\text{if } m_\iota<\omega; \\
        \bigcup \{ f_{\alpha_n}(\iota) \mid n<\omega \},        &\text{if } m_\iota=\omega.
    \end{cases}
\]
If $m_\iota < \omega$, then $\alpha_{m_\iota} \in C_\alpha$, thus
$f_{\alpha_{m_\iota}}(\iota) \in T_{\alpha_{m_\iota}} \subseteq T \restriction C_\alpha$,
and $f_\alpha(\iota) \in {}^\alpha \kappa$ is well-defined.
If $m_\iota = \omega$, then
$f_{\alpha} (\iota) \in {}^\alpha \kappa$,
as it is the limit of a cofinal branch through $(T \restriction \alpha,\s)$.
\end{itemize}

\end{description}

Having constructed $f_\alpha(\iota)$, we now have the following variant of Claim~\ref{a0-goes-right},
proven in the same way as Claim~\ref{a0-goes-right-theta-many-tree}:

\begin{claim} \label{a0-goes-right-theta}
If $\alpha_0 < \alpha$ then
$f_{\alpha_0} (\iota) <_T f_\alpha(\iota)$.
\end{claim}

\end{description}

Finally, as promised, we set
\[
T_\alpha = \{ \mathbf b^\alpha_x \mid x \in T \restriction C_\alpha \} \cup \{ f_\alpha (\iota) \mid \iota < \theta \}.
\]

To verify some of the required properties:
\begin{itemize}
\item[(\ref{normal})]
Just as in Theorem~\ref{ctbl-ascent-thm}.

\item[(\ref{ascent-on-acc-theta-strong})]
\begin{itemize}
\item[(a)]
By construction of $f_\alpha$ in the case $\alpha \leq \theta$.
\item[(b)]
Assuming $\alpha > \theta$:
Just as in Theorem~\ref{ctbl-ascent-thm}, replacing $\omega$ with $\theta$ and $n$ with $\iota$.
\end{itemize}

\item[(\ref{simple-theta-ascent-injection})]
\begin{itemize}
\item[(a)]
If $\alpha \leq \theta$, then $f_\alpha$ was constructed to be a constant function;
\item[(b)]
Just as in Theorem~\ref{ctbl-ascent-thm}, replacing $n_1 < n_2 < \omega$ with $\iota_1 < \iota_2 < \theta$,
and making sure (in case $\alpha_0 = \alpha$) to choose $\beta > \theta$.
\end{itemize}

\item[(\ref{cofinite-ascent-theta})]
If $\alpha \leq \theta$, then this is covered by property~(\ref{ascent-on-acc-theta-strong})(a).
So we assume $\alpha > \theta$,
and then the proof is just as in case $\alpha \in \Gamma$ of Theorem~\ref{NEW-theta-ascent-thm}.

\item [(\ref{slim})]
Just as in Theorem~\ref{ctbl-ascent-thm},
we have $\left|T \restriction C_\alpha \right| \leq \left|\alpha\right|$.
We show that $\left|\{ f_\alpha (\iota) \mid \iota < \theta \}\right|\leq \left|\alpha\right|$,
by considering two cases:

\begin{description}
\item[$\blacktriangleright \alpha \leq \theta$]
In this case,
property~(\ref{tree-successor-ascent-injection-theta})(a) tells us that $f_{\alpha}$ is a constant function,
so that
\[
\left|\range(f_{\alpha})\right|  = 1 < \left|\alpha\right|.
\]

\item[$\blacktriangleright \alpha > \theta$]
In this case, we have
\[
\left|\range(f_{\alpha})\right| \leq \left|\dom(f_{\alpha})\right| = \theta \leq \left|\alpha\right|.
\]
\end{description}

In both cases, we then have
\[
\left|T_\alpha \right| \leq
    \left|T \restriction C_\alpha \right| + \left|\range(f_{\alpha})\right|
    \leq \left| \alpha \right| + \left| \alpha \right|
    = \left|\alpha\right|,
\]
as required.

\item [(\ref{order2-nacc-theta})]
In this case,
$\alpha_0 < \alpha$, and
there exists some $m<\omega$ such that
$\psi (\alpha_k) =\iota^*$ whenever $m<k<\omega$,
so that $f_{\alpha}(\iota^*)$ was defined to be equal to $\mathbf b^\alpha_{f_{\alpha_0}(\iota^*)}$.

\end{itemize}

\end{description}

Now let $T=\bigcup_{\alpha<\kappa}T_\alpha$.
Having built the tree, we now investigate its properties.

\begin{claim}\label{claim612}
The tree $(T,{\subset})$ is $(\chi, \theta^+)$-free.
\end{claim}

\begin{proof}
Fix any nonzero ordinal
$\tau < \chi$, any level $\delta < \kappa$, and
consider any sequence of distinct nodes
\[
\langle w_\xi \mid \xi < \tau \rangle \in {}^\tau T_\delta.
\]
Let
\[
\hat T = \bigotimes_{\xi <\tau}  \cone{w_\xi}
\subseteq T^\tau
\]
be the derived tree determined by this sequence.

Each level of $\hat T$ consists of $\tau$-sequences of elements from the corresponding level of $T$.
The number of such sequences is
\[
\left|T_\alpha \right| ^{\left|\tau\right|} \leq \left|T_\alpha\right|^{<\chi} < \kappa,
\]
using $\tau < \chi $ and the arithmetic hypothesis in the statement of the theorem.
Thus each level of $\hat T$ has size $< \kappa$.

Consider any $A \subseteq \hat T$ with $\left|A\right| = \kappa$.
We need to find $\vec y, \vec z \in A$ such that
$\left| \{ i<\tau\mid \neg( \vec y(i)<_T\vec  z(i))\} \right| < \theta^+$.

An application of $\diamondsuit(H_\kappa)$, nearly identical to the proof of Subclaim \ref{sc621}, entails that
for each $i<\kappa$,
\[
A_i = \{ \beta \in R_i \mid
        A \cap {}^\tau (T \restriction \beta) = S_\beta
        \text{ is a maximal antichain in }\hat T \cap {}^\tau (T \restriction \beta)
\}
\]
is a stationary set.
Thus we can apply the last part of the proxy principle to the sequence
$\langle A_i \mid i < \kappa \rangle$
to obtain a stationary set
\[
W = \{\alpha \in E^\kappa_{\geq \chi} \mid
    \forall i < \alpha \left[\sup \{ \beta \in C_\alpha \mid
                    \suc_\omega (C_\alpha \setminus \beta) \subseteq A_i \}
                = \alpha \right] \}.
\]
As $Z = \{\beta < \kappa \mid     {}^{<\chi} ( T \restriction \beta ) \subseteq \phi[\beta] \}$
 is a club relative to $E^\kappa_{\geq \chi}$, let us fix an ordinal $\alpha \in W \cap Z$ with $\alpha > \delta$.

Since $\left|A\right| = \kappa$ and each level of $\hat T$ has size $< \kappa$,
we can clearly choose some
$\vec v = \langle v_\xi \mid \xi < \tau \rangle \in A \setminus {}^\tau (T \restriction \alpha)$.

For each $\xi<\tau$,
we have $\height(v_\xi) \geq \alpha$, so we can let
$v'_\xi = v_\xi \restriction \alpha \in T_\alpha$,
so that $v'_\xi \leq_T v_\xi$.
Define
\[
I = \{ \xi < \tau \mid v'_\xi = \mathbf b^\alpha_{x_\xi} \text{ for some } x_\xi \in T \restriction C_\alpha \}.
\]

By $\alpha\in E^\kappa_{\ge\chi}$ and the construction of $T_\alpha$ (in particular, property~(\ref{limit-theta-ascent})), we know that
\[
\{ v'_\xi \mid \xi \in \tau \setminus I \} \subseteq \{ f_\alpha (\iota) \mid \iota < \theta \},
\]
so that
$\left| \tau \setminus I \right| \leq \theta < \theta^+$.

\begin{subclaim}\label{sc6121}
There is some $\vec y \in A \cap {}^\tau (T \restriction \alpha)$ such that
$\vec y (\xi) <_T v'_\xi$ for all $\xi \in I$.
\end{subclaim}

\begin{proof}
For every $\xi \in I$,
fix $x_\xi \in T \restriction C_\alpha$ such that $v'_\xi = \mathbf b^\alpha_{x_\xi}$.
Then, for all $\xi \in I$ we must have
\[
v'_\xi = \mathbf{b}^\alpha_{x_\xi}=\bigcup  \{ b^\alpha_{x_\xi}(\beta) \mid \beta \in \dom(b^\alpha_{x_\xi}) \}.
\]
For $\xi \in \tau \setminus I$, simply let $x_\xi = \emptyset$.

By $\cf(\alpha) \geq \chi > \tau$ and $\sup (C_\alpha) = \alpha$,
fix a large enough $\gamma \in C_\alpha$ such that
$\{ x_\xi \mid \xi<\tau \}\s T \restriction \gamma$.

Note also that since $\vec v \in \hat T$, we have that $v'_\xi$ is compatible with $w_\xi$ for all $\xi<\tau$,
so it follows that $x_\xi$ is compatible with $w_\xi$ for all $\xi < \tau$.
Let $\vec x = \langle x_\xi \mid \xi < \tau \rangle$.
Evidently,
$\vec x \in \hat T\cap{}^\tau (T \restriction ((C_\alpha \cup \{0\}) \cap \gamma))$.
Put
\[
i = \phi^{-1} (\vec x).
\]
By $\alpha \in Z$, we have $i < \alpha$.

Since $i<\alpha$, and $\alpha\in W$, we have
$\sup (\nacc(C_\alpha) \cap A_i) = \alpha$.
So let us fix $\beta \in \nacc(C_\alpha) \cap A_i$ with
$\max\{\gamma,\delta\} \leq \sup(C_\alpha\cap\beta)<\beta < \alpha$.
Since $\beta \in A_i$,
we know that $S_\beta = A \cap {}^\tau (T \restriction \beta)$
is a maximal antichain in $\hat T \cap {}^\tau (T \restriction \beta)$.

Write $\beta^-=\sup(C_\alpha\cap\beta)$ and
$\bar T=T\restriction(\beta+1)$.
Define $\vec b$  by setting, for every $\xi < \tau$,
\[
\vec b(\xi) =
    \begin{cases}
        b^\alpha_{x_\xi}(\beta^-)   &\text{if } \xi \in I;  \\
        \emptyset               &\text{otherwise.}
    \end{cases}
\]
Notice that $\vec b(\xi) = \emptyset$ iff $x_\xi = \emptyset$ iff $\xi \notin I$.

By $\beta\in A_i \subseteq R_i$, we have
\[
\psi(\beta) = \phi(i) = \vec x
\in\hat T\cap {}^\tau(T \restriction ((C_\alpha \cup \{0\}) \cap \gamma)) \s
{}^{<\kappa}(T\restriction ((C_\alpha \cup \{0\}) \cap \beta^-))\setminus\{\emptyset\}.
\]
Consequently, $\vec{b}=\sigma_\beta^\alpha$, and for all $x\in T\restriction(C_\alpha\cap\beta)$, we have
\[
b_x^\alpha(\beta) = \free( b_x^\alpha(\beta^-), T\restriction(\beta+1), \vec{b}).
\]

Therefore, consider the set
\[
Q= \left\{ \vec z\in {}^\tau(\bar T\cap{}^{\beta(\bar T)}\kappa) \mid
    \exists \vec y\in S_{\beta(\bar T)} \cap {}^\tau \bar T
    \left(   \forall \xi<\tau\  \vec b(\xi)\cup \vec y(\xi)\s \vec z(\xi)
    \right) \right\}.
\]
Clearly, $Q$ is equal to the set
\[
\left\{ \vec z\in {}^\tau(T_\beta) \mid
    \exists \vec y\in A\cap{}^\tau(T\restriction\beta)\ \forall \xi<\tau\left(\vec b (\xi) \cup \vec y(\xi)\s \vec z(\xi)\right) \right\}.
\]

Since $A\cap{}^\tau(T\restriction\beta)$ is a maximal antichain in $\hat T \cap {}^\tau (T \restriction \beta)$,
and $\vec b\in\hat T$, we get from the normality of $T\restriction (\beta+1)$ that the set $Q$ must be non-empty.
Let $\vec z=\min(Q,<_\phi)$.

For every $\xi \in \tau \setminus I$,  $\vec b(\xi) = \emptyset$.
For every $\xi \in I$,
$\vec b(\xi) = b^\alpha_{x_\xi}(\beta^-) <_T \mathbf b^\alpha_{x_\xi} = v'_\xi$,
and by $\delta\le\beta^-$, we also have  $w_\xi \le_T x_\xi \le_T b^\alpha_{x_\xi}(\beta^-)$.
Since $\langle w_\xi\mid \xi<\tau\rangle$ is a sequence of distinct nodes of $T_\delta$,
we get that for all $\xi \in I$, $\{\xi'<\tau\mid \vec b (\xi) = \vec b (\xi') \}$ is equal to the singleton $\{\xi\}$.
Consequently, for all $\xi \in I$:
\[
b_{x_\xi}^\alpha(\beta)=\vec z(\xi)
\]

Let $\vec y \in A \cap {}^\tau (T \restriction \beta)$ be a witness to the choice of $\vec z$.
Altogether, for all $\xi  \in I$ we have
\[
\vec y (\xi) <_T \vec z(\xi) =  b^\alpha_{x_\xi} (\beta) <_T \mathbf b^\alpha_{x_\xi} = v'_\xi,
\]
as required.
\end{proof}
So, $\vec{y},\vec{v}\in A$ and $|\{ \xi<\tau\mid \neg( \vec y(\xi)<_T\vec  v(\xi))\}|\le|\tau\setminus I|<\theta^+$, as sought.
\end{proof}

\begin{claim}
The tree $(T,{\subset})$ is a $\kappa$-Souslin tree.
\end{claim}
\begin{proof} We briefly go over the proof of Claim~\ref{Souslin-claim-ctbl-ascent}, verifying that all arguments go through.

Let $A\s T$ be a given maximal antichain.
Consequently, for all $\iota<\theta$, the following set is stationary:\[
A^\iota = \{ \beta<\kappa \mid \psi(\beta) = \iota \text{ and }
    A\cap(T\restriction\beta)=S_\beta\text{ is a maximal antichain in }T\restriction\beta
\}.
\]

Hence, we may fix a large enough limit ordinal $\alpha<\kappa$ such that for every $\iota < \theta$,
\[
\sup \{ \beta \in C_\alpha  \mid \suc_\omega (C_\alpha \setminus \beta) \subseteq A^\iota \} = \alpha.
\]

To see that $A \subseteq T \restriction \alpha$,
consider any $v' \in T_\alpha$, and we will find some $y \in A\cap(T\restriction\alpha)$ compatible with $v'$.

Just like in Subclaim \ref{get-b}, an application of property~(\ref{order2-nacc-theta})
entails  $\iota < \theta$,
$\bar\alpha\in \acc(C_\alpha) \cup \{ \alpha \}$, and $x\in T \restriction C_{\bar\alpha}$ such that
\[
\mathbf b^{{\bar\alpha}}_{x} \leq_T v' \text{ and }
\sup \left(\nacc (C_{\bar\alpha}) \cap A^\iota \right) = \bar\alpha.
\]

\begin{subclaim}
There is some $y \in A$ such that $y <_T \mathbf b^{{\bar\alpha}}_{x}$.

In particular, there exists some $y\in A$ such that $y<_T v'$.
\end{subclaim}

\begin{proof}
Fix $\beta \in \nacc(C_{\bar\alpha}) \cap A^\iota$ with $\height(x) < \beta < \bar\alpha$.
Of course, $\beta \in \dom (b^{{\bar\alpha}}_{x})$, and by the construction of $\mathbf b^{{\bar\alpha}}_{x}$,
we know that $b^{{\bar\alpha}}_{x}(\beta) <_T \mathbf b^{{\bar\alpha}}_{x}$.
Since $\beta \in A^\iota$,  we have:
\begin{itemize}
\item $S_\beta = A \cap (T \restriction \beta)$ is a maximal antichain in $T \restriction \beta$;
\item  $\psi(\beta)=\iota$. In particular,
$\psi(\beta)$ is not an element of ${}^{<\kappa}(T\restriction ((C_\alpha \cup \{0\}))\setminus\{\emptyset\}$.
\end{itemize}
Referring back to the construction of $b^{{\bar\alpha}}_{x}(\beta)$,
by $\beta \in \nacc( {C_{\bar\alpha}}) \setminus (\height(x) +1)$,
we have
\[
b_{x}^{{\bar\alpha}}(\beta)=\sealantichain ( b_{x}^{{\bar\alpha}}(\beta^-), T\restriction (\beta+1) ).
\]
So a verification identical to that of the proof of Claim \ref{point-in-antichain} entails some $y\in S_\beta=A \cap (T \restriction \beta)$, such that
\[
y <_T b^{{\bar\alpha}}_{x}(\beta) <_T \mathbf b^{{\bar\alpha}}_{x}.
\]

As $\mathbf b^{{\bar\alpha}}_{x} \leq_T v'$, we have, in particular, $y<_T v'$.
\end{proof}

Thus we have shown that
every $v' \in T_\alpha$ extends some element $y$ of the antichain $A$.
That is, $A \subseteq T \restriction \alpha$. As $A$ was an arbitrary maximal antichain, the proof is complete.
\end{proof}

Properties~(\ref{simple-theta-ascent-injection})(b)
and~(\ref{cofinite-ascent-theta})
guarantee that
$\langle f_{\alpha} \mid \alpha < \kappa \rangle$ is an
injective
$\mathcal F^{\fin}_\theta$-ascent path through $(T,{\subset})$.
\end{proof}
\begin{theorem}\label{theta-ascent-free-thm2}
Suppose that $\kappa$ is any regular uncountable cardinal, $\chi<\kappa$,
$\lambda^{<\chi}<\kappa$ for all $\lambda<\kappa$, and $\p_{14}(\kappa,  {\sqsubseteq}, \kappa, \{E^\kappa_{\ge\chi}\})$ holds.

Then there exists  a prolific slim $\chi$-free $\kappa$-Souslin tree.
\end{theorem}
\begin{proof} Compared to the proof of Theorem \ref{theta-ascent-free-thm1},
we carry the very same construction
of $\langle\langle \mathbf b^\alpha_x \mid x \in T \restriction C_\alpha \rangle\mid\alpha \in \acc(\kappa)\rangle$,
and we \emph{do not} define the functions $\langle f_\alpha\mid\alpha<\kappa\rangle$ at all.
Therefore, $T_\alpha=\{{\mathbf b}^\alpha_x\mid x \in T \restriction C_\alpha \}$
for every limit $\alpha<\kappa$.
Consequently, in the proof of Claim \ref{claim612}, the set $I$ which is defined shortly before Subclaim \ref{sc6121} will be nothing but $\tau$.
So $|\tau\setminus I|=0$, and hence the proof of Claim \ref{claim612} establishes that the outcome slim tree is $(\chi,1)$-free.
That is, $(T,{\subset})$ is a slim $\chi$-free $\kappa$-Souslin tree.
\end{proof}

When reading the statement of the next theorem, the reader should keep in mind that
$\p_{14}(\kappa,  {\sqsubseteq}_{\nu}, \ldots)$ entails
$\p_{14}(\kappa,  {\sqsubseteq}_{\mu}, \ldots)$ for all $\nu\le\mu<\kappa$.

\begin{theorem}\label{free-with-u-ascent}
Suppose that $\nu<\kappa$ are any regular  cardinals,  $\eta$ is a cardinal satisfying $\lambda^{<\eta} < \kappa$ for all $\lambda < \kappa$,
and $U\s{}^{<\kappa}\kappa$ is a given
$\p_{14}^-(\kappa,  {\sqsubseteq}_{\nu}, \kappa, \{E^\kappa_{\ge\eta}\})$-respecting $\kappa$-tree.

Let $\chi = \min \{ \eta, \nu\}$.
If $\diamondsuit(\kappa)$ holds,
then for every infinite $\theta<\kappa$,
there exists a prolific $\chi$-free $\eta$-complete $\kappa$-Souslin tree with an injective $(\mathcal F^{\nu}_\theta,U)$-ascent path.
\end{theorem}
\begin{proof}
By $\diamondsuit(\kappa)$ and Fact \ref{fact32},
let us fix the
functions $\phi:\kappa\leftrightarrow H_\kappa$, $\psi:\kappa\rightarrow H_\kappa$,
sequences $\langle S_\beta \mid \beta<\kappa\rangle$, $\langle R_i  \mid i < \kappa \rangle$,
well-ordering $<_\phi$,
notation $\beta(T)$,
and the functions $\defaultaction:H_\kappa\times H_\kappa\rightarrow H_\kappa$ and
$\free: H_\kappa \times H_\kappa\times H_\kappa\rightarrow H_\kappa$
as described in Section~\ref{section:basicsetup}.

By passing to an isomorphic tree if necessary,
we may assume that for all $u \in U$ and all $\beta < \dom(u)$,
$u(\beta) < \left| U_{\beta+1} \right|$.

Let $\S\s\kappa$ and $\langle \mathbf b^\alpha : (U \restriction C_\alpha) \to {}^\alpha\kappa\cup \{\emptyset\} \mid \alpha<\kappa \rangle$
witness the fact that $U$ is $\p_{14}^-(\kappa,  {\sqsubseteq}_{\nu}, \kappa, \{E^\kappa_{\ge\eta}\})$-respecting.
In particular, $\overrightarrow C=\langle C_\alpha \mid \alpha < \kappa \rangle$ is a witness to
$\p_{14}^-(\kappa,  {\sqsubseteq_{\nu}}, \kappa, \{\S\cap E^\kappa_{\ge\eta}\})$.
Without loss of generality,
we may assume that  $C_\alpha=C_\alpha\setminus(\theta+1)$ whenever $\theta<\alpha<\kappa$.

\begin{claim} \label{Gamma-includes2}\label{Gamma-works2}
Define $\Gamma = \{\alpha \in \acc(\kappa) \setminus(\theta+1)\mid \left( \forall \beta \in \acc(C_\alpha) \right) C_\beta \sqsubseteq C_\alpha\}$. Then:
\begin{enumerate}
\item If $\alpha\in\Gamma$ and $\bar\alpha\in\acc(C_\alpha)$, then ${\bar\alpha}\in\Gamma$;
\item $\Gamma \supseteq \{ \alpha\in \acc(\kappa) \setminus(\theta+1)\mid \otp(C_\alpha) \geq \nu \text{ or $\nacc(C_\alpha)$ contains a limit ordinal} \}$.
\end{enumerate}

In particular, $\Gamma$ covers the stationary set $E^\kappa_{\ge\nu}\setminus (\theta +1)$.
\end{claim}
\begin{proof} Just like the proof of Claim \ref{Gamma-includes}.
\end{proof}

For every $\alpha \in \acc(\kappa)$, $\beta\le\alpha$,
and $x \in {}^{<\kappa}\kappa$ with $\dom(x) < \sup(C_\alpha \cap \beta)$,
denote
\[
\Phi^\alpha_\beta(x) = \sup (\acc(C_\alpha \cap \beta) \cup \{ \min((C_\alpha \cap \beta) \setminus (\dom(x)+1))\}).
\]
Notice that:
\begin{itemize}
\item $\Phi^\alpha_\alpha(x)$ coincides with $\alpha_x$ that was defined right after Claim \ref{full-domination-acc-theta-tree} (provided $\dom(x)=\height(x)$, that is, provided $x$ is placed into the tree $T$);
\item $\dom(x) < \Phi^\alpha_\beta(x) \leq \sup(C_\alpha \cap \beta)$.
%\item $\dom(x) < \Phi_\alpha^\beta(x) \leq \beta$;
\item $\Phi^\alpha_\beta(x)\in C_\alpha \cup \{\alpha\}$;
\item $\Phi^\alpha_\beta(x) = \sup(\acc(C_\alpha \cap \beta))$ iff $\dom(x) < \sup(\acc(C_\alpha \cap \beta))$.
\end{itemize}

As always, the tree $T$ will be a downward-closed subset of $^{<\kappa} \kappa$,
so that each level $T_\alpha$ will be a subset of $^\alpha \kappa$,
and the tree relation $\leq_T$ will simply be extension of sequences.
We will construct,
simultaneously by recursion over $\alpha < \kappa$,
the levels $\langle T_\alpha \mid \alpha < \kappa \rangle$ of the tree $T$
as well as the functions $\langle\langle f_{x, u} \mid x\in T\restriction\alpha, u \in U_\alpha \rangle\mid \alpha<\kappa \rangle$
and the nodes $\langle\langle \mathbf b^\alpha_x \mid x \in T \restriction C_\alpha \rangle\mid \alpha \in \Gamma \rangle$
so that after each stage $\alpha$,
properties~(1)--(\ref{full-small-ascent})
of the construction in Theorem~\ref{tree-theta-ascent-complete-thm}
are satisfied, with the following changes to clauses (\ref{b-theta-ascent}) and (\ref{order2-nacc-theta-many-tree}):

\begin{enumerate}
\setcounterref{enumi}{b-theta-ascent}
\addtocounter{enumi}{-1}
\item \label{b-tree-ascent-free}
If $\alpha \in \Gamma$, then for every $x \in T \restriction C_\alpha$,
$\mathbf b^\alpha_x \in T_\alpha$ is the limit of the
increasing, continuous, cofinal sequence $b^\alpha_x$
in $( T \restriction \alpha, \subseteq )$,
satisfying the following properties:
\begin{enumerate}
\item
$\dom (b^\alpha_x) = C_\alpha \setminus \height(x)$;
\item
$b^\alpha_x (\height(x)) =x$;
\item
For all $\beta \in \dom(b^\alpha_x)$, $b^\alpha_x (\beta) \in T_\beta$;
\item
If $\beta^- < \beta$ are two consecutive points in $\dom(b^\alpha_x)$,
then
\[
b_x^\alpha(\beta)=\begin{cases}
    \free ( b_x^\alpha(\beta^-), T\restriction(\beta+1),
        \varsigma^\alpha_{\beta}),
                            &\psi(\beta)\in{}^{<\kappa}(\kappa\times (T\restriction \beta^-) \times (U\restriction \beta^-))\setminus\{\emptyset\};  \\
    \defaultaction \left( b_x^\alpha(\beta^-), T\restriction(\beta+1) \right),    &\text{otherwise},
\end{cases}
\]
where $\varsigma_{\beta}^\alpha:\dom(\psi(\beta))\rightarrow T\restriction \{0,\beta^-\}$ is defined by stipulating:\label{63sigmabetaalpha}
\[
\varsigma_{\beta}^\alpha(\xi)=\begin{cases}
b^\alpha_{\varphi_\alpha(\beta)(\xi)}(\beta^-),&\text{if } \varphi_\alpha(\beta)(\xi)\neq\emptyset;\\
\emptyset,&\text{otherwise,}
\end{cases}
\]
and\label{63varphibetaalpha}
\[
\varphi_\alpha(\beta)(\xi)=\begin{cases}
y,  &\psi(\beta)(\xi) = (\iota,y,z) \in
        (\kappa\setminus\theta) \times (T\restriction(C_\alpha\cap\beta^-)) \times \{\emptyset\};    \\
{f_{y, \mathbf{b}^\alpha(z)\restriction\Phi^\alpha_\beta(y)} (\iota)},
    &\psi(\beta)(\xi)=(\iota,y,z)\in \theta\times (T\restriction \beta^-)\times (U\restriction(C_\alpha\cap\beta^-))
        \ \&\ \mathbf{b}^\alpha(z)\restriction\Phi^\alpha_\beta(y)\in U_{\Phi^\alpha_\beta(y)} ;\\
\emptyset,   &\text{otherwise.}
\end{cases}
\]

\end{enumerate}

\setcounterref{enumi}{order2-nacc-theta-many-tree}
\addtocounter{enumi}{-1}
\item \label{order2-nacc-tree-ascent-free}
If $\alpha \in \Gamma$ satisfies $\sup (\acc(C_\alpha)) < \alpha$,
and if there is some $\beta < \alpha$ and $g\in{}^{<\chi} (H_\kappa \times H_\kappa \times H_\kappa)$
such that $\psi[C_\alpha \setminus \beta] = \{ g \}$,
then for every $\iota < \theta$ such that $(\iota, h, h') \in \range(g)$ for some $h, h'$,
every $x \in T \restriction (\sup(\acc(C_\alpha)))$, and every $u \in U_\alpha$, we have
\[
f_{x, u} (\iota) = \mathbf b^\alpha_{f_{x, u \restriction (\sup(\acc(C_\alpha)))} (\iota)}.
\]
\end{enumerate}

The following instance of the Coherence Claim Template from page \pageref{coherence-template}
utilizes the successful interaction of the tree $(U,\subset)$ with the sequence $\overrightarrow C$.

\begin{claim}\label{coherence-Gamma-free}
Fix ordinals ${\bar\alpha} < \alpha$, both in $\Gamma$, such that $C_{\bar\alpha} = C_\alpha \cap {\bar\alpha}$.
Suppose $T \restriction \alpha$ has been constructed to satisfy the above properties.
Suppose also that, for every $y \in T \restriction C_{\bar\alpha}$,
$b^\alpha_y \restriction (C_{\bar\alpha} \setminus \height(y))$ has already been constructed.
Then for every $x \in T \restriction C_{\bar\alpha}$,
\[
b^{\bar\alpha}_x = b^\alpha_x \restriction \left( C_{\bar\alpha} \setminus \height(x) \right).
\]
\end{claim}

\begin{proof}
Following the proof of Claim~\ref{coherence-free},
the differences arise only when $\beta \in \nacc(C_{\bar\alpha})$,
so we focus on that case.

Since $C_{\bar\alpha}$ is a club, $\beta$ must have an immediate predecessor in $C_{\bar\alpha} \setminus \height(x)$,
which we call $\beta^-$.
Then $\beta^-$ and $\beta$ are two consecutive points in $\dom (b^{\bar\alpha}_x)$, and also in $\dom (b^\alpha_x)$.

\begin{description}
\item [$\blacktriangleright \blacktriangleright$ Suppose
    $\psi(\beta) \in {}^{<\kappa}(\kappa\times (T\restriction \beta^-) \times (U \restriction \beta^-))\setminus\{\emptyset\}$]
Recall that $C_{\bar\alpha} \cap \beta = C_\alpha \cap \beta$ and $C_{\bar\alpha} \cap \beta^- = C_\alpha \cap \beta^-$.
Notice that the definition of $\Phi^\alpha_\beta(x)$ depends only on $\dom(x)$ and $C_\alpha \cap \beta$.
In particular, for any $y \in T \restriction \beta^-$,
\[
\Phi^\alpha_\beta(y) = \Phi^{\bar\alpha}_\beta(y)\leq \sup (C_{\bar\alpha} \cap \beta) = \beta^-
    < \beta < {\bar\alpha} < \alpha,
\]
and then by Clause~(\ref{respectcohere}) of Definition~\ref{respecting}
we have for any $z \in U \restriction C_{\bar\alpha}$:
\[
\mathbf b^{\bar\alpha}(z) \restriction \Phi^{\bar\alpha}_\beta(y)
    = (\mathbf b^\alpha(z) \restriction {\bar\alpha}) \restriction \Phi^{\bar\alpha}_\beta(y)
    = \mathbf b^\alpha(z) \restriction \Phi^\alpha_\beta(y)
\]
and hence $\varphi_{\bar\alpha}(\beta)$ coincides with $\varphi_\alpha(\beta)$.
So, by the induction hypothesis, $\varsigma^{\bar\alpha}_\beta=\varsigma^\alpha_\beta$,
and then for every $x \in T \restriction (C_{\bar\alpha} \cap \beta)$, we have,
using property~(\ref{b-tree-ascent-free})(d) and the induction hypothesis,
\begin{align*}
b^{\bar\alpha}_x (\beta)
    &= \free ( b_x^{\bar\alpha}(\beta^-), T\restriction(\beta+1),\sigma^{\bar\alpha}_\beta)\\
    &= \free ( b_x^\alpha(\beta^-), T\restriction(\beta+1),\sigma^\alpha_\beta)
        = b^\alpha_x (\beta).
\end{align*}

\item [$\blacktriangleright \blacktriangleright$ Otherwise]
For every $x \in T \restriction (C_{\bar\alpha} \cap \beta)$, we have,
again using property~(\ref{b-tree-ascent-free})(d) and the induction hypothesis,
\[
b^{\bar\alpha}_x (\beta) = \defaultaction ( b^{\bar\alpha}_x (\beta^-), T \restriction (\beta+1))
                = \defaultaction ( b^\alpha_x (\beta^-), T \restriction (\beta+1) )
                = b^\alpha_x (\beta).
\qedhere
\]
\end{description}

\end{proof}

The recursive construction proceeds as follows:

\begin{description}
\item[Base case, successor ordinal, and limit $\alpha \notin \Gamma$]
These steps of the construction are identical to the corresponding steps in Theorem~\ref{tree-theta-ascent-complete-thm}.

\item[Limit level, $\alpha \in\Gamma$]
We begin by constructing $\mathbf b^\alpha_x$ for each $x \in T \restriction C_\alpha$.

For each $x \in T \restriction C_\alpha$,
we will use $C_\alpha$ to determine a cofinal branch through $(T \restriction \alpha,\s)$, containing $x$,
by defining an increasing, continuous sequence $b^\alpha_x$ of nodes.
The domain of the sequence $b^\alpha_x$ will be
${C_\alpha} \setminus \height(x)$.

We define the values $b^\alpha_x (\beta)$ of the sequences by recursion over $\beta \in {C_\alpha}$,
\emph{simultaneously} for all $x \in T \restriction {C_\alpha}$,
where for every $\beta \in \dom(b^\alpha_x)$ we will have $b^\alpha_x(\beta) \in T_\beta$.
So we fix $\beta \in {C_\alpha}$, and we assume that for every $x \in T \restriction {C_\alpha}$ we have already defined
$b^\alpha_x \restriction (\dom(b^\alpha_x) \cap \beta)$.
We must define the value of $b^\alpha_x(\beta)$ for all $x$ such that $\beta \in \dom(b^\alpha_x)$,
that is, for all $x \in T \restriction ({C_\alpha} \cap (\beta+1))$:

\begin{description}
\item[$\blacktriangleright x \in T_\beta$]
We take care of these nodes separately, because for these nodes we have $\dom(b^\alpha_x) \cap \beta = \emptyset$,
so that the sequence is just starting here.
Let $b^\alpha_x(\beta) = x$.

It remains to define $b^\alpha_x(\beta)$ for all $x \in T \restriction ({C_\alpha} \cap \beta)$.
For all such $x$ we have $\dom(b^\alpha_x) \cap \beta \neq \emptyset$.

\item[$\blacktriangleright \beta \in \nacc({C_\alpha})$]
Let $\beta^-$ denote the predecessor of $\beta$ in ${{C_\alpha}}$.
For every $x \in T \restriction ({C_\alpha} \cap \beta)$, we have $ \beta^- \in \dom(b^\alpha_x) \cap \beta$,
so that $b^\alpha_x (\beta^-)$ has already been defined.

The definition of $b_x^\alpha(\beta)$ splits into two possibilities:

\begin{description}
\item [$\blacktriangleright \blacktriangleright
    \psi(\beta) \in {}^{<\kappa}(\kappa\times (T\restriction \beta^-) \times (U \restriction \beta^-))\setminus\{\emptyset\}$]
Define $\varsigma_\beta^\alpha$ as in page \pageref{63sigmabetaalpha}, and then
for every $x \in T \restriction ({C_\alpha} \cap \beta)$, let
\[
b_x^\alpha(\beta)=\free( b_x^\alpha(\beta^-), T\restriction(\beta+1),\varsigma^\alpha_\beta).
\]

Since $b_x^\alpha (\beta^-)$ belongs to the normal tree $T\restriction (\beta+1)$,
we get from the Extension Lemma (page \pageref{extendfact}) that
$b_x^\alpha (\beta)$ is an element of $T_\beta$ extending $b_x^\alpha (\beta^-)$.

\item [$\blacktriangleright \blacktriangleright$ Otherwise]
For every $x \in T \restriction ({C_\alpha} \cap \beta)$, let
\[
b_x^\alpha(\beta)=\defaultaction( b_x^\alpha(\beta^-), T\restriction(\beta+1)).
\]
In this case, we get from the Extension Lemma that
$b_x^\alpha(\beta)$ is an element of $T_\beta$ extending $b_x^\alpha(\beta^-)$.

\end{description}

\item[$\blacktriangleright \beta \in \acc({C_\alpha})$]
In this case, for every $x \in T \restriction ({C_\alpha} \cap \beta)$
we define (just as in previous theorems)
\[
b^\alpha_x(\beta) = \bigcup \{ b^\alpha_x(\gamma) \mid \gamma \in \dom(b^\alpha_x) \cap \beta \}.
\]

It is clear that $b^\alpha_x(\beta) \in {}^\beta \kappa$,
and in fact we have $b^\alpha_x (\beta) \in T_\beta$
using Claims~\ref{Gamma-works2} and \ref{coherence-Gamma-free}.
\end{description}

We then let
\[
\mathbf b^\alpha_x = \bigcup \{ b^\alpha_x (\beta) \mid \beta \in \dom(b^\alpha_x) \},
\]
and it is clear that $\mathbf b^\alpha_x \in {}^{\alpha}\kappa$.

Next, we fix $x \in T \restriction \alpha$ and $u \in U_\alpha$, and
we must construct a function $f_{x, u}: \theta\rightarrow{}^\alpha \kappa$.
Fix $\iota < \theta$, and let us prescribe a function value $f_{x, u}(\iota) \in {}^\alpha \kappa$.

For notational simplicity, let us write $\alpha_x$ for $\Phi^\alpha_\alpha(x)$.
Just as in the proof of Theorem~\ref{NEW-theta-ascent-thm},
the definition of $f_{x, u} (\iota)$ splits into two possibilities:

\begin{description}%[style=unboxed]
\item [$\blacktriangleright \alpha_x = \alpha$]

In particular, $\sup(\acc(C_\alpha)) = \alpha$.
The same proof of  Claim~\ref{full-domination-acc-theta-tree} entails that
the sequence $\langle f_{x, u \restriction \beta}(\iota) \mid \height(x)\in \beta \in \acc(C_\alpha)  \rangle$ is
increasing, and by $\sup(\acc(C_\alpha)) = \alpha$, it is cofinal in $( T \restriction \alpha, \subseteq )$,
so we let
\[
f_{x, u} (\iota) = \bigcup \{ f_{x, u \restriction \beta} (\iota) \mid \height(x)\in \beta \in \acc (C_\alpha)  \}.
\]

Clearly, $f_{x, u} (\iota) \in {}^\alpha \kappa$.

\item [$\blacktriangleright \alpha_x < \alpha$] In this case,
$C_\alpha \setminus\alpha_x$ is an $\omega$-type cofinal subset of $\alpha$.
Let $\langle \alpha^m_x \mid m<\omega \rangle$ denote the
increasing enumeration of $C_\alpha \setminus\alpha_x$,
so that $\alpha_x = \alpha^0_x$.
Let us  define $f_{x, u}(\iota)$
by considering several possibilities:

\begin{itemize}
\item\label{gchi}
If there exist some $m<\omega$,
a function $g\in{}^{<\chi} (H_\kappa \times H_\kappa \times H_\kappa)$, and $h, h' \in H_\kappa$ such that
$(\iota,h, h') \in \range(g)$,
and $\psi (\alpha^k_x) =g$ whenever $m<k<\omega$,
then let
\[
f_{x, u}(\iota)=\mathbf{b}^\alpha_{f_{x, u\restriction \alpha_x} (\iota)}.
\]

\item
Otherwise, consider the ordinal
\[
m_\iota=\sup \left\{ m<\omega \mid \langle f_{x, u \restriction \alpha^n_x}(\iota) \mid n\le m \rangle\text{ is $<_T$-increasing} \right\},
\]
and let
\[
f_{x, u}(\iota)=
    \begin{cases}
        \mathbf{b}^\alpha_{f_{x, u \restriction \alpha_x^{m_\iota}}(\iota)},    &\text{if } m_\iota<\omega; \\
        \bigcup \{ f_{x, u \restriction \alpha^n_x}(\iota) \mid n<\omega \},        &\text{if } m_\iota=\omega.
    \end{cases}
\]

\end{itemize}

In all cases, it is clear that $f_{x, u} (\iota) \in {}^\alpha \kappa$,
as it is the limit of a cofinal branch through $(T \restriction \alpha,\s)$.

\end{description}

\begin{claim} \label{a0-goes-right-theta-free-tree}
If $\alpha_x < \alpha$ then
$f_{x, u \restriction \alpha_x} (\iota) <_T f_{x, u} (\iota)$.
\end{claim}
\begin{proof} As in the proof Claim~\ref{a0-goes-right-theta-many-tree}.\end{proof}

Having constructed $\mathbf b^\alpha_x \in {}^\alpha\kappa$ for every $x \in T \restriction C_\alpha$
and the function $f_{x, u} : \theta \to {}^\alpha\kappa$ for every $x \in T \restriction \alpha$ and every $u \in U_\alpha$,
the decision as to which elements of $^\alpha \kappa$ are included in $T_\alpha$
is exactly as in Theorem~\ref{tree-theta-ascent-complete-thm}:

$$T_\alpha=\begin{cases}
\{ f_{x, u} (\iota) \mid x \in T \restriction \alpha, u \in U_\alpha, \iota < \theta \}\cup\{ \mathbf b^\alpha_x \mid x \in T \restriction C_\alpha \},&\text{if }\alpha\in E^\kappa_{\ge\chi}\cap\Gamma;\\
\{ f_{x, u} (\iota) \mid x \in T \restriction \alpha, u \in U_\alpha, \iota < \theta \},&\text{if }\alpha \in E^\kappa_{\geq\chi} \setminus \Gamma;\\
\{\bigcup b\mid b\text{ is a cofinal branch through }(T\restriction\alpha,\s)\},&\text{if }\alpha \in E^\kappa_{<\chi}.
\end{cases}$$

The required properties are verified just as in previous theorems,
with the exception of:

\begin{itemize}

\item[(\ref{tree-cobounded-ascent-theta})]
Fix $\beta < \alpha$, $x \in T \restriction \beta$, and $u \in U_\alpha$.
Referring back to the construction of $f_{x, u}$, there are two cases to check:

\begin{description}
\item[$\blacktriangleright \alpha_x = \alpha$]
Just as in the proof of property~(\ref{cofinite-ascent}) of Theorem~\ref{ctbl-ascent-thm}.

\item[$\blacktriangleright  \alpha_x < \alpha$]

Define $F \in \mathcal F^{\nu}_\theta$ as in
the proof of the corresponding property in Theorem~\ref{NEW-theta-ascent-thm} (with the obvious substitutions).
This time,
if there exist
some $m<\omega$ and
a function $g\in{}^{<\chi} (H_\kappa \times H_\kappa \times H_\kappa)$  such that
$\psi (\alpha^k_x) =g$ whenever $m<k<\omega$, then replace $F$ with
\[
F \setminus \{ \iota < \theta \mid (\iota, h, h') \in \range(g) \text{ for some } h, h' \}.
\]
The subtracted set has size $\left| \range(g) \right| \leq \dom(g) < \chi \leq \nu$.
Thus we still have $F \in \mathcal F^{\nu}_\theta$.

Now, fix any $\iota \in F$, and we must show that $f_{x, u \restriction \beta} (\iota) <_T f_{x, u} (\iota)$:
First, by construction of $F$, it is not the case that there exist some $m<\omega$,
a function $g\in{}^{<\chi} (H_\kappa \times H_\kappa \times H_\kappa)$, and $h, h' \in H_\kappa$ such that
$\psi (\alpha^n_x) =g$ whenever $m<n<\omega$,
and $(\iota,h, h')\in\range(g)$.
We then proceed as in the proof of property~(\ref{cofinite-ascent-theta}) of Theorem~\ref{NEW-theta-ascent-thm}
to show that $f_{x, u \restriction \beta} (\iota) <_T f_{x, u} (\iota)$,
as required.
\end{description}

\item[(\ref{order2-nacc-tree-ascent-free})]
Suppose that
$\alpha \in \Gamma$ satisfies $\sup (\acc(C_\alpha)) < \alpha$,
and we can fix some $\beta < \alpha$ and $g\in{}^{<\chi} (H_\kappa \times H_\kappa \times H_\kappa)$
such that $\psi[C_\alpha \setminus \beta] = \{ g \}$.
Consider any $\iota < \theta$ such that $(\iota, h, h') \in \range(g)$ for some $h, h'$,
any $x \in T \restriction (\sup(\acc(C_\alpha)))$, and any $u \in U_\alpha$.
Since $\height(x) < \sup(\acc(C_\alpha))$, we have
$\alpha_x = \sup(\acc(C_\alpha))$,
so that we follow the case $\alpha_x < \alpha$ in the definition of $f_{x, u}(\iota)$.
In that case, under the given assumptions,
we have defined $f_{x, u} (\iota)$ to be equal to
$\mathbf{b}^\alpha_{f_{x, u\restriction \alpha^0_x} (\iota)}$,
where $\alpha^0_x = \alpha_x$, so that the result follows.
\end{itemize}
\end{description}

Now, let $T=\bigcup_{\alpha<\kappa}T_\alpha$.

\begin{claim}\label{claim364}
The tree $(T,{\subset})$ is a $\chi$-free $\kappa$-Souslin tree.
\end{claim}
\begin{proof}
Fix any nonzero ordinal
$\tau < \chi$, any level $\delta < \kappa$, and
consider any sequence of distinct nodes
\[
\langle w_\xi \mid \xi < \tau \rangle \in {}^\tau T_\delta.
\]
Let
\[
\hat T = \bigotimes_{\xi < \tau}  \cone{w_\xi} \subseteq T^\tau
\]
be the derived tree determined by this sequence.
We need to show that $(\hat T,<_{\hat T})$ is a $\kappa$-Souslin tree.

The tree $(T,{\subset})$ is clearly a $\kappa$-tree,
and by arithmetic considerations we already met, so is $(\hat T,<_{\hat T})$.
Let $A \subseteq \hat T$ be a maximal antichain.
We will show that $A \subseteq \hat T \restriction \alpha$ for some $\alpha < \kappa$.

Now, for each $i<\kappa$, the following set is stationary:
\[
A_i = \{ \beta \in R_i \mid
        A \cap {}^\tau (T \restriction \beta) = S_\beta
        \text{ is a maximal antichain in }\hat T \cap {}^\tau (T \restriction \beta)
\}.
\]
Thus we can apply the last part of the proxy principle to the sequence
$\langle A_i \cap \acc(\kappa) \mid i < \kappa \rangle$
to obtain a stationary set
\[
W = \{\alpha \in E^\kappa_{\geq \eta}\cap\S \mid
    \forall i < \alpha \left[\sup \{ \beta \in C_\alpha \mid
                    \suc_\omega (C_\alpha \setminus \beta) \subseteq A_i \cap \acc(\kappa) \}
                = \alpha \right] \}.
\]
Since this set is a stationary subset of $E^\kappa_{\ge\eta}$, it must intersect the (relative) club subset of $E^\kappa_{\ge\eta}$:
\[
Z = \{\beta < \kappa \mid     {}^{<\chi} \left( (\theta+1) \times(T \restriction \beta) \times(U \restriction \beta) \right) \subseteq \phi[\beta] \}
\]
in a stationary set.
Thus we can fix an ordinal $\alpha \in W \cap Z$ with $\alpha > \max \{ \delta, \theta \}$.

\begin{subclaim}\label{get-acc}
$\alpha\in\Gamma\cap\S$ and $\sup (\acc (C_\alpha)) = \alpha$.
\end{subclaim}

\begin{proof}
We have $\alpha \in \Gamma$ just as in the proof of Subclaim~\ref{alpha-in-Gamma}. By $\alpha\in W$, we also have $\alpha\in\S$.

To see that $\sup (\acc (C_\alpha)) = \alpha$:
Suppose not.
Let $\alpha_0 = \sup (\acc (C_\alpha))$, and we are assuming $\alpha_0 < \alpha$.
Thus $C_\alpha \setminus \alpha_0$ is a cofinal subset of $\alpha$ of order-type $\omega$.
Then we must be able to fix $\beta_1, \beta_2 \in C_\alpha \setminus \alpha_0$ such that
$\suc_\omega (C_\alpha \setminus \beta_1) \subseteq A_1$ and
$\suc_\omega (C_\alpha \setminus \beta_2) \subseteq A_2$.
But $\suc_\omega (C_\alpha \setminus \beta_1)$ and $\suc_\omega (C_\alpha \setminus \beta_2)$
share a common tail,
while $A_1 \cap A_2 = \emptyset$.  This is a contradiction.
\end{proof}

To see that the maximal antichain  $A$ is a subset of $\hat T \restriction \alpha$,
consider any $\langle v_\xi' \mid \xi < \tau \rangle \in \hat T \cap(T_\alpha)^\tau$.

Since $\alpha \in \Gamma \cap E^\kappa_{\geq\eta}$,
property~(\ref{tree-Gamma-and-large}) tells us that
every node of $T_\alpha$ was constructed in one of two ways:
either it is $\mathbf b^\alpha_x$ for some $x \in T \restriction C_\alpha$,
or it is $f_{x, u} (\iota)$ for some $x \in T \restriction \alpha$, some $u \in U_\alpha$, and some $\iota < \theta$.
In particular, each $v'_\xi$ is of one of these forms.
We now define a vector
$g \in {}^\tau ((\theta+1) \times (T\restriction \alpha) \times (U\restriction C_\alpha) )$
to capture the construction of the components $v'_\xi$.
For each $\xi<\tau$,
we define $g(\xi)$ as follows:
\begin{itemize}
\item If there exists some $x_\xi\in T\restriction C_\alpha$ such that $v_\xi'=\mathbf {b}^\alpha_{x_\xi}$,
then we let $g(\xi)=(\theta,x_\xi, \emptyset)$, where $x_\xi$ is the $<_\phi$-least such element;
\item Otherwise,
there exist some $\iota_\xi <\theta$, $x_\xi\in T\restriction\alpha$ and $u_\xi \in U_\alpha$
such that $v_\xi'=f_{x_\xi, u_\xi}(\iota_\xi)$.
Since $\alpha \in \S$, Clause~(\ref{respectingonto}) of Definition~\ref{respecting}
tells us that we can find some $z_\xi \in U \restriction C_\alpha$ such that $\mathbf b^\alpha(z_\xi) = u_\xi$.
In this case, we let $g(\xi)=(\iota_\xi,x_\xi, z_\xi)$ be the $<_\phi$-least such triple.
\end{itemize}

By $\cf(\alpha) \geq \eta \ge\chi > \tau$ and $\sup (\acc (C_\alpha)) = \alpha$,
fix a large enough $\gamma \in \acc(C_\alpha)$ such that
$\{ x_\xi \mid \xi<\tau \}\s T\restriction\gamma$ and
$\{ z_\xi \mid \xi<\tau \}\s U\restriction\gamma$.
Thus in fact $g \in {}^\tau ((\theta+1) \times (T\restriction \gamma) \times (U\restriction (C_\alpha \cap \gamma)))$.

Put $i=\phi^{-1}(g)$, and it follows from $\alpha \in Z$ that $i < \alpha$.
So, by $\alpha \in W$, there exists some
$\epsilon\in C_\alpha\setminus\max\{\gamma,\delta\}$ such that
$\suc_\omega(C_\alpha\setminus\epsilon)\s A_i$.
Fix such an $\epsilon$, and write ${\bar\alpha}=\sup(\suc_\omega(C_\alpha\setminus\epsilon))$.
By Subclaim~\ref{get-acc} we cannot have ${\bar\alpha} = \alpha$,
and hence ${\bar\alpha}\in\acc(C_\alpha)$.
Since $\alpha \in \Gamma$, we have
${\bar\alpha}\in\Gamma$ and $C_{{\bar\alpha}} = C_\alpha \cap {\bar\alpha}$, so that
\[
C_{{\bar\alpha}} \setminus (\epsilon +1) = (C_\alpha \cap {\bar\alpha}) \setminus (\epsilon+1)
    = \suc_\omega (C_\alpha \setminus \epsilon) \subseteq A_i \subseteq R_i,
\]
and also (using $\gamma \in \acc(C_\alpha)$)
\[
\gamma \leq \sup(\acc(C_{{\bar\alpha}})) = \sup(\acc(C_\alpha \cap {\bar\alpha})) \leq \epsilon < {\bar\alpha} < \alpha.
\]

For every $\beta \in A_i \subseteq R_i$ we have $\psi(\beta) = \phi(i) = g$.
Thus
\[
\psi[C_{{\bar\alpha}} \setminus (\epsilon +1)]
    = \psi [\suc_\omega (C_\alpha \setminus \epsilon)] = \psi[A_i] = \psi[R_i] = \{g\}.
\]

Fix some $\beta \in C_{{\bar\alpha}} \setminus (\epsilon+1) \subseteq \nacc (C_{{\bar\alpha}})$,
and define $\beta^- = \max (C_{{\bar\alpha}} \cap \beta)$.
In particular, $\psi(\beta) = g$, and
we  also have
\[
\gamma \leq \sup(\acc(C_{{\bar\alpha}})) = \sup(\acc(C_\alpha \cap \beta))
    \leq \epsilon \leq {\beta^-} < {\beta} < {\bar\alpha},
\]
and it follows that for every $x \in T \restriction \gamma$,
\[
\Phi^\alpha_\beta(x) = \sup(\acc(C_\alpha \cap \beta)) = \sup(\acc(C_{\bar\alpha})).
\]

\begin{subclaim}\label{complicated}
For every $\xi < \tau$,
we have
\[
\mathbf b^{{{\bar\alpha}}}_{\varphi_{{\bar\alpha}}({{\beta}}) (\xi)} <_T v'_\xi.\footnote{The definition of $\varphi_{{\bar\alpha}}(\beta)$ may be found on page \pageref{63varphibetaalpha}.}
\]
\end{subclaim}
\begin{proof}
Consider any $\xi < \tau$.
We have $\height(x_\xi) < \gamma$, so that $\Phi^\alpha_\beta(x_\xi)=\sup(\acc(C_{{\bar\alpha}}))$.
Now, recalling that $\psi(\beta) = g$,
we consider two possibilities, based on the two parts in the definition of $g(\xi)$:

\begin{itemize}
\item
If
$g(\xi)=(\theta,x_\xi, \emptyset)$, then we must have $x_\xi \in T \restriction (C_\alpha \cap \gamma)$,
and by definition of the function $\varphi_{{\bar\alpha}}({\beta})$ in this case we have
$\varphi_{{\bar\alpha}}({{\beta}}) (\xi) = x_\xi$, and by the Coherence Claim~\ref{coherence-Gamma-free} it follows that
\[
\mathbf b^{{{\bar\alpha}}}_{\varphi_{{\bar\alpha}}({{\beta}}) (\xi)} =
\mathbf b^{{{\bar\alpha}}}_{x_\xi} = b^\alpha_{x_\xi} ({\bar\alpha}) <_T \mathbf b^\alpha_{x_\xi} = v'_\xi.
\]

\item
If
$g(\xi)=(\iota_\xi,x_\xi, z_\xi)$ for $\iota_\xi<\theta$, then
we must have $x_\xi \in T \restriction \gamma$ and $z_\xi \in U \restriction (C_\alpha \cap \gamma)$,
as well as $u_\xi = \mathbf b^\alpha(z_\xi) \in U_\alpha$, let alone $u_\xi\restriction\sup(\acc(C_{\bar\alpha}))\in U_{\sup(\acc(C_{\bar\alpha}))}$.
By definition of the function $\varphi_{{\bar\alpha}}({\beta})$ in this case we have
\[
\varphi_{{\bar\alpha}}({{\beta}}) (\xi) = f_{x_\xi, u_\xi \restriction \sup(\acc(C_{{\bar\alpha}}))}(\iota_\xi).
\]
Since ${\bar\alpha} \in \Gamma$,  $\sup(\acc(C_{{\bar\alpha}})) < {\bar\alpha}$,
$\psi[C_{{\bar\alpha}} \setminus (\epsilon +1)] = \{g\}$,
$(\iota_\xi, x_\xi, z_\xi) \in \range(g)$,
$x \in T \restriction \gamma \subseteq T \restriction (\sup(\acc(C_\alpha)))$,
and $u_\xi \restriction {\bar\alpha} \in U_{{\bar\alpha}}$,
we apply property~(\ref{order2-nacc-tree-ascent-free}) to ${\bar\alpha}$,
obtaining
\[
f_{x_\xi, u_\xi \restriction {\bar\alpha}} (\iota_\xi) =
\mathbf b^{{{\bar\alpha}}}_{f_{x_\xi, u_\xi \restriction (\sup(\acc(C_{{\bar\alpha}})))} (\iota_\xi)}.
\]
Furthermore, since $\alpha \in \Gamma$ and ${\bar\alpha} \in \acc(C_\alpha)$,
property~(\ref{tree-ascent-on-acc-theta}) gives
$f_{x_\xi, u_\xi \restriction {\bar\alpha}}(\iota_\xi) <_T f_{x_\xi, u_\xi}(\iota_\xi)$.
Altogether, it follows that
\[
\mathbf b^{{{\bar\alpha}}}_{\varphi_{{\bar\alpha}}({{\beta}}) (\xi)} =
\mathbf b^{{{\bar\alpha}}}_{f_{x_\xi, u_\xi \restriction (\sup(\acc(C_{{\bar\alpha}})))}(\iota_\xi)} =
    f_{x_\xi, u_\xi \restriction {\bar\alpha}} (\iota_\xi) <_T f_{x_\xi, u_\xi} (\iota_\xi) = v'_\xi.
\]
\end{itemize}
Thus in both cases the subclaim is proven.
\end{proof}

Write $\bar T=T\restriction({\beta}+1)$, and $$\vec b= \langle b^{{{\bar\alpha}}}_{\varphi_{{\bar\alpha}}({\beta})(\xi)}({\beta^-}) \mid \xi<\tau \rangle.$$
By Subclaim~\ref{complicated}, for all $\xi<\tau$, we have
\[
\vec b(\xi) <_T \mathbf b^{{{\bar\alpha}}}_{\varphi_{\bar\alpha}({\beta})(\xi)} <_T v'_\xi.
\]

By $\psi(\beta)=i=\phi^{-1}(g)$, we have $\vec{b}=\varsigma^{{\bar\alpha}}_\beta$,
and so
for all $x\in T\restriction(C_\alpha \cap{\beta})$, we have
\[
b_x^{{{\bar\alpha}}} ({\beta}) = \free ( b_x^{{{\bar\alpha}}} ({\beta^-}), T\restriction({\beta}+1), \vec{b}).
\]

Therefore, consider the set
\[
Q= \left\{ \vec z\in {}^\tau(\bar T\cap{}^{{\beta}(\bar T)}\kappa) \mid
    \exists \vec y\in S_{{\beta}(\bar T)} \cap {}^\tau \bar T
    \left(
        \forall \xi<\tau\  \vec b(\xi)\cup \vec y(\xi)\s \vec z(\xi)
    \right) \right\}.
\]
Since ${\beta} \in A_i$,
$S_{\beta} = A \cap {}^\tau (T \restriction {\beta})$ is a maximal antichain in $\hat T \cap {}^\tau (T \restriction {\beta})$.
So, $Q$ is equal to the set
\[
\left\{ \vec z\in {}^\tau(T_{\beta}) \mid
    \exists \vec y\in A\cap{}^\tau(T\restriction{\beta})\ \forall \xi<\tau\left(\vec b (\xi) \cup \vec y(\xi)\s \vec z(\xi)\right) \right\}.
\]

As $\langle v_\xi' \mid \xi<\tau\rangle\in \hat T$ and ${\beta^-} \geq \epsilon \geq \delta$,
we get that $w_\xi<_T \vec b(\xi)$ for all $\xi<\tau$,
so that $\vec b \in \hat T$.

Since $A\cap(T\restriction{\beta})^\tau$ is a maximal antichain in $\hat T \cap (T \restriction {\beta})^\tau$,
and $\vec b\in\hat T$, we get from the normality of $T\restriction ({\beta}+1)$ that the set $Q$ must be non-empty.
Let $\vec z=\min(Q,<_\phi)$. Since $w_\xi \le_T \vec b(\xi)$ for all $\xi<\tau$,
and $\langle w_\xi\mid \xi<\tau\rangle$ is a sequence of distinct nodes of $T_\delta$,
we get that for all $\xi<\tau$, $\{\xi'<\tau\mid \vec b (\xi) = \vec b (\xi') \}$ is equal to the singleton $\{\xi\}$.

Altogether, for all $\xi<\tau$:
\[
b_{\varphi_{{\bar\alpha}}({\beta})(\xi)}^{{{\bar\alpha}}}({\beta})
    = \free ( b_{\varphi_{{\bar\alpha}}({\beta})(\xi)}^{{{\bar\alpha}}} ({\beta^-}), T\restriction({\beta}+1), \vec b )
    = \vec z(\xi).
\]

Let $\vec y \in A \cap {}^\tau (T \restriction {\beta})$ be a witness to the choice of $\vec z$.
Then for all $\xi  < \tau$ we have
\[
\vec y (\xi) <_T \vec z(\xi) =  b_{\varphi_{{\bar\alpha}}({\beta})(\xi)}^{{{\bar\alpha}}}({\beta})
    <_T \mathbf b_{\varphi_{{\bar\alpha}}({\beta})(\xi)}^{{{\bar\alpha}}} <_T v'_\xi,
\]
as sought.

As $A$ was an arbitrary maximal antichain in $(\hat T,<_{\hat T})$,
we have established that it is a $\kappa$-Souslin tree.
But $(\hat T,<_{\hat T})$ was an arbitrary derived tree of $(T,{\subset})$, so we are done.
\end{proof}

The fact that $(T,{\subset})$ is $\eta$-complete is exactly what is provided by property~(\ref{full-small-ascent}) of the recursion.

Define $f_{\emptyset, \emptyset} : \theta \to \{\emptyset\}$ to be the constant function.
Then properties~(\ref{tree-successor-ascent-injection-theta})(b),
(\ref{tree-cobounded-ascent-theta}) and~(\ref{theta-distinct-ascent}) guarantee that
$\langle f_{\emptyset, u} \mid u \in U \rangle$ is an injective
$(\mathcal F^{\nu}_\theta, U)$-ascent path through $(T,{\subset})$.
\end{proof}

\begin{theorem}\label{thm65}
Suppose that $\cf(\nu)=\nu<\theta^+<\chi<\cf(\kappa)=\kappa$ are infinite cardinals,
$\lambda^{<\chi}<\kappa$ for all $\lambda<\kappa$, and $\p_{14}(\kappa,  {\sqsubseteq}, \kappa, \{E^\kappa_{\ge\chi}\})$ holds.

Then there exists a prolific  $\nu$-free, $(\chi,\theta^+)$-free, $\chi$-complete $\kappa$-Souslin tree with an injective $\mathcal F^{\nu}_\theta$-ascent path.
\end{theorem}
\begin{proof} Let $\langle C_\alpha\mid\alpha<\kappa\rangle$ be a witness to $\p_{14}(\kappa,  {\sqsubseteq}, \kappa, \{E^\kappa_{\ge\chi}\})$.
Without loss of generality, $C_\alpha=C_\alpha\setminus\{0\}$ for all $\alpha<\kappa$.
By recursion over $\alpha < \kappa$,
construct
the levels $\langle T_\alpha \mid \alpha < \kappa \rangle$ of the tree $T$,
as well as the functions $\langle f_{\alpha} \mid \alpha<\kappa \rangle$
and the nodes
$\langle\langle \mathbf b^\alpha_x \mid x \in T \restriction C_\alpha \rangle\mid \alpha \in \acc(\kappa)\rangle$,
in a way that is almost identical to the construction of Theorem \ref{free-with-u-ascent}, but more relaxed in the following senses:
\begin{itemize}
\item[(a)] There, the second parameter of the proxy principle was $\sqsubseteq_\nu$,
and so, for the sake of normality, we constructed an ascent path $\langle f_{x,u}\mid u\in U\rangle$ for each $x\in T$. Here, we work with $\sq$,
and we obtain normality by defining $\mathbf{b}^\alpha_x$ for all $\alpha\in \acc(\kappa)=\Gamma$. Consequently, it suffices to have $\langle f_{x,u}\mid u\in U\rangle$ only for $x=\emptyset$.
\item[(b)] There, we constructed an $(\mathcal F^\nu_\theta,U)$-ascent path, whereas, here, we merely want an $\mathcal F^\nu_\theta$-ascent path.
\end{itemize}

Altogether, we shall only construct $\langle f_u\mid u\in U\rangle$,
where $U=\bigcup_{\alpha<\kappa}{}^\alpha1$, with $\mathbf{b}^\alpha(z)={}^\alpha1$ for all $z$.
Consequently, for every $\alpha\in E^\kappa_{\ge\chi}$, we shall have
$$T_\alpha = \{ \mathbf b^\alpha_x \mid x \in T \restriction C_\alpha \} \cup \{ f_{{}^\alpha1} (\iota) \mid \iota < \theta \},$$
and hence the proof of Claim \ref{claim612} goes through, establishing that $(T,{\subset})$ is $(\chi,\theta^+)$-free.

\begin{itemize}
\item[(c)] In the hypothesis of Theorem \ref{free-with-u-ascent}, $\chi$ was $\le\nu$. Here, we make them equal, and use $\chi$ for a different purpose.
By ``make them equal'', we mean that on property (\ref{order2-nacc-tree-ascent-free}) from page \pageref{order2-nacc-tree-ascent-free}, we restrict our attention to functions $g\in{}^{<\nu} (H_\kappa \times H_\kappa \times H_\kappa)$.
\end{itemize}
Consequently, the proof of Claim \ref{claim364} establishes that $(T,{\subset})$ is $\nu$-free.

\begin{itemize}
\item[(d)] There, the fourth parameter of the proxy principle was $\{E^\kappa_{\ge\eta}\}$ for some $\eta\ge\chi$,
and the outcome tree was $\eta$-complete. Here, we have $\eta=\chi$.
\end{itemize}
Consequently, the outcome tree is $\chi$-complete.
\end{proof}

As made clear in Sections \ref{sectioncountablewidth} and \ref{sectionwiderwidth}, to any construction of a slim tree,
there is a counterpart construction of a complete tree, and vice versa.
In particular, we have the following variations of the preceding theorems.

\begin{theorem}\label{free-with-u-ascent-slim}
Suppose that $\nu<\kappa$ are any regular  cardinals,  $\chi\le\nu$ is a cardinal satisfying $\lambda^{<\chi} < \kappa$ for all $\lambda < \kappa$,
and $U\s{}^{<\kappa}\kappa$ is a given slim
$\p_{14}^-(\kappa,  {\sqsubseteq}_{\nu}, \kappa, \{E^\kappa_{\ge\chi}\})$-respecting $\kappa$-tree.

If $\diamondsuit(\kappa)$ holds,
then for every infinite $\theta<\kappa$,
there exists a prolific slim $\chi$-free $\kappa$-Souslin tree with an injective $(\mathcal F^{\nu}_\theta,U)$-ascent path.
\end{theorem}

\begin{theorem}
Suppose that $\cf(\nu)=\nu<\theta^+<\chi<\cf(\kappa)=\kappa$ are infinite cardinals,
$\lambda^{<\chi}<\kappa$ for all $\lambda<\kappa$, and $\p_{14}(\kappa,  {\sqsubseteq}, \kappa, \{E^\kappa_{\ge\chi}\})$ holds.

Then there exists a prolific slim $\nu$-free, $(\chi,\theta^+)$-free, $\kappa$-Souslin tree with an injective $\mathcal F^{\nu}_\theta$-ascent path.
\end{theorem}

\section{Appendix:  Producing a Binary Tree Similar to Any Given Tree}

In this short section, we analyze a natural process that, for any given $\kappa$-tree $(X,<_X)$,
produces a downward-closed subtree $T$ of ${}^{<\kappa}2$ of great resemblance to the original one.
Note that we do not assume that the tree $(X,<_X)$ is Hausdorff.

\subsection{The process}

Suppose that $\kappa$ is a regular uncountable cardinal.
Let $(X,<_X)$ be a $\kappa$-tree. As $|X_\alpha|<\kappa$ for all $\alpha<\kappa$,
we may recursively find a sequence of injections $\langle \pi_\alpha:X_\alpha\rightarrow\kappa\mid \alpha<\kappa\rangle$
such that for all $\alpha<\beta<\kappa$,
$\sup(\range(\pi_\alpha))<\min(\range(\pi_\beta))$.
Let $\pi=\bigcup_{\alpha<\kappa}\pi_\alpha$.
Note that if $y,z\in X$ and $y<_X z$ in $X$, then $\pi(y)<\pi(z)$.
So for all $\delta<\kappa$ and $x\in X_\delta$, the set of ordinals $[x]=\{ \pi(y)\mid y\in X, y<_X x\}$ has order-type $\delta$.
Fix a bijection $\psi:\kappa\rightarrow\kappa\times\kappa$.  For all $\delta<\kappa$ and $x\in X_\delta$, let

\begin{itemize}
\item $u_x:\delta\rightarrow [x]$ be the order-preserving isomorphism;
\item $t_x:\delta\rightarrow 2$, where $t_x(\beta)=1$ iff there exists $(\alpha,\gamma)\in\delta\times\delta$ such that $\psi(\beta)=(\alpha,\gamma)$ and $u_x(\alpha)=\gamma$.
\end{itemize}

Consider the club $E=\{\delta<\kappa\mid \pi[X\restriction\delta]\s\delta\ \&\ \psi[\delta]=\delta\times\delta\}$,
and let $$T=\{ t_x\restriction\beta\mid \beta\le\delta, \delta\in E, x\in X_\delta\}.$$

Then $T$ is a downward-closed subtree of ${}^{<\kappa}2$. Next, we shall compare $(T,{\subset})$ with $(X,<_X)$.

\subsection{The analysis}

\begin{lemma}\label{bin-pres-extend}
Suppose $x,x'\in X$, and  $x\le_X x'$. Then $u_x\s u_{x'}$.
If moreover, $x,x'\in X\restriction E$, then $t_x\s t_{x'}$.

In particular, $T_\delta=\{ t_x\mid x\in X_\delta\}$ for all $\delta\in E$.
\end{lemma}
\begin{proof}
By transitivity of $<_X$, we have $[x']\supseteq [x]$. Then, by the choice of $\pi$, $[x']$ is an end-extension of $[x]$.
Consequently, $u_{x'}$ is an end-extension of $u_x$. It follows that if $x,x'\in X\restriction E$, then $t_x\subseteq t_{x'}$.

Now, if $\delta\in E$ and $t\in T_\delta$, then there exists some $x'\in X\restriction( E\setminus\delta)$ such that $t=t_{x'}\restriction\delta$. But then, if $x$ is the unique element of $X_\delta$ with $x\le_T x'$,
then $t_x\s t_{x'}$ and $\dom(t_x)=\delta$. That is, $t_x=t$.
\end{proof}

\begin{lemma}
$(T,{\subset})$ is a $\kappa$-tree.
\end{lemma}
\begin{proof}
It is clear that $\{ \alpha \mid T_\alpha \neq \emptyset \} = \sup(E) = \kappa$.
Fix $\beta < \kappa$, and we must show that $\left| T_\beta \right| < \kappa$.
Let $\delta = \min (E \setminus \beta)$.
By construction of $T$, each node in $T_\beta$ is extended by some node in $T_\delta$.
By the previous claim, $|T_\delta|\le |X_\delta|$.
Altogether, $\left| T_\beta \right| \leq \left| T_\delta \right|\le |X_\delta|<\kappa$.
\end{proof}

\begin{lemma}\hfill
\begin{enumerate}
\item If $(X,<_X)$ is normal, then so is $(T,{\subset})$;
\item If $(X,<_X)$ is normal and splitting, then $(T\restriction E,\subset)$ is also splitting.
\end{enumerate}
\end{lemma}
\begin{proof}\hfill
\begin{enumerate}
\item
Consider any $t \in T$ and ordinal $\alpha$ such that $\height_T(t) < \alpha < \kappa$.
We must find some element of $T_\alpha$ extending $t$.
Let $\eta = \height_T(t)$.
By Claim \ref{bin-pres-extend}, there exists $x\in X_{\min(E\setminus\eta)}$ such that $t\s t_x$.
Since $E$ is unbounded in $\kappa$, we choose some $\delta \in E \setminus \max \{\alpha, \height_X(x)\}$.
By normality of $X$, we find some $x' \in X_\delta$ extending $x$.
Then Claim~\ref{bin-pres-extend} gives $t_x \subseteq t_{x'}$.
It follows that $t_{x'} \restriction \alpha \in T_\alpha$, and
$$t = t_x \restriction \eta = t_{x'} \restriction \eta \subseteq t_{x'} \restriction \alpha,$$
as required.

\item
Suppose $t\in T\restriction E$. Write $\alpha=\height_T(t)$, and let $\delta=\min(E\setminus(\alpha+1))$.
By Claim \ref{bin-pres-extend}, there exists $x\in X_{\alpha}$ such that $t=t_x$. Since $(X,<_X)$ is splitting,
let $x_0,x_1\in T_{\alpha+1}$ be two distinct extensions of $x$. Since $(X,<_X)$ is normal, for all $i<2$,
we may pick $y_i\in X_\delta$ that extends $x_i$. Then $t\s t_{y_i}\in T_\delta$ for all $i<2$. Now, to see that $t_{y_0}\neq t_{y_1}$,
write  $\gamma_i=\pi(x_i)$. Since $\alpha+1<\delta\in E$, we have $\gamma_i=\pi_{\alpha+1}(x_i)<\delta$.
Since $\pi$ is injective, $\gamma_0\neq\gamma_1$.
Let $\beta$ be such that $\psi(\beta)=(\alpha+1,\gamma_1)$.
Then $\beta<\delta$ and $t_{y_1}(\beta)=1$, while $t_{y_0}(\beta)=0$.
\qedhere
\end{enumerate}
\end{proof}

\begin{lemma}
If $( X,<_X)$ has no chains of cardinality $\kappa$, then neither does $(T,{\subset})$.
\end{lemma}
\begin{proof}
Suppose $C$ is a chain in $( T,\subset)$ of cardinality $\kappa$.
Let $f = \bigcup C$, so that $f \in {}^\kappa 2$, and consider the set $B=f^{-1}\{1\}$. Then, put $$D=\{ y\in X\mid \exists(\alpha,\beta)\in\kappa\times B[(\alpha,\pi(y))=\psi(\beta)]\}.$$

We first show that $D$ is a chain in $(X,<_X)$. For this, let  $y_0,y_1$ be arbitrary elements of $D$.
For each $i<2$, pick $(\alpha_i,\beta_i)\in\kappa\times B$ such that $(\alpha_i,\pi(y_i))=\psi(\beta_i)$.

Pick $c\in C$ with $\dom(c)>\max\{\alpha_0,\alpha_1\}$. Write $\epsilon=\dom(c)$. By definition of $T$,
there exists $x\in X\restriction (E\setminus\epsilon)$ such that $c\s t_x$.
Recalling the definition of $f$, we have
$f\restriction\epsilon=c=t_x\restriction\epsilon$.
In particular,  $t_x(\beta_i)=1$ for all $i<2$.
By definition of $t_x$, then, we have $u_x(\alpha_i)=\pi(y_i)$ for all $i<2$. In particular, $\pi(y_i)\in[x]$ for all $i<2$.
So $y_i<_X x$ for all $i<2$. As $(X,<_X)$ is a tree,  $y_0$ and $y_1$ are comparable.

Thus, to get a contradiction it suffices to show that the chain $D$ has cardinality $\kappa$.
As $\pi$ and $\psi$ are injective and $\kappa$ is regular, this reduces to showing that $B$ is unbounded in $\kappa$.
Let $\alpha$ be an arbitrary element of $E$. We shall find some $\beta>\alpha$ and $c\in C$ with $c(\beta)=1$.

Let $\delta=\min(E\setminus (\alpha+1))$. Pick $c\in C$ with $\dom(c)\ge\delta$, and then pick $x\in X\restriction E$ such that $c\s t_x$.
Let $y$ denote the unique predecessor of $x$ in level $\alpha$. Then $u_x(\alpha)=\pi(y)$, say it is $\gamma$. Let $\beta$ be such that $\psi(\beta)=(\alpha,\gamma)$.
By $\alpha\in E$, we have $\psi[\alpha]=\alpha\times\alpha$, and hence $\beta>\alpha$.
By $\delta\in E\setminus(\alpha+1)$, we have $\pi[T_\alpha]\s\delta$,
and so from $y\in T_\alpha$, we infer that $\gamma=\pi(y)<\delta$. By $\alpha<\delta$ and $\pi[\delta]=\delta\times\delta$,
we altogether infer that $\alpha<\beta<\delta$. Finally, by $c=t_x\restriction\delta$, we have $c(\beta)=t_x(\beta)=1$,
as indeed $\psi(\beta)=(\alpha,\gamma)$ and $u_x(\alpha)=\gamma$.
\end{proof}

\begin{lemma} Suppose that $S\s\kappa$ is such that
for every antichain $A\s X$, $\{\height_X(x)\mid x\in A\}\cap S$ is nonstationary.

Then, for every antichain $B\s T$,  $\{\height_T(t)\mid t\in B\}\cap S$ is nonstationary.
\end{lemma}
\begin{proof}
Suppose that $B\s T$, and $S'=\{\height_T(t)\mid t\in B\}\cap S$ is stationary.
By Lemma \ref{bin-pres-extend}, for all $\alpha\in S'\cap E$,
we may pick $x_\alpha\in X_\alpha$ such that $\{ t_{x_\alpha}\mid \alpha\in S'\cap E\}\s B$.
As $\{ \height_X(x_\alpha)\mid \alpha\in S'\cap E\}$ is a stationary subset of $S$, the hypothesis
entails  $\alpha<\beta$ in $S'\cap E$ such that $x_\alpha <_X x_\beta$.
Then, by Lemma \ref{bin-pres-extend}, $t_{x_\alpha}\s t_{x_\beta}$. In particular, $B$ is not an antichain in $(T,{\subset})$.
\end{proof}

\begin{lemma}\label{claim-antichain-04}
If $(X,<_X)$ has no antichains of cardinality $\kappa$, then neither does $(T,{\subset})$.
\end{lemma}
\begin{proof}
Suppose $A \subseteq T$ is an antichain of cardinality $\kappa$.
We enumerate $A$ as $\{ t^i \mid i < \kappa \}$.
For each $i<\kappa$,
we must have $t^i = t_{x_i} \restriction \beta_i$ for some $x_i \in U \restriction (E \setminus \beta_i)$,
where $\beta_i = \height_T(t^i)$.
For $i < j < \kappa$, $t^i$ and $t^j$ are incomparable elements of the antichain $A$,
and hence  $t_{x_i}$ and $t_{x_j}$ are incomparable. In particular,  $t^i\mapsto x_i$ is one-to-one,
and  $\{ x_i \mid i<\kappa\}$ has size $\kappa$. As $(X,<_X)$ has no antichain of size $\kappa$,
we may pick $i<j<\kappa$ such that $x_i$ and $x_j$ are comparable.
But then, by Claim~\ref{bin-pres-extend},  $t_{x_i}$ and $t_{x_j}$ are comparable. This is a contradiction.
\end{proof}

\begin{lemma}\label{lemma85} Suppose $(X,<_X)$ is a $(\chi,\eta)$-free $\kappa$-Souslin tree (e.g., $\eta=1)$.

If  $\lambda^{<\eta}<\kappa$ for all $\lambda<\kappa$, then:
\begin{enumerate}
\item $\lambda^{<\chi}<\kappa$ for all $\lambda<\kappa$;
\item  $(T,{\subset})$ is a $(\chi,\eta)$-free $\kappa$-Souslin tree.
\end{enumerate}
\end{lemma}
\begin{proof}
(1) By Lemma \ref{normalandsplitting}, let $D\s E$ be a club such that $(X\restriction D,<_X)$ is normal and splitting.
\begin{claim}\label{claim861} For every $y\in X\restriction D$ and cardinal $\mu<\kappa$,
there exists some $\beta<\kappa$ such that $|X_\gamma\cap\cone{y}|\ge\mu$ whenever $\beta\le\gamma<\kappa$.
\end{claim}
\begin{proof} Let $y$ and $\mu$ be as in the hypothesis.
Since $D$ is a club in $\kappa$, we may choose some $\beta\in D$ such that $\otp(D\cap\beta\setminus\height_X(y))>\mu$.
Pick $z\in X_\beta$ with $y<_X z$.
Since $z_\downarrow$ is linearly ordered by $<_X$,
we may find $\{ y_i\mid i<\mu\}\s X\restriction D$ with $y_0=y$ that is $<_X$-increasing below $z$.
Since $(X\restriction D,<_X)$ is splitting, for all $i<\mu$, let us pick $x_i\in X_{\height(y_{i+1})}$ that extends $y_i$ and is distinct from  $y_{i+1}$.
Then $\{ x_i\mid i<\mu\}\s X\restriction (D\cap\beta)$ is an antichain above $y$. Finally, given $\gamma<\kappa$ with $\gamma\ge\beta$, by normality of $(X\restriction D,<_X)$, we may pick $\{ z_i\mid i<\mu\}\s X_{\gamma}$
such that $x_i<_X z_i$ for all $i<\mu$. In particular, $|X_\gamma\cap y^\uparrow|\ge|\{z_i\mid i<\mu\}|=\mu$.
\end{proof}
Towards a contradiction, suppose that there exist $\lambda<\kappa$ and $\tau<\chi$ such that $\lambda^\tau\ge\kappa$.
Let $\lambda<\kappa$ be the least cardinal for which there exists $\tau<\chi$ satisfying $\lambda^\tau\ge\kappa$.
Then, let $\tau<\chi$ be the least cardinal such that $\lambda^\tau\ge\kappa$. In particular,  $\mu=\lambda^{<\tau}$ is $<\kappa$, and $\eta\le\tau<\chi$.
By Claim \ref{claim861},
let us find an ordinal $\beta\in D$ such that $|X_{\beta}|\ge\tau$, and then pick a sequence $\langle y_\xi\mid \xi<\tau\rangle$ of distinct elements from $X_{\beta}$.
Next, by Claim \ref{claim861}, find a large enough $\gamma<\kappa$ such that $|X_{\gamma}\cap\cone{y_\xi}|\ge\mu$ for all $\chi<\tau$.
For any $\xi<\tau$, pick an injection $\psi_\xi:\lambda^\xi\rightarrow X_{\gamma}\cap\cone{y_\xi}$.

Consider the derived tree $\hat X=\bigotimes_{\xi < \tau}\cone{y_\xi}$.
For every function $h\in{}^\tau\lambda$, define $\vec{x}_h:\tau\rightarrow X_\gamma$ by stipulating
$$\vec{x}_h(\xi)=\psi_\xi(h\restriction\xi).$$

Then $\{ \vec{x}_h\mid h\in{}^\tau\lambda\}$ is a collection of $\lambda^\tau$ many nodes of the $\gamma^{\text{th}}$ level of $\hat X$
with the property for any two distinct $h,h'\in{}^\tau\lambda$, the set $\{ \xi<\tau\mid \vec{x}_h(\xi)= \vec{x}_{h'}(\xi)\}$ has size $<\tau$.
Recalling that $\lambda^\tau\ge\kappa$ and $\tau\ge\eta$, we have obtained a contradiction to the hypothesis that $(X,<_X)$ is $(\chi,\eta)$-free.

(2)
Suppose that $\langle w_\xi \mid \xi<\tau\rangle$ is a sequence of $<\chi$ many distinct elements from $T_\alpha$, for some $\alpha<\kappa$,
and that $\langle \vec t^i\mid i<\kappa\rangle$ is an antichain in the derived tree $\hat T=\bigotimes_{\xi < \tau}  \cone{w_\xi}$.
Let $\delta=\min(E\setminus\alpha)$.
By discarding an initial segment, we may assume that  $\height_{\hat T}(\vec t^i)>\delta$ for all $i<\kappa$.

As in the proof of Claim \ref{claim-antichain-04}, for all $i<\kappa$ and $\xi<\tau$, we may find $x_{i,\xi}\in X\restriction E$ such that $\vec t^i(\xi)\s t_{x_{i,\xi}}$,
and then $\langle \langle x_{i,\xi}\mid \xi<\tau\rangle\mid i<\kappa\rangle$ forms an antichain in the product tree $\bigotimes_{\xi < \tau}  X$.
But this does not yet contradict anything, so we continue.

For all $i<\kappa$ and $\xi<\tau$, let $v_{i,\xi}$ denote the unique element of $X_\delta$ that is $\le_X x_{i,\xi}$.
Note that if $\xi<\zeta<\tau$, then $v_{i,\xi}\neq v_{i,\zeta}$. Indeed, otherwise, we would get $t_{x_{i,\xi}}\restriction\delta=t_{v_{i,\xi}}=t_{v_{i,\zeta}}=t_{x_{i,\zeta}}\restriction\delta$,
contradicting the fact that $w_\xi\s t_{x_{i,\xi}}\restriction\delta$ and $w_{\xi'}\s t_{x_{i,\xi'}}\restriction\delta$.
By $|X_\delta|<\kappa$ and Clause (1), there must exist $\langle v_\xi\mid \xi<\tau\rangle$ and some $I\in[\kappa]^\kappa$
such that $v_{i,\xi}=v_\xi$ for all $\xi<\tau$ and $i\in I$.
So $\langle \langle x_{i,\xi}\mid \xi<\tau\rangle\mid i\in I\rangle$ forms an antichain in the derived tree $\bigotimes_{\xi < \tau}  \cone{v_\xi}$,
contradicting the hypothesis that $(X,<_X)$ is $\chi$-free.
\end{proof}

\begin{lemma}\label{lemma87}
Suppose that $\mathcal F$ is a collection of sets over a cardinal $\theta$.
\begin{enumerate}
\item If $\mathcal F$ is upward-closed, then whenever $(X,<_X)$ admits an $\mathcal F$-ascent path, so does $(T,{\subset})$;
\item If $\mathcal F$ is a filter, then whenever $(X,<_X)$ admits an injective $\mathcal F$-ascent path, so does $(T,{\subset})$.
\end{enumerate}
\end{lemma}

\begin{proof}\hfill
\begin{enumerate}
\item
Suppose $\vec f = \langle f_\alpha \mid \alpha < \kappa \rangle$ is an $\mathcal F$-ascent path through $(X,<_X)$.
We need to construct an $\mathcal F$-ascent path $\vec h=\langle h_\alpha \mid \alpha < \kappa \rangle$ through $(T, \subset)$.
Note that since $\mathcal F$ is an upwards-closed family over $\theta$, for all $\alpha<\kappa$, we have $\dom(f_\alpha)=\bigcup\mathcal F=\theta$.
Thus, for every $\alpha < \kappa$, we construct the function $h_\alpha : \theta\to T_\alpha$ as follows:
First, let $\alpha^+ = \min (E \setminus \alpha)$.
Then, for every $i<\theta$, let $h_\alpha(i) = t_{f_{\alpha^+}(i)} \restriction \alpha$.
Since $f_{\alpha^+}(i) \in X_{\alpha^+}$ where $\alpha^+ \in E$, and $\alpha \leq \alpha^+$,
it is clear that $h_\alpha(i) \in T_\alpha$ for every $i<\theta$.

Let us check that $\vec h$ really is an $\mathcal F$-ascent path through $(T, \subset)$.
Consider $\alpha <\beta < \kappa$.
Let $\alpha^+ = \min (E \setminus \alpha)$ and $\beta^+ = \min (E \setminus \beta)$,
so that $\alpha^+ \leq \beta^+$.
Consider the set
\[
A_{\alpha,\beta} = \left\{i <\theta \mid f_{\alpha^+} (i) \leq_X f_{\beta^+}(i) \right\}.
\]

Since $\vec f$ is an $\mathcal F$-ascent path through $(X,<_X)$, we must have
$A_{\alpha,\beta} \in \mathcal F$. Then:
\begin{itemize}
\item
$h_{\alpha}(i) \le_T t_{f_{\alpha^+}(i)} \le_T t_{f_{\beta^+}(i)}$ for all $i\in A_{\alpha,\beta}$,
recalling Claim~\ref{bin-pres-extend};
\item $h_\beta(i) \le_T t_{f_{\beta^+}(i)}$ for all $i<\theta$.
\end{itemize}

So for all $i\in A_{\alpha,\beta}$, $h_\alpha(i)$ and $h_\beta(i)$ are both $\le_T t_{f_{\beta^+}(i)}$,
and hence they are compatible. By $\alpha\le\beta$, then:
\[
A_{\alpha,\beta}\s \left\{i<\theta \mid h_{\alpha} (i) \leq_T h_{\beta}(i) \right\}.
\]

This shows that if $\mathcal F$ is upward-closed,
then $\vec h$ is an $\mathcal F$-ascent path through $(T, \subset)$.

\item
Suppose $\vec f = \langle f_\alpha \mid \alpha < \kappa \rangle$ is an injective $\mathcal F$-ascent path through $(X,<_X)$.
We will show that $\vec h$ as constructed in part~(1) really is an injective
$\mathcal F$-ascent path through $(T, \subset)$.

Choose some $\alpha < \kappa$ and $B_1 \in \mathcal F$ such that $f_\alpha \restriction B_1$ is injective.
Consider any $\delta \in E \setminus (\alpha+1)$,
and set:
\[
B_2 =\left\{i <\theta \mid f_\alpha(i) <_X f_{\delta}(i) \right\}.
\]
Then $B_2 \in \mathcal F$, and we will show that $h_\delta \restriction (B_1 \cap B_2)$ is injective.

Consider any distinct $i,j\in B_1 \cap B_2$.
By $\alpha<\delta$, it follows that
\[
u_{f_\delta(i)}(\alpha)=\pi(f_\alpha(i))\neq \pi(f_\alpha(j))=u_{f_\delta(j)} (\alpha).
\]
Letting $\beta = \psi^{-1}(\alpha, \pi(f_\alpha(i))$, we have $\beta < \delta$ (since $\delta \in E$),
and since $f_\delta(i)\neq f_\delta(j)$, we have
\[
h_\delta(i)(\beta) = t_{f_\delta(i)}(\beta) = 1 \neq 0 = t_{f_\delta(j)}(\beta) = h_\delta(j)(\beta),
\]
and it follows that $h_\delta(i) \neq h_\delta(j)$.

Consequently, the map $h_\delta \restriction (B_1 \cap B_2)$ is injective.
Thus, if $\mathcal F$ is a filter, then $h_\delta$ is injective on a set from $\mathcal F$,
showing that $\vec h$ is an injective $\mathcal F$-ascent path through $(T, \subset)$.
\qedhere
\end{enumerate}
\end{proof}

\section*{Acknowledgements}
This work was partially supported by German-Israeli Foundation for Scientific Research and Development,
Grant No.~I-2354-304.6/2014. Some of the results of this paper were announced by
the second author at the \emph{Forcing and its applications: Retrospective Workshop}, Toronto, April 2015,
and by the first author at the \emph{22nd Boise Extravaganza in Set Theory} conference, San Francisco, June 2015.
The first author thanks Luke Serafin for motivating the definition of $\p_{14}$, by suggesting the introduction of a simplified notation
for the principle $\p(\ldots)$ in cases where some of the eight parameters are fixed.
Both authors thank the organizers of the corresponding workshops for providing a joyful and stimulating environment.

\bibliographystyle{plain}

\begin{thebibliography}{10}

\bibitem{MR0314621}
J.~Baumgartner, J.~Malitz, and W.~Reinhardt.
\newblock Embedding trees in the rationals.
\newblock {\em Proc. Nat. Acad. Sci. U.S.A.}, 67:1748--1753, 1970.

\bibitem{MR836425}
Shai Ben-David and Saharon Shelah.
\newblock Souslin trees and successors of singular cardinals.
\newblock {\em Ann. Pure Appl. Logic}, 30(3):207--217, 1986.

\bibitem{axioms}
Ari~M. Brodsky and Assaf Rinot.
\newblock A microscopic approach to Souslin-tree constructions. Part {I}.
\newblock Submitted, 2015.

\bibitem{rinot23}
Ari~M. Brodsky and Assaf Rinot.
\newblock A microscopic approach to Souslin-tree constructions. Part {II}.
\newblock In preparation.

\bibitem{MR1059055}
C.~C. Chang and H.~J. Keisler.
\newblock {\em Model theory}, volume~73 of {\em Studies in Logic and the
  Foundations of Mathematics}.
\newblock North-Holland Publishing Co., Amsterdam, third edition, 1990.

\bibitem{MR1376756}
James Cummings.
\newblock Souslin trees which are hard to specialise.
\newblock {\em Proc. Amer. Math. Soc.}, 125(8):2435--2441, 1997.

\bibitem{MR2811288}
James Cummings and Menachem Magidor.
\newblock Martin's maximum and weak square.
\newblock {\em Proc. Amer. Math. Soc.}, 139(9):3339--3348, 2011.


\bibitem{MR645905}
Keith~J. Devlin.
\newblock Morass-like constructions of {$\aleph _{2}$}-trees in {$L$}.
\newblock In {\em Set theory and model theory ({B}onn, 1979)}, volume 872 of
  {\em Lecture Notes in Math.}, pages 1--36. Springer, Berlin-New York, 1981.

\bibitem{MR732661}
Keith~J. Devlin.
\newblock Reduced powers of {$\aleph _{2}$}-trees.
\newblock {\em Fund. Math.}, 118(2):129--134, 1983.

\bibitem{MR750828}
Keith~J. Devlin.
\newblock {\em Constructibility}.
\newblock Perspectives in Mathematical Logic. Springer-Verlag, Berlin, 1984.

\bibitem{MR384542}
Keith~J. Devlin and H{\.a}vard Johnsbr{\.a}ten.
\newblock {\em The {S}ouslin problem}.
\newblock Lecture Notes in Mathematics, Vol. 405. Springer-Verlag, Berlin,
  1974.


\bibitem{MR780933}
D.~H. Fremlin.
\newblock {\em Consequences of {M}artin's axiom}, volume~84 of {\em Cambridge
  Tracts in Mathematics}.
\newblock Cambridge University Press, Cambridge, 1984.

\bibitem{MR0049973}
Djuro Kurepa.
\newblock Sur une propri\'et\'e caract\'eristique du continu lin\'eaire et le
  probl\`eme de {S}uslin.
\newblock {\em Acad. Serbe Sci. Publ. Inst. Math.}, 4:97--108, 1952.

\bibitem{MR603771}
Richard Laver and Saharon Shelah.
\newblock The {$\aleph _{2}$}-{S}ouslin hypothesis.
\newblock {\em Trans. Amer. Math. Soc.}, 264(2):411--417, 1981.

\bibitem{lucke}
Philipp L\"ucke.
\newblock Ascending paths and forcings that specialize higher Aronszajn trees.
\newblock In preparation.

\bibitem{MR2781096}
Assaf Rinot.
\newblock On guessing generalized clubs at the successors of regulars.
\newblock {\em Ann. Pure Appl. Logic}, 162(7):566--577, 2011.

\bibitem{MR3347468}
Assaf Rinot.
\newblock Chromatic numbers of graphs---large gaps.
\newblock {\em Combinatorica}, 35(2):215--233, 2015.

\bibitem{MR3354439}
Assaf Rinot.
\newblock Putting a diamond inside the square.
\newblock {\em Bull. Lond. Math. Soc.}, 47(3):436--442, 2015.

\bibitem{rinotschindler}
Assaf Rinot and Ralf Schindler.
\newblock Square with built-in diamond-plus.
\newblock Submitted, 2015.

\bibitem{MR675955}
Saharon Shelah.
\newblock {\em Proper forcing}, volume 940 of {\em Lecture Notes in
  Mathematics}.
\newblock Springer-Verlag, Berlin-New York, 1982.

\bibitem{MR964870}
Saharon Shelah and Lee Stanley.
\newblock Weakly compact cardinals and nonspecial {A}ronszajn trees.
\newblock {\em Proc. Amer. Math. Soc.}, 104(3):887--897, 1988.

\bibitem{MR908147}
Stevo Todor{\v{c}}evi{\'c}.
\newblock Partitioning pairs of countable ordinals.
\newblock {\em Acta Math.}, 159(3-4):261--294, 1987.

\bibitem{MR2965421}
Stevo Todor{\v{c}}evi{\'c} and V{\'{\i}}ctor Torres~P{\'e}rez.
\newblock Conjectures of {R}ado and {C}hang and special {A}ronszajn trees.
\newblock {\em MLQ Math. Log. Q.}, 58(4-5):342--347, 2012.

\end{thebibliography}

\end{document}